%% file: main.tex
\documentclass{amsart}

\usepackage{amsfonts,amsmath,amsthm,amssymb}
\usepackage{mathtools}
\usepackage{commath}
\usepackage[mathscr]{euscript}
\usepackage{moreenum}
\usepackage{enumitem}
\usepackage{xcolor}
\usepackage{listings}
\lstdefinelanguage{isabelle}{
  morekeywords=[1]{theory,imports,begin,end,lemma,theorem,corollary,
    proposition,definition,abbreviation,fun,function,primrec,datatype,
    type_synonym,locale,context,interpretation,instantiation,instance,
    text,section,subsection},
  morekeywords=[2]{fixes,assumes,shows,obtains,where,and,for,if,proof,qed,by,apply,done,using,unfolding,have,show,hence,thus,then,from,with,also,finally,moreover,ultimately,obtain,note,fix,assume,next,case,sorry,oops},
  sensitive=true,
  morecomment=[n]{(*}{*)},
  morestring=[b]",
  literate=
    {≤}{{$\le$}}1 {≥}{{$\ge$}}1 {≠}{{$\neq$}}1
    {λ}{{$\lambda$}}1 {∀}{{$\forall$}}1 {∃}{{$\exists$}}1
    {⇒}{{$\Rightarrow$}}2 {⟹}{{$\Longrightarrow$}}3
    {→}{{$\to$}}2 {⟶}{{$\longrightarrow$}}3 {⟷}{{$\longleftrightarrow$}}3
    {∧}{{$\wedge$}}1 {∨}{{$\vee$}}1 {¬}{{$\neg$}}1
    {∈}{{$\in$}}1 {∉}{{$\notin$}}1 {⊆}{{$\subseteq$}}1
    {∪}{{$\cup$}}1 {∩}{{$\cap$}}1 {∘}{{$\circ$}}1
    {×}{{$\times$}}1 {≡}{{$\equiv$}}1 {⋀}{{$\bigwedge$}}1
    {⟨}{{$\langle$}}1 {⟩}{{$\rangle$}}1
    {α}{{$\alpha$}}1 {β}{{$\beta$}}1 {σ}{{$\sigma$}}1
    {ℕ}{{$\mathbb{N}$}}1 {ℤ}{{$\mathbb{Z}$}}1
    {₀}{{$_0$}}1 {₁}{{$_1$}}1 {₂}{{$_2$}}1,
}

\definecolor{isakeyword}{RGB}{0,0,160}
\definecolor{isakeyword2}{RGB}{0,160,40}
\definecolor{isacomment}{RGB}{120,120,120}
\usepackage[T1]{fontenc}
\usepackage[style=alphabetic]{biblatex}
\usepackage{xurl}
\usepackage{graphicx}
\usepackage{changes}
\usepackage{needspace}
\usepackage{comment}

\usepackage[hidelinks]{hyperref}
\newcommand{\coloredlink}[2]{{\color{blue}\href{#1}{#2}}}
\newcommand{\afpfact}[2]{\coloredlink{https://isa-afp.org/thys/Diophantine_Universal_Pairs/#1.html\##2|fact}{Formalization}}
\newcommand{\afpthy}[1]{\coloredlink{https://isa-afp.org/thys/Diophantine_Universal_Pairs/#1.html}{Formalization}}

\usepackage[capitalise]{cleveref}

\newtheorem{mainstep}{Theorem}

\newtheorem{theorem}{Theorem}[section]
\newtheorem{THEOREM}{Theorem}

\newtheorem{lemma}[theorem]{Lemma}
\newtheorem*{lemma*}{Lemma}
\newtheorem{corstep}[mainstep]{Corollary}
\newtheorem{corollary}[theorem]{Corollary}
\newtheorem{COROLLARY}[THEOREM]{Corollary}

\newtheorem{definition}[theorem]{Definition}
\newtheorem{FormalizationNote}[theorem]{Formalization Remark}

\crefname{FormalizationNote}{formalization remark}{formalization remarks}
\Crefname{FormalizationNote}{Formalization Remark}{Formalization Remarks}

\newcommand{\proofstep}[2]{\par\bigskip\noindent\textbf{#1}{#2}}
\newcommand{\proofsubstep}[2]{\par\medskip\noindent\emph{#1}}{}

\theoremstyle{definition}

\newcommand{\N}{{\mathbb N}}
\newcommand{\Z}{{\mathbb Z}}
\newcommand{\Q}{{\mathbb Q}}

\newcommand{\B}{{\mathscr B}}

\renewcommand{\L}{{\mathcal L}}
\newcommand{\lp}{\left(}
\newcommand{\rp}{\right)}
\newcommand{\CalA}{{\mathcal A}}

\newcommand{\eps}{\varepsilon}

\newcommand{\thmIdec}{\mathscr}
\newcommand{\X}{\thmIdec X}
\newcommand{\Y}{\thmIdec Y}
\renewcommand{\S}{\thmIdec S}
\newcommand{\T}{\thmIdec T}
\newcommand{\RR}{\thmIdec R}
\newcommand{\K}{\thmIdec K}
\newcommand{\bLowercase}{\mathfrak{b}}

\DeclareMathOperator{\mask}{mask}
\DeclareMathOperator{\code}{code}
\DeclareMathOperator{\coeffs}{coeffs}
\DeclareMathOperator{\values}{value}

\newcommand{\virtuallypositive}{suitable for coding}

\newcommand{\Newpage}{}
\newcommand{\nocontentsline}[3]{}
\newcommand{\tocless}[2]{\bgroup\let\addcontentsline=\nocontentsline#1{#2}\egroup}

\usepackage{xcolor}

\newcommand{\sunref}[1]{}

\numberwithin{equation}{section}

\title[Diophantine equations over $\Z$]{Diophantine equations over~$\mathbb Z$: \\ universal bounds and parallel formalization}
\author[J. Bayer, M. David, M. Ha{\ss}ler, Yu. Matiyasevich, D. Schleicher]{Jonas Bayer \and Marco David \and Malte Ha{\ss}ler \and Yuri Matiyasevich \and Dierk Schleicher}

\begin{document}

\begin{abstract}
This paper explores multiple closely related themes: bounding the complexity of Diophantine equations over the integers and developing mathematical proofs in parallel with formal theorem provers. 

Hilbert's Tenth Problem (H10) asks about the decidability of Diophantine equations and has been answered negatively by Davis, Putnam, Robinson and Matiyasevich. It is natural to ask for which subclasses of Diophantine equations H10 remains undecidable. Such subclasses can be defined in terms of \emph{universal pairs}: bounds on the number of unknowns $\nu$ and degree $\delta$ such that all Diophantine equations can be rewritten in at most this complexity. Our work develops explicit universal pairs $(\nu, \delta)$ for integer unknowns, focused on reducing the number of unknowns to $\nu = 11$. This achieves new bounds that cannot be obtained by naive translations from known results over~$\N$.

In parallel, we have conducted a formal verification of our results using the proof assistant Isabelle. While formal proof verification has traditionally been applied a posteriori to known results, this project integrates formalization into the discovery and development process. In a final section, we describe key insights gained from this approach and its implications for mathematical practice. Our work contributes both to the study of Diophantine equations and to the broader question of how mathematics is conducted in the 21\textsuperscript{st} century.
\end{abstract}

\maketitle

\begin{center}
    \textsc{Formal proof development} \cite{up-afp,itp-paper} \url{https://isa-afp.org/entries/Diophantine_Universal_Pairs.html}
\end{center}
\vspace{1.25\baselineskip}

\setcounter{tocdepth}{1}
\tableofcontents

\newpage
\addtocounter{section}{-1}

\section{Introduction}
\label{Sec:Intro}
\setcounter{mainstep}{3}
\input{intro.tex}

\setcounter{mainstep}{0}
\setcounter{THEOREM}{0}
\newpage 

\section{Encoding the Solutions of a Polynomial}
\label{Sec:Coding}
\input{section1}
\newpage

\section{From a Polynomial Representation to a Binomial Coefficient}
\label{Sec:PolyBinom}
\input{section2}

\section{Lucas Sequences}
\label{Sec:Lucas}
\input{section3}
\newpage

\section{From Binomial Coefficients to Diophantine Relations}
\label{Sec:BinomDioph}
\input{section4}
\newpage

\section{Relation-Combining}
\label{Sec:RelationCombining}
\input{section5} 
\newpage

\section{Universal Pairs}
\label{Sec:UniversalPairs}
\input{section6}
\newpage

\section{Collaboration with a Proof Assistant}
\label{Sec:Formalization}
\input{section7}
\newpage

\appendix

\section{Degree Calculations}
\label{sec:appendix}
\input{degree-calculation}

\newpage
\section*{Acknowledgments}\label{sec:acknowledgements}
We are indebted to the group of students who contributed at various stages of the formalization. In particular, we thank Timothé Ringeard and Xavier Pigé for leading the formal proof of the Bridge Theorem; Anna Danilkin and Annie Yao for their work on \texttt{poly\_extract}, \texttt{poly\_degree} and further developments in multivariate polynomial theory; Mathis Bouverot-Dupuis, Paul Wang, Quentin Vermande, and Theo André for their formalization of coding theory; Loïc Chevalier, Charlotte Dorneich, and Eva Brenner for the formalization of Matiyasevich's polynomial and relation combining; and Zhengkun Ye and Kevin Lee for their contributions to formal degree calculations.

We thank Thomas Serafini and Simon Dubischar for their mathematical contributions to the project.

The above highlights their primary contributions, though many have supported other aspects of this work whenever they were able.
We also thank all other members of the student work group for their involvement in the project. We gratefully acknowledge support from the D\'epartement de math\'ematiques et applications at \'Ecole Normale Sup\'erieure de Paris.

\printbibliography

\end{document}

%% file: intro.tex
This paper has three motivations:
\begin{itemize}
\item[$\triangleright$]
bounding the complexity of Diophantine equations over the integers; 
\item[$\triangleright$]
formal verification of mathematical proofs by proof assistants; and
\item[$\triangleright$]
exploring the \emph{parallel development} of proofs and their formalization.
\end{itemize}

With the last point, we would like to emphasize the close relationship between our mathematical and formal motivations. The parallel development of a paper proof and a formal proof in this project has allowed for many synergies, which has greatly influenced the mathematical content and mathematical exposition in this text.

\subsection{Universal Pairs for Diophantine Equations}
A \emph{Diophantine equation} is a polynomial equation in several variables over $\N$ or over $\Z$, where the coefficients are always in $\Z$. In 1900, David Hilbert formulated his legendary problems that are nowadays known as the \emph{23 Hilbert Problems}~\cite{hilbert-problems}. His 10th problem (H10) asks for an algorithm to tell whether a given Diophantine equation has a solution (that is, a collection of input variables so that the value of the polynomial is zero). In 1970, the research of Davis, Putnam, Robinson and Matiyasevich (DPRM) culminated in a negative answer to this question: there is no such algorithm~\cite{dpr, exp-diophantine}. In other words, Diophantine equations are undecidable. The negative solution to H10 is really a positive result because it shows that Diophantine equations are more powerful than previously thought: every recursively enumerable subset of $\N$ is a Diophantine set. This fact is now known as the DPRM theorem.

It is natural to ask if there are subclasses of Diophantine equations that are already undecidable. The complexity of a Diophantine equation can be measured as a pair $(\nu,\delta)$ of positive integers, where $\nu$ is the number of unknowns and $\delta$ is the degree. The same Diophantine set can be represented by different equations which lead to different pairs $(\nu,\delta)$. In general, one can achieve low degrees by many unknowns, or conversely few unknowns by high degrees.

There are countably many Diophantine equations and countably many different Diophantine sets. A remarkable result is the existence of \emph{universal pairs}: a pair $(\nu,\delta)$ is called universal if every Diophantine subset of $\N$ can be described by a Diophantine equation that satisfies this complexity bound. The existence of such universal pairs has historically been a surprise, and used as an argument for why one might not expect a result as strong as the DPRM theorem---see the following excerpt from Kreisel's review\footnote{\url{https://mathscinet.ams.org/mathscinet/article?mr=133227}} of an article by Davis, Putnam and Robinson~\cite{dpr}:

\begin{quote}
[...] it is not altogether plausible that all (ordinary) Diophantine problems are uniformly reducible to those in a fixed number of variables of fixed degree, which would be the case if all [recursively enumerable] sets were Diophantine.
\end{quote}

Universal pairs are different depending on whether unknowns in $\N$ or in $\Z$ are considered. Matiyasevich and Robinson have shown in 1975 that reducing the number of unknowns to 13 is possible in the case of $\N$~\cite{MR75}, later announcing a reduction to $9$ unknowns~\cite{M77}. First bounds for universal pairs over $\N$ were then given in an article by James P. Jones~\cite{9var}. Gasarch's 2021 survey~\cite{gasarch-pairs-overview} documents known universal pairs for both $\N$ and $\Z$.

We are not aware of any non-trivial universal pairs when the unknowns are integers. Zhi-Wei Sun has shown that all Diophantine sets $\CalA \subseteq \N$ can be described by a Diophantine equation with 11 integer unknowns~\cite{sun-thesis, Sun} (later improved to 10~\cite{sun10}), but without investigating the degree of such an equation (see \cref{sec:sun}).

The topic of this paper are universal pairs over $\Z$; one motivation for this choice is that Hilbert's original question was about Diophantine equations over $\Z$. It is quite elementary to turn an arbitrary Diophantine equation over $\Z$ into an equivalent equation over $\N$ and vice versa (every integer is the difference of two positive integers, and every positive integer is the sum of four squares). The goal of this paper is to develop bounds for universal pairs over $\Z$ in a more direct and efficient way, and specifically to keep the number of unknowns low.

\subsection{Can Formal Proof Verification Guide Mathematical Discovery?} 
The mathematics of bounding the complexity of Diophantine equations requires intricate constructions with many moving parts. The reader will experience this firsthand: writing down the Diophantine equation which yields our universal pair will require defining more than 30 intermediate polynomials. Following these definitions, it is highly error prone to verify every conclusion by hand.
Formalization helps to overcome this challenge. 

In previous work, we have formalized the DPRM theorem in Isabelle~\cite{isabelle-formalization}. Given this experience, we asked if a comparable process could help support our ongoing research on universal pairs. Stepping back, we view our article also as a contribution to the ongoing debate on ``how to do mathematics in the 21\textsuperscript{st} century''~\cite{qed-manifesto, qed-manifesto-2, proof-between-generations}. In particular, the present case study seeks to answer the question:
\begin{quote}
    Can a proof assistant aid the process of discovery in mathematical research, rather than just providing an a posteriori certificate of correctness for mathematical publications?
\end{quote}

A key feature of our text is that the following results have been developed in parallel on pen and paper and in the proof assistant Isabelle. We sometimes also say that our manuscript was formalized \emph{in situ}, i.e.\ in its natural environment---a mathematics department---, and before all details of the mathematics had been fleshed out.\footnote{This contrasts with the hitherto more common approach that one might call \emph{in vitro} formalization: a published and polished result is taken from the mathematical literature to be carefully implemented and verified in a proof assistant.} And indeed, the formalization process did not just verify our result but helped uncover (sometimes quite non-trivial!) gaps and mistakes in existing proofs. Moreover, the formalization helped to simplify them and visualized how modifications of the assumptions affect the correctness of our statements. We highlight these changes throughout the mathematical text, called ``Formalization Remarks.'' See also 
Section~\ref{Sec:Formalization} for details.

\subsection{Statement of Results}

We now provide statements of the principal mathematical results in this paper, together with exact definitions of the terms used in the introduction. 

\begin{definition}[Diophantine set: \afpfact{Diophantine_Definition}{Diophantine_Definition.is_diophantine_over_N_def}]
We say that a subset $\CalA \subseteq \N$ is \emph{Diophantine} over $\N$ (resp.\ over $\Z$) if there exists $\nu \in \N$ and a polynomial $P_{\CalA} \in \Z[X_0, X_1, ..., X_\nu]$ such that all $a\in\N$ satisfy
\begin{equation}
a \in \CalA \iff \exists\; \mathbf{z} \in \N^\nu (\textup{resp. } \mathbf{z}\in \Z^\nu) \colon P_{\CalA}(a,\mathbf{z}) = 0
\;.
\end{equation}
In this case, we say that the polynomial $P_{\CalA}$ \emph{represents} the Diophantine set $\CalA$.
\end{definition}

The arguments of $P_\CalA$ have different characters: we call $a\in\N$ the \emph{parameter} of $P_\CalA$ and $\mathbf{z}$ the \emph{unknowns} of $P_\CalA$; both types together are known as the \emph{variables}. 

\begin{definition}[Universal pair: \afpfact{Universal_Pairs}{Universal_Pairs.universal_pair_N_def}]
A pair $(\nu,\delta)\in\N^2$ is called a \emph{universal pair over $\N$ (resp.\ over $\Z$)} if every Diophantine set $\CalA\subset\N$ is represented by a polynomial $P_\CalA$ with at most $\nu$ unknowns in $\N$ (resp. in $\Z$) and of degree at most $\delta$. 
\end{definition}

We write $(\nu,\delta)_\N$ and  $(\nu,\delta)_\Z$ in order to distinguish universal pairs over $\N$ resp.\ over $\Z$. Our main result is the following: 

\begin{THEOREM}[Universal pairs from $\N$ to $\Z$: \afpfact{Universal_Pairs}{Universal_Pairs.universal_pairs_Z_from_N}]
\label{step:FirstUniversalPair}
Let $(\nu, \delta)_\N$ be universal. Then
\[ \big(11, \eta(\nu, \delta) \big)_\Z \]
is universal where
\[
\eta(\nu, \delta) = 
15 \, 616 + 233\,856 \; \delta + 233\,952 \; \delta \, (2 \delta + 1)^{\nu+1} + 467\,712 \; \delta^2 \, (2 \delta + 1)^{\nu+1}
\, .
\]
\end{THEOREM}

We can use this to calculate the following explicit universal pairs. 

\begin{COROLLARY}[Universal pairs over $\Z$: \afpfact{Universal_Pairs}{Universal_Pairs.universal_pair_1}\footnote{This corollary is formalized by assuming the universal pairs $(58, 4)_\N$ and $(32, 12)_\N$ which are due to Jones~\cite{9var}, i.e., only the computation of $\eta(\nu, \delta)$ has been verified here.}]
\label{THM:UniversalPairZ}
The pairs
\begin{align*}
	& (11, 1\,681\,043\,235\,226\,619\,916\,301\,182\,624\,511\,918\,527\,834\,137\,733\,707\,408\,448\,335\,539\,840) \\
	& \approx (11, 1.68105 \cdot 10^{63})
\end{align*}
and
\begin{align*}
& (11, 950\,817\,549\,694\,171\,759\,711\,025\,515\,571\,236\,610\,412\,597\,656\,252\,821\,888)\\
& \approx (11, 9.50818 \cdot 10^{53})
\end{align*}
are universal over $\Z$.
\end{COROLLARY}

The universal pair $(58,4)$ over $\N$ is known \cite{9var}, and we have $\eta(58,4) \lessapprox 1.68105 \cdot 10^{63}$, which yields the first pair given above.

Jones also mentions the universal pair $(32, 12)_\N$, which yields the second universal pair over $\Z$ stated above, using $\eta(32, 12) \lessapprox 9.50818 \cdot 10^{53}$. Note that while the two universal pairs $(58,4)_\N$ and $(32, 12)_\N$ are not comparable (one has lower degree, while the other has fewer unknowns), they lead to comparable universal pairs over~$\Z$ with 11 unknowns. We state both because the better pair depends on $(32,12)_\N$ for which Jones does not provide a proof. 

It is possible to generalize these results to Diophantine sets of $n$-tuples of parameters $A \subseteq \N^n$. In this case, the degree of the universal pair becomes a function of the number of parameters $n$. We do not focus on multiparameter Diophantine equations in this paper. 

Every universal pair $(\nu,\delta)_\N$ over $\N$ yields trivially a universal pair $(4\nu,2\delta)_\Z$ over $\Z$, replacing each of the $\nu$ unknowns $X_i\in\N$ by $a_i^2+b_i^2+c_i^2+d_i^2$ with $a_i,b_i,c_i,d_i\in\Z$, using Lagrange's Four Squares Theorem. Slightly more efficiently, we can replace $X_i$ by $a_i^2+b_i^2+c_i^2+c_i$ using Lemma~\ref{lemma:threeSquares}, so we obtain a universal pair $(3\nu,2\delta)_\Z$. Furthermore, Sun has shown that a universal pair with $\nu=n$ over $\N$ implies the existence of a universal pair over $\Z$ with $\nu=2n+2$ (without explicitly stating degrees)~\cite{Sun1992relation}.

In view of these simple universal pairs over~$\Z$, our pairs in Corollary~\ref{THM:UniversalPairZ} have many fewer variables but at the expense of gigantic degrees: it is highly non-trivial to reduce the number of variables, and this is the focus of this article.

\begin{FormalizationNote}\label{formalization:degree-calculation}
Both \Cref{step:FirstUniversalPair} and \Cref{THM:UniversalPairZ} involve extremely large numbers whose precise values crucially depend on the long list of polynomials introduced throughout the proof. During the development of this project, these values were either missing or incorrect for the longest time.

In fact, even when we submitted the first publication on this work,\footnote{\url{https://arxiv.org/pdf/2505.16963v1}} the numbers in our theorems were still wrong. At the time, we had only formalized \Cref{thm:mainthm}, but not an explicit expression for $\eta$ nor the degree bounds of \cref{sec:appendix}. Instead, we manually copied these bounds into Mathematica to compute $\eta$ as their composition. This was not sufficient to prevent mistakes. In the end, it was only through the rigor of the formal proof in Isabelle that we achieved strong confidence in the correctness of our main result.
\end{FormalizationNote}

\subsection{Relation to Hilbert's 10th Problem}
Hilbert's 10th Problem asks for an algorithm to determine whether or not a given Diophantine equation over $\Z$ has a solution. As a consequence of the DPRM theorem there cannot be a general algorithm for all Diophantine equations. Specific classes of Diophantine equations may well have a solution algorithm.
For example, the solvability of Diophantine equations with one unknown is always decidable.\footnote{The rational root theorem gives a finite list of possible solutions. In particular, to decide solvability in the integers, it suffices to check each factor of the constant term.}
When the number of unknowns is two, an algorithm is known for a special case~\cite{baker}.
Moreover, Diophantine equations of degree two (with an arbitrary number of unknowns) are also decidable~\cite{deg2decidable,deg2decidable2}.
Our results imply the following.

\begin{COROLLARY}[Hilbert's 10th Problem in 11 unknowns]
    Consider the class of Diophantine equations with at most 11 unknowns and at most degree $1.69 \cdot 10^{63}$. Hilbert's Tenth Problem is unsolvable for this class: there is no algorithm that can determine for all equations in this class whether they have a solution or not. 
\end{COROLLARY}

\subsection{Related Work in Formalized Mathematics}

Related mathematics on Diophantine equations has also been formalized in Lean~\cite{lean-formalization} and in Mizar~\cite{mizar-formalization}. The full negative result to Hilbert's Tenth Problem has been formalized in Rocq's library of undecidable problems~\cite{coq-formalization}. All of these related projects formalize existing literature (\emph{in vitro}) and do not quantify the complexity of their constructions through universal pairs.

We rely on the formalization of multivariate polynomials~\cite{polynomials-afp} from Isabelle's \emph{Archive of Formal Proofs} (AFP). Although this contribution provides many useful definitions and properties, it is far from complete given the needs of this project. Our extension of this library has grown to more than 3000 lines of Isabelle code.

The advent of modern large-language models (LLMs) has seen a large increase in research-level results in early 2026 that have been discovered, formalized or both, either fully or partly autonomously by commercial frontier models. Ilin~\cite[Table~1]{ilin2026} sketches an overview of the recent literature in his report on one such project. 

The busy beaver challenge~\cite{BB5} computed the value of the fifth busy beaver number $\mathrm{BB}(5)$ (the largest number of steps any Turing machine with 5 states may take before halting) via formalization in Rocq~\cite{mxdys}. This project compares to ours in that \emph{i)}~both formalize results about the limits of what is computable and involve large natural numbers, certified correct by an interactive theorem prover (ITP), \emph{ii)}~both were large collaborative efforts, though $\mathrm{BB}(5)$ was an online collaboration, and ours was in-person, and \emph{iii)}~both were completed before potent modern LLMs became widely available. Given the rapid and widespread adoption of artificial intelligence (AI) models in the formalization community, both projects may also be the last research-level (\emph{in-situ}) formalizations to be completed entirely without the help of AI models.

\subsection{Structure of the Proof}

Here we describe the main steps in the development of the results and how they are related to the structure of the paper. The task is to describe an arbitrary Diophantine set in terms of a polynomial with a bounded number of variables. This task consists of three main steps, described by Theorems~\ref{step:polynomial-to-binomial}, \ref{step:Sun4}, and \ref{thm:mainthm}. 

In Section~\ref{Sec:Coding}, we introduce \emph{coding techniques} that encode finite vectors of integers into a single (large) natural number. These finite vectors are, in particular, the arguments of a fixed polynomial in several variables.

Section~\ref{Sec:PolyBinom} deals with the first main step (Theorem~\ref{step:polynomial-to-binomial}): we start with a fixed Diophantine set, described by a polynomial equation over $\N$ with any number of unknowns, and turn this into an equivalent divisibility condition on binomial coefficients. The advantage is that this is a single condition, but it is exponentially Diophantine, not Diophantine.

One of the key results in the field was the discovery that exponentially Diophantine equations can be reduced to ordinary Diophantine equations, and one standard method is to use Lucas sequences. The necessary results on Lucas sequences are developed in Section~\ref{Sec:Lucas}. While most of the main results are classical, we need to be careful and do this in a way that allows us to keep track of the degrees and the numbers of variables (and prepare the ground for formalization).

Section~\ref{Sec:BinomDioph} develops the second main step of the proof (Theorem~\ref{step:Sun4}): the divisibility condition on binomial coefficients, as produced in Theorem~\ref{step:polynomial-to-binomial}, is turned into an explicit collection of Diophantine conditions. The number of variables in these equations no longer depends on the degree of the original polynomial. 

Section~\ref{Sec:RelationCombining} on \emph{relation combining} brings together the results of the previous sections. First, we describe an efficient way, following~\cite{MR75}, to combine different conditions in a single equation. Then, we formulate the last main step (Theorem~\ref{thm:mainthm}).

We continue in Section~\ref{Sec:UniversalPairs} by computing explicit universal pairs over $\Z$. In this section, we construct the explicit function $\eta\colon\N\times\N\to\N$, mentioned earlier, which turns a universal pair $(\nu,\delta)_\N$ into a universal pair $(11,\eta(\nu,\delta))_\Z$. 

Section~\ref{Sec:Formalization} is not part of the proof, but instead reflects how the proof has been shaped by the parallel formalization. Finally, we provide an appendix in which we detail the computations of the degrees of the various polynomials that are used throughout the text.

\subsection{Related Work of Zhi-Wei Sun}\label{sec:sun}

This paper and many of its results are closely related to previous work by Zhi-Wei Sun, and everything rests on many of his results. In particular, our Theorems~\ref{step:polynomial-to-binomial}, \ref{step:Sun4} and \ref{thm:mainthm} are due to him~\cite[Theorems 3.1, 4.1 \& 4.2 and 1.1]{Sun}. Analogues of many of our lemmas can also be found in the work cited above. Mathematically, our main contributions are the determination of universal pairs via computation of degree bounds, that is, Theorem~\ref{step:FirstUniversalPair} and Corollary~\ref{THM:UniversalPairZ}. 

In addition, we have made an effort to formulate the main intermediate results as conceptually as possible. In particular, one would ideally like to have ``if and only if'' statements, for instance, in Theorems~\ref{step:polynomial-to-binomial}, \ref{step:Sun4}, and \ref{thm:mainthm}. Although this is not always feasible, we tried to formulate these results as symmetrically as possible.

We note that, after submission of this text for peer review, Sun has improved his original result to use only ten integer unknowns~\cite{sun10}. 

\subsection{Notation}
We follow the convention that $0$ is a natural number, so the set $\mathbb{N} = \{0, 1, 2, \dots\}$. We write $\square$ for the set of square numbers and $n \uparrow$ for the set of powers of $n$. Moreover, we write $[c, d) := \{ x \in \mathbb{Z} \colon c\le x < d \}$ as well as, analogously, $(c, d) := \{ x \in \mathbb{Z} \colon c < x < d \}$. 

Unless explicitly mentioned otherwise, all quantities are integers or integer-valued polynomials. Polynomials are always denoted by capital letters or script capital letters (the only exception being $\bLowercase$). The variables of these polynomials are typically (but not always) lowercase letters. We write vectors and multi-indices in boldface, such as $\mathbf{z} = (z_1, \ldots, z_\nu)$. 

When we write ``$P$ has a solution'' we always mean $\exists\, \mathbf{z}: P(\mathbf{z}) = 0$. Whereas when we write ``$P$ has a solution for (the given parameter) $a \in \N$'' we always mean $\exists\, \mathbf{z}: P(a, \mathbf{z}) = 0$.

%% file: section1.tex
One of the key steps in our arguments will be to consolidate several variables into a single (large) number. Given a polynomial equation $P(z_0,z_1, \dots, z_\nu) = 0$, we define a ``code'' $c$ for the variables as follows: $c := \sum_{i=0}^\nu z_i \B^{n_i}$, where $\B\in\N$ is an appropriate base number and the $n_i>0$ are strictly increasing integers. The goal of this section is to connect the polynomial $P$ to properties of the code $c$. This connection should be suitable to be described by Diophantine conditions. We will describe this connection in terms of carries when adding modulo $2$. 

\begin{definition}[Counting digits and carries: \afpfact{Bit_Counting}{Bit_Counting.count_bits_def}] 
\label{def:carry}
For a number $a\in\N$, denote by $\sigma(a)$ the sum of its digits in base $2$ (the number of digits $1$). For $a,b\in\N$ let $\tau(a, b)$ denote the number of carries that occur in the addition of $a$ and $b$ in base $2$. 
\end{definition}

\subsection{Elementary Properties of \texorpdfstring{$\tau$ and $\sigma$}{Digit Functions}}
The following properties will be useful. Most of them are formalized in the file \coloredlink{https://isa-afp.org/thys/Diophantine_Universal_Pairs/Bit_Counting.html}{Bit\_Counting}.

\begin{lemma}[Basic properties]
\label{Lem:TauBasic}
The digit count function $\sigma$ and the carry count function $\tau$ have the following properties.
\begin{enumerate}
\item \label{item:SigmaPowerTwo} $\sigma(2^k-1) = k$; 
\item \label{item:closedtausigma} if $a < 2^k$, then $\sigma(a) + \sigma(2^k-1-a) = k$;   
\item \label{item:tausigma} $\tau(a,b) = \sigma(a)+\sigma(b)-\sigma(a+b)$; 
\item \label{item:shifttau} if $N \in 2\uparrow$ and $a < N$, then $\sigma(a+bN) = \sigma(a) + \sigma(b)$ for all $b\in\N$; in particular, $\sigma(bN) = \sigma(b)$; 
\item \label{item:tauaddcombine} if $a, b < 2^k$ with $\tau(a, b) = 0$ and $\tau(c, d) = 0$, then $\tau(a+c \cdot 2^k, b+d \cdot 2^k) = 0$;  
\item \label{item:tauaddsplit} if $a, b < 2^k$ and $\tau(a + c \cdot 2^k, b + d \cdot 2^k) = 0$ then $\tau(a, b) =0$ and $\tau(c, d) = 0$;
\item \label{item:taudiv} $N \in 2\uparrow$, $\tau(N-1, a) = 0 \iff N | a$.
\end{enumerate}
\end{lemma}
All these properties are elementary and not hard to prove.

\begin{lemma}[Highest power of two in binomial coefficient: \afpfact{Bit_Counting}{Bit_Counting.count_bits_divibility_binomial}] 
The highest power of two that divides $\binom{2X}{X}$ has exponent exactly $\sigma(X)$.

Concisely and equivalently, write $\sigma(X)\ge n \iff 2^n | \binom{2X}{X}$.
\label{Lem:BinomDivisibility}
\end{lemma}
\begin{proof}
According to Legendre's formula, the highest power of $2$ that appears in $X!$ has exponent $\lfloor X/2\rfloor+\lfloor X/4\rfloor+\lfloor X/8\rfloor+\dots  $, so the highest power of $2$ in $\binom{2X}{X}$ has exponent
\begin{equation}
\lfloor 2X/2\rfloor+\lfloor 2X/4\rfloor+\lfloor 2X/8\rfloor +\dots - 2\lfloor X/2\rfloor-2\lfloor X/4\rfloor - 2\lfloor X/8\rfloor-\dots
\label{Eq:LegendreFormula}
\end{equation}

The key is that, for every $m\ge 1$, the difference $e_m:=\lfloor 2X/2^m\rfloor-2\lfloor X/2^m\rfloor$ equals the binary digit at position $2^{m-1}$ in $X$: this is obvious for $m=1$, and for $m>1$ the observation follows by replacing $X$ with $X/2^{m-1}$. As such, for every $m\ge 1$, we have $e_m \in \{0,1\}$.

The result follows by adding over all binary digits of $X$, like in the sum \eqref{Eq:LegendreFormula}.
\end{proof}

\subsection{Expressing the Carry Counting Function through a Binomial Coefficient} 
The following lemma shows that the number of carries being zero is equivalent to the divisibility of a binomial coefficient. Note that it has been known since the first proof of DPRM that binomial coefficients are Diophantine; together with the following lemma this means that the condition of $\tau$ being zero can be represented in a Diophantine way. 

\begin{lemma}[\afpfact{Tau_Reduction}{Tau_Reduction.Tau_Reduction.tau_as_binomial_coefficient}]
\label{lemma:tau-binom} \sunref{\cite[Lemma 2.2]{Sun}}
Let $N \in 2 \uparrow$ and $S, T \in \mathbb{N}$ and suppose that $0 \le S, T < N$. Then for $R := (S + T + 1) N + T + 1$ we have
\[ \tau(S, T) = 0 \iff N^2 \,|\, \binom{2 (N-1)R}{ (N-1) R} 
\;.\]
\end{lemma}

\begin{proof}
Write $N=2^k$. We compute
\begin{align*}
\tau(S,T)&\stackrel{(*1)}{=}\sigma(S)+\sigma(T)-\sigma(S+T)
\\
&\stackrel{(*2)}{=}\sigma(N-1)-\sigma(N-1-S)+\sigma(N-1)-\sigma(N-1-T)-\sigma(S+T)
\\
&\stackrel{(*3)}{=} 2\sigma(N-1)-\sigma(N-1-T+N(N-1-S)+N^2(S+T))
\\
&\stackrel{(*4)}{=} 2k-\sigma((N-1)R)
\;.
\end{align*}
Here we used various properties from Lemma~\ref{Lem:TauBasic}: in step (*1), this was property \eqref{item:tausigma}; in step (*2), properties \eqref{item:closedtausigma} and \eqref{item:SigmaPowerTwo}; in step (*3), we used property \eqref{item:shifttau} twice (shifting the three numbers $N-1-T$, $N-1-S$ and $S+T$ to different positions); and finally, in step (*4), we used property \eqref{item:SigmaPowerTwo} and an elementary calculation. 

We therefore have $\tau(S,T)=0$ if and only if $\sigma((N-1)R)=2k$. If this is not the case, then $\tau(S,T)>0$ and hence $\sigma((N-1)R)<2k$. In other words, $\tau(S,T)=0$ if and only if $\sigma((N-1)R)\ge 2k$. 

The claim now follows directly from Lemma~\ref{Lem:BinomDivisibility}.
\end{proof}

\subsection{Masking} Having established the above elementary properties we can now continue to give a precise definition for the concept of code mentioned at the beginning of this section. The following definition is mainly found in the file \coloredlink{https://isa-afp.org/thys/Diophantine_Universal_Pairs/Masking.html}{Masking}.

\begin{definition}[Code and mask]
We say that an integer $c$ is a \emph{code in base $\B$ for the tuple $\mathbf{z} = (z_1, ..., z_\nu)$} with respect to a tuple $\mathbf{n} = (n_1,\dots,n_\nu)$ of strictly increasing numbers if
\begin{equation}\label{eq:codeDef}
 c= \code(\mathbf{z},\B,\mathbf{n}) := \sum_{i=1}^\nu z_i \B^{n_i}
\;.
\end{equation}

In this situation, for integer $b < \B$, we define
\[ M = \mask(b, \B , \mathbf{n}) := \sum_{i=0}^{n_\nu} m_i \B^i \]
where $ m_i := \begin{cases}
\B-b & \text{if } i = n_j \text{ for some } j \in\{ 1, \ldots, \nu\}; \\
\B-1 & \text{otherwise}.
\end{cases} $
 \label{def:mask}
\end{definition}

The concept of mask will make it possible to tell whether or not a given number $c$ is a code. It is intuitive (and useful) when $b$ and $\B$ are powers of $2$. If $\B=2^K$ and $b=2^k$, then each $m_i$ consists of $K$ digits in base $2$: in the second case, $m_i$ is a string of $K$ digits $1$, while in the first case, $m_i$ is first a string of $K-k$ digits $1$, then $k$ digits $0$, so that $m_i$ is divisible by $b$. 

A key property is that $\tau(g,M)=0$ if and only if, in base $2$, at all positions where $M$ has a $1$, necessarily $g$ has a $0$. In this sense, the number $M$ serves as a ``mask'' for $g$, forcing the non-zero binary digits of $g$ to positions where $M$ has a zero. The following lemma makes the relationship between code and mask precise. 

\Newpage

\begin{lemma}[Masking: \afpfact{Masking}{Masking.masking_lemma.masking_lemma}] \label{lemma:code-mask} \sunref{\cite[Lemma 2.3]{Sun}} 
Let $b, \B \in \N$ be powers of $2$ with $b \le \B$ and $\B \geq 2$. Moreover, let a tuple $\mathbf{n} = (n_1,\dots,n_\nu)$ of strictly increasing numbers be given. Define $M = \mask(b, \B, \mathbf{n})$.

Then for any integer $g\ge 0$, 
\begin{equation}
g<\B^{n_\nu+1} \text{ and } \tau(g, M) = 0
\label{Eq:Code2}
\end{equation}
if and only if 
\begin{equation}
\exists z_1	, \ldots, z_\nu \in [0, b) \colon  \text{$g$ is a code in base }\B \text{ for } (z_1, ..., z_\nu) \text{ with respect to $\mathbf{n}$}.
\label{Eq:Code1}
\end{equation}
\end{lemma}

\begin{proof}
The key condition is $\tau(g,M)=0$. If $\B=2^K$, the mask $M$, written in base 2, consists of blocks of length $K$. If a given $i$ is different from all $n_j$, then $m_i=2^K-1$, so all digits in this block are equal to $1$, so $\tau(g,M)=0$ requires that all digits of $g$ in this block are zero. 

In the other case, we have $m_i=\B-b$, so we need that the digits of $g$ within this block form a binary number less than $b$. 

The condition $g<\B^{n_\nu+1}$ assures that $g$ does not have binary digits at large positions that are not checked by the mask $M$. 

Therefore, for a given mask $M$ and tuple $\mathbf{n}$, we have $\tau(g,M)=0$ if and only if we can write $g=\sum_{j=1}^{n_\nu} z_j\B^j$ where all $z_j\in[0,b)$, and $z_j=0$ except when $j=n_i$ for some $i$. 
\end{proof}

\subsection{Encoding Polynomial Solutions}
We can now say when a given number is a code $c$. It remains to relate a code $c$ to the solutions of a given polynomial $P$. To this end, we will re-code (the coefficients of) the multivariate polynomial $P$ into a univariate polynomial.

\begin{definition}[Coefficient-encoding and Value Polynomials: \afpfact{Lemma_1_8_Defs}{Lemma_1_8_Defs.Lemma_1_8_Defs.P_coeff_def}]
\label{Def:ConsolidatingPoly}
Consider a multivariate polynomial
\begin{align}
    P(X_0, \dots, X_\nu) = \sum_{\substack{\mathbf{i} \in \N^{\nu+1} \\ \|\mathbf{i}\| \le \delta}} a_{\mathbf{i}} \, \mathbf{X}^{\mathbf{i}}
\label{Eq:MultiPolynomial}
\end{align}
in $\nu+1$ variables with integer coefficients $a_{\mathbf{i}}$, and of degree $\delta$. Here $\mathbf{X}=(X_0, ..., X_\nu)$ is the tuple of variables and $\mathbf{i}=(i_0,\dots,i_\nu)$ is a multi-index.

Then we define the \emph{coefficient-encoding polynomial of $P$} as follows:
\begin{align}                                
\coeffs(Y) := \sum_{\substack{\mathbf{i} \in \mathbb{N}^{\nu+1} \\ \|\mathbf{i}\| \leq \delta}} \mathbf{i} ! (\delta - \|\mathbf{i}\|)! a_{\mathbf{i}} Y^{(\delta+1)^{\nu+1} - \sum_{s=0}^\nu i_s (\delta+1)^s }
\;.
\label{Def:coeffs}
\end{align}
Based on this, we define the \emph{value polynomial}
\begin{equation}
\values(x,\B) := x^\delta \coeffs(\B) +  \sum_{j=0}^{( 2\delta+1)(\delta+1)^\nu} \frac{\B}{2} \B^{j}
\;.
\label{eq:value}
\end{equation}
\end{definition}

We use multi-indices $\mathbf{i}, \mathbf{j} \in \N^{\nu + 1}$ as follows: $\|\mathbf{i}\|=\sum_s i_s$, $\mathbf{i}!=\prod_s i_s!$, and $\mathbf{X}^{\mathbf{i}}=\prod_s  X_s^{i_s}$. Then the multinomial coefficient is $\binom{\delta}{\mathbf{i}} := \frac{\delta !}{\mathbf{i}!(\delta-\|\mathbf{i}\|)!}$ and the scalar product is $\mathbf{i}\cdot\mathbf{j}=\sum_s i_s j_s$.
We will use $\values(x,\B)$ only when $\B$ is even, so all coefficients are integers.

The role of the coefficient-encoding polynomial is to place the coefficients of $P$ in specific positions of the base $\B$ expansion of $\coeffs(\B)$. This ensures that when $\coeffs(\B)$ is multiplied with $(1+c)^\delta$ where $c = \code(\mathbf{z}, \B, \mathbf{n})$, one obtains a number whose leading digit in base $\B$ corresponds to the evaluation $P(z_0, \ldots, z_\nu)$.    
This relationship is made precise in the next lemma. 

\begin{lemma}[\afpfact{Lemma_1_8}{Lemma_1_8.Lemma_1_8.lemma_1_8}] \label{lemma:polynomial-zero-tau-zero} \sunref{\cite[Lemma 2.4]{Sun}}
Let $P$ be a polynomial as in \eqref{Eq:MultiPolynomial} with degree $\delta$, let $\mathbf{z} = (z_0, \ldots, z_\nu)$ be a tuple of natural numbers, and $\mathbf{n}=(n_0,\dots,n_\nu)$ with $n_j = (\delta + 1)^j$ for all $j$.

Let $\mathcal L$ be an integer such that $\mathcal L \ge |a_{\mathbf{i}}|$ for all $\mathbf{i} \in \N^{\nu+1}$, and let $\B \in 2 \uparrow$ such that
$ \B > 2\delta! \, \mathcal L (1 + z_0+\dots+z_\nu)^\delta$.  
Define 
\begin{align}
\K & := \values(1+\code(\mathbf{z},\B,\mathbf{n}),\B)
\;. \label{eq:code-K-def}
\end{align}

Then we have
\begin{equation}
P(\mathbf{z}) = 0 \iff \tau\left(\K, (\B/2-1)\B^{n_{\nu+1}}\right) = 0.
\end{equation}
\end{lemma}

\begin{proof}
The key to the proof is again the carry counting property $\tau = 0$. The number 
$(\B/2-1)\B^{n_{\nu+1}}$ works as a mask: in base $\B$, it has (from right to left) $n_{\nu+1}$ digits $0$ and then a single digit $\B/2-1$. We thus have $\tau=0$ if and only if $\K$, in base $\B$, has at position $n_{\nu+1}$ either a digit $0$ or $\B/2$. 

By definition \eqref{eq:code-K-def}, the first summand in $\K$ is a product of two factors that we write $K'\cdot K''$ for short. Expressing these in base $\B$, we have $K'=\sum k'_j\B^j$ and similarly for $K''$; then the $\B^{n_{\nu+1}}$ term in $\K$ has the coefficient
\[
\sum_j k'_j k''_{n_{\nu+1}-j}
\;.
\]
We have to check only this coefficient. 

The second summand in the definition of $\K$ takes care of the fact that the coefficients $a_{\mathbf{i}}$ can be negative. So far, the base $\B$ expansion given by expanding equations~\eqref{Def:coeffs} and~\eqref{eq:value} may have positive or negative coefficients, but less than $\B/2$ in absolute values. Adding terms $(\B/2)\B^i$ makes sure that the base $\B$ expansion of $\K$ in its final form only has coefficients in $[0,\B)$ as required. 

In fact, the lower bound on $\B$ assures that no coefficient is zero, so what we need to check is whether it equals $\B/2$; or equivalently whether the $\B^{n_{\nu+1}}$ term in the first summand of $\K$ equals zero.

The first factor in the definition of $\K$ involves powers of $\mathbf{z}$, while the second factor involves the coefficients $a_{\mathbf{i}}$. The value of the coefficient in question is a single sum, and it turns out (and we have to verify) that it equals exactly $P(\mathbf{z})$, up to a non-zero factor.

Using the binomial theorem, we have
\[(1 + \code(\mathbf{z}, \B, \mathbf{n}))^\delta = \sum_{\|\mathbf{i}\| \leq \delta} \binom{\delta}{\mathbf{i}} \,\mathbf{z}^{\mathbf{i}} \, \B^{\mathbf{i}\cdot \mathbf{n}}
\;.
\]

Therefore
\begin{equation}
(1 + \code(\mathbf{z},\B , \mathbf{n}))^\delta \coeffs(\B) 
= \sum_{\substack{\|\mathbf{i}\| \leq \delta \\ \|\mathbf{j}\| \leq \delta }} \binom{\delta}{\mathbf{i}} \, \mathbf{j} ! \, (\delta - \|\mathbf{j}\|)! \, a_{\mathbf{j}}\, \mathbf{z}^{\mathbf{i}} \B^{n_{\nu+1}+\mathbf{i} \cdot \mathbf{n}-\mathbf{j} \cdot \mathbf{n}}
\;.
\label{Eq:Multinomial}
\end{equation}

Let us notice that the monomial $\B^{n_{\nu+1}}$ appears exactly when $\mathbf{i}\cdot \mathbf{n} - \mathbf{j} \cdot \mathbf{n}=(\mathbf{i}-\mathbf{j})\cdot \mathbf{n}=0$, hence $\mathbf{i} = \mathbf{j}$. 
In this case,  we have $\binom{\delta}{\mathbf{i}}\,\mathbf{j}! \, (\delta-\|\mathbf{j}\|)!=\delta!$, so the coefficient of $\B^{n_{\nu+1}}$ equals 
\[
\sum_{\|\mathbf{i}\| \leq \delta} \delta!\,a_{\mathbf{i}}\, \mathbf{z}^{\mathbf{i}} = \delta! \,P(z_0,...,z_\nu)
\;.
\]

This completes the conceptual proof of the lemma, but we still have to check that all coefficients of $\K$ are in $(0,\B)$. 

Any such coefficient of $\K$ differs from the analogous coefficient in  \eqref{Eq:Multinomial} by the shift of $\B/2$. Consider one of those coefficients, corresponding to the exponent $n_{\nu+1} + \mathbf{k} \cdot \mathbf{n}$. The double sum reduces to those $\mathbf{i}$ and $\mathbf{j}$ with $\mathbf{i} - \mathbf{j}=\mathbf{k}$. Since $|a_{\mathbf{j}}|\le \L$, the absolute value of the coefficient is at most
\[
\delta! \L  \,\sum_{\|\mathbf{i}\|\le\delta} \frac{(\mathbf{k}-\mathbf{i})! (\delta-\|\mathbf{k}-\mathbf{i}\|)!}{\mathbf{i}!(\delta-\|\mathbf{i}\|)!} \mathbf{z}^{\mathbf{i}}
\le 
\delta! \L  \,\sum_{\|\mathbf{i}\|\le\delta} \frac{\delta!}{\mathbf{i}!(\delta-\|\mathbf{i}\|)!} \mathbf{z}^{\mathbf{i}}
=\delta!\L (1+\|\mathbf{z}\|)^\delta <\B/2
\;,
\]
using the multinomial expansion of $(1+\|\mathbf{z}\|)^\delta$.
\end{proof}

We just obtained a way to check whether the tuple $\mathbf{z}$ represented by a code $c(\mathbf{z}, \B)$ is a solution to the equation $P(\mathbf{z}) = 0$ with a $\tau$-condition on $c$. In the next section, we show how to merge this condition and the masking condition together. 

Let us also notice that the variable $z_0$ will play a special role in things to come. Indeed, while $z_1$ through $z_\nu$ will be the unknowns, encoded by a non-negative integer $g$, $z_0$ will be the parameter $a$, which will (or not) be in the set $\mathcal A$.

As a final remark, we note that the coding techniques developed in this and the next section are amenable to generalization to a general prime base $p$, not just base $2$. This has previously been elaborated by Sun~\cite{Sun}, but is not relevant in the scope of this paper. 

%% file: section2.tex
The goal of this section is to combine Lemmas~\ref{lemma:tau-binom}, \ref{lemma:code-mask} and \ref{lemma:polynomial-zero-tau-zero} from the previous section into the technical Theorem~\ref{step:polynomial-to-binomial} which can later be used to apply the coding ideas all at once. Specifically, the two separate $\tau$ conditions in Lemma~\ref{lemma:code-mask} and Lemma~\ref{lemma:polynomial-zero-tau-zero} can be combined into a single condition by using that
\[
\tau(a, b) =0 \text{ and } \tau(c, d) = 0 \quad \text{ if and only if } \quad
\tau(a + c2^k, b + d2^k) = 0
\]
which holds whenever $a, b < 2^k$ (this is Lemma~\ref{Lem:TauBasic}, item~\eqref{item:tauaddsplit}). One can then continue to apply Lemma~\ref{lemma:tau-binom} to further transform the joint $\tau$ condition into a divisibility of a binomial coefficient.

Note that this process frequently requires checking the various inequalities that appear in the statements of the aforementioned lemmas. Therefore, the main technical challenge of this section lies in defining all involved quantities in a way that allows us to deduce appropriate bounds when needed.

This section starts with a comprehensive definition of all quantities involved in the proof of its main theorem. We then continue to establish lower and upper bounds on these quantities. Finally, everything will be set up to state and prove Theorem~\ref{step:polynomial-to-binomial}.

In the following definition, note that the parameter $a$, the first variable of the polynomial $P$, is not related to the coefficients of $P$ denoted by $a_{i_0, \ldots, i_\nu}$.

\begin{definition}[Various polynomials: \afpfact{Coding_Theorem_Definitions}{Coding_Theorem_Definitions.coding_variables.b_def}] \label{def:codingVariables}
Consider a polynomial $P=P(a, z_1, \ldots, z_\nu)$ and denote its coefficients by  $a_{i_0, \ldots, i_\nu}$ and its degree by $\delta$. Set

\[ \mathcal{L} := \sum_{\substack{\mathbf{i} \in \mathbb{N}^{\nu+1} \\ |\mathbf{i}| \le \delta}} |a_{\mathbf{i}}| \;.
\]
Let $r \in \N$ be minimal such that
\begin{equation}
\beta:=4^r >2\delta!\L (\nu + 2)^\delta \;.
\label{Eq:DefBeta}
\end{equation}
Then define the constants
\begin{align*}
\alpha &:=\delta (\delta +1)^\nu +1  \quad \text{ and } \\
\gamma &:=\beta^{(\delta +1)^\nu}
\end{align*}
and set $n_j = (\delta + 1)^j$. We define the following polynomials in the variables $a, f, g$ (we omit arguments on the right hand side for better readability).
\begin{subequations}
\begin{align}
\bLowercase(a, f) &:= 1 + 3(2a + 1)f  \label{Def:b} \\
\B (a, f) &:= \beta \, \bLowercase^\delta \\
M(a, f) &:= \mask(\bLowercase, \B, \mathbf{n}) \\
N_0(a, f) &:= \B^{(\delta+1)^\nu + 1} \\
N_1(a, f) &:= 4\B^{(2 \delta + 1)(\delta + 1)^\nu + 1} \\
N(a, f) &:= N_0 N_1 \\
c(a, f, g) &:= 1 + a \B + g \\
\K(a, f, g) &:= \values(c,\B)
\\
\S(a, f, g) &:= g + 2\K N_0 \\
\T(a, f) &:= M + (\B - 2)\B^{(\delta+1)^{\nu + 1}} N_0 \\
\RR(a, f, g) &:= (\S + \T + 1)N + \T + 1 \\
\X(a, f, g) &:= (N-1)\RR \\
\Y(a, f) &:= N^2
\end{align}
\end{subequations}
\end{definition}

\begin{lemma}[\afpfact{Lemma_2_2}{Lemma_2_2.lemma_2_2_useful}] \label{lemma:arbitraryF}
Let $a \in \N$. For all $Z > 0$ there is $f \ge Z$ such that
\begin{equation*}
 \bLowercase(a, f) \in \square \quad \text{and} \quad \bLowercase(a, f) \in 2 \uparrow.
\end{equation*}

\end{lemma}

\begin{proof}
Since $4$ and $3(2a+1)$ are coprime, Euler's theorem implies
\[
4^{\phi(3(2a+1))} \equiv 1 \pmod{3(2a+1)}
\]
and hence for all $m >0$
\[
4^{m\phi(3(2a+1))} \equiv 1 \pmod{3(2a+1)}
\;.
\]
The numbers on the left are all powers of $4$, hence squares and powers of $2$. They have the form $1+3(2a+1)f$ for $f\in\N$. Since they become arbitrarily large, the claim follows.
\end{proof}


\subsection{Lower Bounds}

\begin{definition}[Suitable for coding: \afpfact{Suitable_For_Coding}{Suitable_For_Coding.suitable_for_coding_def}]
We call a polynomial $P(a, z_1, \ldots, z_\nu)$ with variables in $\N$ \emph{\virtuallypositive} if it satisfies the following conditions:

\begin{enumerate}
  \item \label{item:exists-positive-zi} for every solution $(a, z_1, \ldots, z_\nu)$ of $P$ there exists a solution $(a, z_1', \ldots, z_\nu')$ such that at least one of the $z_i'$ is positive;
  \item \label{item:a0-positive} the constant coefficient is positive;
  \item the degree of $P$ is positive.
\end{enumerate}
\end{definition}

\begin{lemma}[\afpfact{Suitable_For_Coding}{Suitable_For_Coding.suitable_for_coding_diophantine_equivalent}] \label{lemma:p-suitable-for-coding}
Let the polynomial $P \in \Z[a, z_1, \ldots, z_\nu]$ be given. There exists a polynomial $\overline{P} \in \Z[a, z_1, \ldots, z_{\nu+1}]$ that is suitable for coding and such that $\overline{P}$ has a solution for the parameter $a$ if and only if $P$ has a solution for $a$. Moreover, $\overline{P}$ can be given as
\[ \overline{P}(a, z_1, \ldots, z_{\nu+1}) = P(a, z_1, \ldots, z_\nu)^2 + (z_{\nu+1} - 1)^2 \; . \]
\end{lemma}
\begin{proof}
We start by showing that $\overline{P}$ is suitable for coding. First, note that if $\overline{P}(a, z_1, \ldots, z_{\nu+1}) = 0$ then in particular $z_{\nu + 1} = 1$; fulfilling condition \eqref{item:exists-positive-zi}. One easily calculates the constant coefficient and degree of $\overline{P}$ to verify the remaining two conditions.

It is left to show that $P$ and $\overline{P}$ have solutions for exactly the same parameters. Fix $a \in \N$. We can show that
\[ \begin{split}
& \exists z_1, \ldots, z_\nu \in \N: P(a, z_1, \ldots, z_\nu) = 0 \\
\iff & \exists z_1, \ldots, z_\nu \in \N: P(a, z_1, \ldots, z_\nu)^2 + (1-1)^2 = 0 \\
\iff & \exists z_1, \ldots, z_\nu, z_{\nu+1} \in \N: P(a, z_1, \ldots, z_\nu)^2 + (z_{\nu+1}-1)^2 = 0 \; . \\
 \end{split} \]
Notice that the expression in the last line is exactly $\overline{P}$.
\end{proof}

\begin{lemma}[Lower bounds on the polynomials: \afpfact{Lower_Bounds}{Lower_Bounds.coding_variables.defs_non_negative}]
\label{lemma:coding-lower-bounds}
Suppose $P$ is \virtuallypositive{} and $a\ge 0$, $f>0$, and $g\ge 0$. Then the polynomials in Definition~\ref{def:codingVariables} satisfy the following bounds:
    \begin{align*}
    \bLowercase(a, f) &> 2 \\
    \B (a, f) &\geq 2 \\
    N_0 (a, f) &\geq 2 \\
    N_1 (a, f) &\geq 2^3 \\
    N (a, f) &\geq 2^4 \\
    \S (a, f, g) &\geq 0 \\
    \T (a, f) &\geq 0 \\
    \RR (a, f, g) &> 0
\;.
    \end{align*}
\end{lemma}

\begin{proof}
For simplicity, we omit writing the arguments $a, f, g$ in all polynomials.
Since $a, f$ and $g$ are all nonnegative, most of the claimed properties are immediate. The proof of this lemma proceeds in the order in which quantities have been defined in Definition~\ref{def:codingVariables}.

First note that $\beta=4^r$ is positive. One verifies that $\bLowercase = 1 + 3(2a + 1)f > 2$ and thus $\B = \beta \bLowercase^\delta = 4^r \bLowercase^\delta > 2$. This means that
\[M = \mask(\bLowercase, \B, \mathbf{n}) = \sum_{j=0}^{n_\nu} m_j \B^j \]
is nonnegative because $\B \geq \bLowercase > 2$ and $m_j \geq \B - \bLowercase$. The polynomials $N_0, N_1, N$ and $c$ are defined as products and powers of the previous quantities and thus must also be positive.

The remaining quantities to check are $\RR$, $\S$, and $\T$, and their definition involves $\K$. The definition of $\K = \values (c, \B)$, in turn, includes the term $\coeffs(\B)$, so we start by showing that $\coeffs(\B) > 0$. It turns out that in the definition of $\coeffs(\B)$ in \eqref{Def:coeffs}, the $\mathbf{i}=0$ term dominates, so we can estimate
\begin{align*}
\left|\coeffs(\B) - \delta !\, a_{0,\ldots,0} \B ^ {(\delta+1)^{\nu+1}} \right|
&\leq
\sum_{0 < \|\,\mathbf{i}\,\| \leq \delta} \delta ! \, |a_{\mathbf{i}}|\,
  \B^{(\delta+1)^{\nu+1} - \sum_{s=0}^\nu i_s (\delta+1)^s}
\\
&\stackrel{(*1)}{\leq} \delta ! \, \mathcal L \sum_{r=0}^{(\delta+1)^{\nu+1}-1} \B^r
\\
&\stackrel{(*2)}{\leq} (\B - 1) \frac{\B^{(\delta+1)^{\nu + 1}} - 1}{\B - 1}
< \B^{(\delta+1)^{\nu+1}}
\;.
\end{align*}
In (*1), we re-organize the summation over $\mathbf{i}$ as follows: we write $r:=(\delta+1)^{\nu+1}-\sum_{s=0}^\nu i_s(\delta+1)^s$, interpreting the tuple $\mathbf{i}$ as digits of an integer in base $\delta+1$. Moreover, we use  $\L = \sum_{\mathbf{i}} |a_{\mathbf{i}}|\ge a_{\mathbf{i}}$. In (*2), we use $\delta!\,\L<\beta<\B$.

Using the hypothesis that $P$ is \virtuallypositive, we finally have
\begin{align*}
  \coeffs(\B) > (\delta!\, a_{0,\ldots,0} - 1) \B^{(\delta+1)^{\nu+1}} \geq 0
\;.
\end{align*}
This allows us to deduce that $\K > 0$. The remaining bounds on $\S, \T, \RR$ follow immediately from their definition.
\end{proof}

\begin{FormalizationNote}\label{formalization:K-gt-0}
The initial draft of the proof of \cref{lemma:coding-lower-bounds} claimed the lower bound $\K > 0$ without detailed justification, perhaps because its definition suggests that $\K$ only has positive digits. We noted the difficulty of rigorously proving the bound when its formal proof required significant effort: $\approx$ 500 lines of Isabelle code, or half a page in this paper.
\end{FormalizationNote}

Now we discuss further lower bounds under the same hypotheses as in the previous lemma.

\begin{lemma}[\afpfact{Lower_Bounds}{Lower_Bounds.coding_variables.lower_bounds}]
\label{lemma:xybounds}
Suppose $P$ is \virtuallypositive{} and $a\ge 0$, $f>0$, and $g\ge 0$.
Then the polynomials $\bLowercase, \X$ and $\Y$ from Definition~\ref{def:codingVariables} satisfy:
	\begin{align}
	\X(a, f, g) &\ge 3 \bLowercase(a, f) \\
	\Y(a, f) & \ge \max \{\bLowercase(a, f), 2^8 \} \;.
	\end{align}
\end{lemma}

\begin{FormalizationNote}\label{formalization:afg-assumption}
The precise assumptions on $a, f, g$, and $P$ in \cref{lemma:xybounds} were initially missing in our manuscript and only became apparent when formalization revealed that the lemma was unprovable as originally stated.
\end{FormalizationNote}

\begin{proof}
We will use the estimates developed in the previous lemma. The first condition follows by unfolding the definitions of $\X, \RR, N, N_1, \B$:
\begin{align*}
\X &= (N-1)\RR \geq \RR = (\S + \T + 1)N + \T + 1 \geq N = N_0 N_1 \geq N_1 \\
&= 4 \B^{(2\delta+1)(\delta+1)^\nu + 1} \geq 4 \B = 4 \beta \bLowercase ^\delta \geq 3\bLowercase
\;.
\end{align*}

We derive the first lower bound for $\Y$ in the same way:
\[ \Y = N^2 \geq N = N_0 N_1 \geq N_0 = \B^{(\delta + 1)^\nu + 1} \geq \B = \beta \bLowercase ^\delta \geq \bLowercase
\;.
\]
The second bound $\Y\ge 2^8$ follows directly from the definition $\Y=N^2$.
\end{proof}


\Newpage

\subsection{Upper Bounds}

\begin{lemma}[Upper bounds on $\S, \T$: \afpfact{Coding_Theorem}{Coding_Theorem.coding_theorem.T_bound}] \label{lemma:coding-upper-bounds}
Suppose $P$ is \virtuallypositive{} and $a\ge 0$, $f>0$. Then we have that $M < N_0$.
Moreover, if $g \in [0, 2\bLowercase \B^{(\delta+1)^\nu})$, then $\S < N$ and $\T < N$.
\end{lemma}

\begin{proof}
We start bounding $M$ by unfolding the definition of mask in Definition~\ref{def:mask}. In that definition, all $m_j$ are either $\B - 1$ or $\B - \bLowercase$ so in particular $m_j < \B$. This gives
\begin{equation}
  M =\mask(\bLowercase,\B,\mathbf{n})= \sum_{j=0}^{n_\nu} m_j \B^j < \B^{n_{\nu}+1} = \B^{(\delta+1)^\nu + 1} = N_0 \label{eq:MltN0}.
\end{equation}

We now fix $g \in [0, 2\bLowercase \B^{(\delta+1)^\nu})$. Note that $\S$ and $\T$ are defined in terms of the polynomials $\K, N_0, M, \B$, which in turn are defined in terms of $c$ and $\coeffs(\B)$.

Our first bound goes as follows, using $\bLowercase>a$ by definition:
\begin{equation}  \label{eq:c-bound}
c^\delta
= (1+a\B+g)^\delta
< \big(\bLowercase\B + 2\bLowercase\B^{(\delta+1)^\nu}\big)^\delta
< \big(3\bLowercase\B^{(\delta+1)^\nu}\big)^\delta
= 3^\delta \bLowercase^\delta \B^{\delta(\delta+1)^\nu}
\;.
\end{equation}
Now we establish an upper bound on $\coeffs(\B)$ in the same way as before (again using $|a_{\mathbf{i}}|\le \L$)
\begin{align} \label{eq:DB-bound}
    \coeffs(\B) &= \sum_{\|\mathbf{i}\| \le \delta} \mathbf{i}!\, (\delta - \|\mathbf{i}\|)!\, a_{\mathbf{i}}\, \B^{(\delta+1)^{\nu+1} - \sum_{s=0}^{\nu}i_s(\delta+1)^s} \nonumber \\
    &<\sum_{\|\mathbf{i}\| \le \delta} \delta ! \,|a_{\mathbf{i}}|\,\B^{(\delta+1)^{\nu+1}} < \delta!\,\mathcal{L}\B^{(\delta+1)^{\nu+1}}
\;.
\end{align}
We observe one final inequality, using the definition of $\beta$ in \eqref{Eq:DefBeta}.
\begin{equation} \label{eq:B-bound}
  \B = \beta \bLowercase^\delta = 4^r \bLowercase^\delta > 2 \bLowercase^\delta 3^\delta \delta! \mathcal{L}
\;.
\end{equation}
We can combine the inequalities  \eqref{eq:c-bound}, \eqref{eq:DB-bound} and \eqref{eq:B-bound} to obtain
\begin{align}
  \left| c^\delta \coeffs(\B) \right| <
  3^\delta \bLowercase^\delta  \B^{\delta(\delta+1)^\nu}  \delta!\,\L\B^{(\delta+1)^{\nu+1}}
  < \frac{1}{2}\B^{(2\delta+1)(\delta+1)^{\nu}+1}
\;.
\end{align}
Our next step is to give an upper bound on $\K$, as follows:
\begin{align} \label{eq:K-bound}
  \K &= \values(c, \B) = c^\delta \coeffs(\B) + \sum_{i=0}^{(2 \delta+1)(\delta + 1)^\nu} \frac{\B}{2} \B^{i} \nonumber \\
& <  \frac{1}{2}\B^{(2\delta+1)(\delta+1)^{\nu}+1} + \frac{\B}{2}\cdot  2\,\B^{(2\delta+1)(\delta+1)^\nu} = \frac{3}{2}\B^{(2\delta+1)(\delta+1)^{\nu}+1}
\;.
\end{align}

Now everything is set up to prove the bound on $\T$:
\begin{equation} \begin{split}
  \T &= M + (\B-2)\B^{(\delta+1)^{\nu + 1}} N_0 < N_0 + (\B-2)\B^{(\delta+1)^{\nu + 1}} N_0 \\
  &\le \B^{(\delta+1)^{\nu + 1} + 1} N_0 \le N_0 N_1 = N
\;.
\end{split} \end{equation}

We continue with the bound on $\S = g + 2\K N_0$. First note that
\begin{equation}
  g < 2 \bLowercase\B^{(\delta+1)^\nu} < \B^{(\delta + 1)^\nu + 1}\label{eq:gltN0}
\end{equation}
Combining this with \eqref{eq:K-bound}, we obtain
\begin{align*}
    g + 2\K &<  \B^{(\delta + 1)^\nu + 1} + 3\B^{(2\delta +1)(\delta+1)^\nu +1} < 4\B^{(2\delta +1)(\delta+1)^\nu +1} = N_1
\;.
\end{align*}
Using this we can finally prove the upper bound on $\S$
\begin{equation}
  \S = g + 2 \K N_0 < (g + 2\K) N_0 \le N_1 N_0 = N
\;.
\qedhere
\end{equation}
\end{proof}

Now we are ready to state and prove our first main theorem, which turns a Diophantine set of arbitrary complexity (i.e.\ arbitrary number of unknowns $\nu$) into an equivalent condition of fixed complexity, involving a binomial coefficient. The key ingredient is a \emph{code} $g$ in the sense of Definition~\ref{def:mask}.
In the subsequent sections of this paper, the binomial coefficient will then be transformed into a Diophantine polynomial of bounded complexity.

\subsection{Encoding a Diophantine Set}

The first main step in the proof is to encode a Diophantine set, described by a polynomial $P$, by number theoretic conditions. These involve exponentials, but have bounded complexity, so they can be turned in later steps into polynomial equations with bounded complexity.

\begin{mainstep}[Coding Theorem: \afpfact{Coding_Theorem}{Coding_Theorem.coding_theorem.coding_theorem_direct}]
\label{step:polynomial-to-binomial} \sunref{\cite[Theorem 3.1]{Sun}}
Given a polynomial $P$ that is suitable for coding, define the numbers $\alpha$ and $\gamma$ as well as the polynomials $\bLowercase, \X, \Y$ as in Definition~\ref{def:codingVariables}. For $a \ge 0$ and $f \neq 0$ such that
\begin{equation}
\bLowercase := \bLowercase(a, f) \in 2\uparrow
\end{equation}
 the following are equivalent:
\begin{equation}
	\exists (z_1,\ldots,z_\nu)\in [0, \bLowercase)^\nu\colon P(a, z_1, \ldots, z_\nu) =0
\label{Eq:polynomial-zero}
	\end{equation}
	and
\begin{equation}
	\exists g \in [\bLowercase, \gamma \bLowercase^\alpha) \colon  \Y(a, f) \, |\, \binom{2 \X(a, f, g)}{\X(a, f, g)}
\;.
\label{Eq:divisibility-binomial}
	\end{equation}

Moreover, \eqref{Eq:divisibility-binomial} $\implies $ \eqref{Eq:polynomial-zero} already holds when $g \in [0, 2\gamma \bLowercase^\alpha)$.

\end{mainstep}

Condition~\eqref{Eq:polynomial-zero} is stronger than saying that $P$ has a Diophantine solution for parameter $a$ with $\nu$ unknowns: it also requires that the $z_1, \ldots, z_\nu$ are bounded. However, when using the theorem we are able to choose $f$ freely and thus increase $\bLowercase$ such that it is larger than any given $z_1, \ldots, z_\nu$.

Like before, we often suppress the dependence of our polynomials on $a, f, g$. The numbers $a$ and $f$ are fixed anyway, but occasionally we write $\X(g)$ when we want to emphasize the dependence on the code $g$.

\begin{FormalizationNote}\label{formalization:b-squared}
Our initial statement of \Cref{step:polynomial-to-binomial} contained the assumption, $\bLowercase(a, f) \in \square$, originating from a more general version of the analogous theorem by Sun~\cite[Theorem 3.1]{Sun}. At some point of the project we realized that this assumption might not be needed for our purposes and experimented with what happens if we remove it from the formal statement in Isabelle. The proof assistant still judged our whole proof as correct, immediately confirming that the condition was in fact redundant as we had supposed. 
\end{FormalizationNote}

\begin{proof}[Proof of the theorem]
We prove this theorem by showing two separate implications. First, however, we will establish some basic properties which will be useful for both directions.

Observe that a priori, $f$ can be negative. However, we clearly have $\bLowercase>0$ because $\bLowercase\in 2\uparrow$, and thus $f>0$ in any case because $\bLowercase = 1 + 3(2a+1)f$. This implies
\begin{equation} \label{eq:a-lt-b}
  a < \bLowercase
\;.
\end{equation}
Moreover, we can use the estimates from Lemmas~\ref{lemma:coding-lower-bounds}--\ref{lemma:coding-upper-bounds} that were proved under the hypothesis $f>0$.

Next, we note that the hypothesis $\bLowercase \in 2 \uparrow$ implies that $\B =4^r\bLowercase^\delta \in 2\uparrow$ and hence
\begin{equation} \label{eq:vars-pow-2}
N_0, N_1, N \in 2 \uparrow
\;.
\end{equation}

\proofstep{The implication \eqref{Eq:polynomial-zero} $\implies$ \eqref{Eq:divisibility-binomial}.}

By hypothesis, there is a tuple $(z_1,\ldots,z_\nu) \in [0, \bLowercase)^\nu$ such that $P(a, z_1, \ldots, z_\nu) = 0$. We need to construct $g \in [\bLowercase, \gamma \bLowercase^\alpha)$ such that $\Y | \binom{2\X(g)}{\X(g)}$. Define
\[
g := \sum_{i=1}^\nu z_i \B^{(\delta+1)^i}
\;,
\]
which is a code in the sense of Definition~\ref{def:mask}.

We start by checking that $g$ is indeed in the interval $[\bLowercase, \gamma \bLowercase^\alpha)$. Since $P$ is \virtuallypositive{}, there is an $i$ such that $z_i > 0$, so
\begin{equation}
g \geq z_i \B = z_i \beta \bLowercase^\delta \geq \bLowercase
\;. \label{eq:g-geq-b}
\end{equation}
We also have the upper bound
\begin{equation} \label{eq:g-upper-bound}
g \leq \sum_{i=1}^{\nu} (\bLowercase-1) \B^{(\delta+1)^i}
< \bLowercase \B^{(\delta+1)^\nu}
= \bLowercase (\beta \bLowercase^\delta)^{(\delta+1)^\nu}
= \beta^{(\delta+1)^\nu} \bLowercase^{\delta (\delta+1)^\nu + 1}
= \gamma \bLowercase^\alpha
\;.
\end{equation}

It remains to show that $\Y | \binom{2\X}{\X}$. To prove this we will now consecutively apply the coding lemmas from the previous section. We begin by verifying that all of the hypotheses of Lemma~\ref{lemma:polynomial-zero-tau-zero} are satisfied.

We know $\L = \sum_{\mathbf{j}} |a_{\mathbf{j}}| \geq |a_{\mathbf{i}}|$, and  have already observed that $\B \in 2 \uparrow$. We have $z_1,\ldots,z_\nu<\bLowercase$; setting $z_0:=a$ for convenience, \eqref{eq:a-lt-b} shows $z_0<\bLowercase$ and thus $\|\mathbf{z}\|<\bLowercase (\nu+1)$. Using the definitions of $\B$ and $\beta$, we conclude
\[
\B = \beta \bLowercase^\delta > 2\delta! \L (\nu+2)^\delta \bLowercase^\delta
> 2 \delta! \L (1 + \|{\mathbf{z}}\|)^\delta
\;.
\]

Since $z_0 = a$, we can simply write
\[
c = 1 + a \B + g = 1 + \sum_{i=0}^\nu z_i \B^{(\delta+1)^i} = 1 + \code(\mathbf{z}, \B, \mathbf{n})
\;.
\]
Consequently, $\K$ coincides with $K$ defined in Lemma~\ref{lemma:polynomial-zero-tau-zero},  equation~\eqref{eq:code-K-def}. Since $P(z_0,\ldots,z_\nu)=0$ by hypothesis, we can apply Lemma~\ref{lemma:polynomial-zero-tau-zero} and obtain
\begin{align}
\tau\left(\K, (\B/2-1)\B ^{(\delta+1)^{\nu + 1}}\right) = 0
\;.
\end{align}
Multiplying both arguments of $\tau$ by $2$ allows us to eliminate the division:
\begin{equation}
  \tau\lp 2 \K, (\B-2)\B ^{(\delta+1)^{\nu + 1}}\rp = 0
\;. \label{eq:tau2KB0}
\end{equation}

We now apply Lemma~\ref{lemma:code-mask}. The hypotheses $\bLowercase, \B \in 2 \uparrow$ have been checked before. The inequality $\bLowercase \leq \B$ follows from the definition $\B = \beta \bLowercase^\delta$ and we have shown $\B \geq 2$ in Lemma~\ref{lemma:coding-lower-bounds}. Note that $g \geq 0$ by \eqref{eq:g-geq-b} and that by definition $g$ is a code for $(z_1, \ldots, z_\nu)$. The lemma thus implies
\begin{equation}
\tau(g, M) = 0
\;.
\label{eq:tau-g-M-is-0}
\end{equation}

Our next claim is  that the two conditions \eqref{eq:tau2KB0} and \eqref{eq:tau-g-M-is-0} can be combined, using Lemma~\ref{Lem:TauBasic} \eqref{item:tauaddcombine}, into the single condition
\begin{equation} \label{eq:tau-2-g-K}
\tau(\S, \T) =\tau\lp g + 2 \K N_0, M + (\B-2) \B^{(\delta+1)^{\nu + 1}}N_0\rp  = 0
\;.
\end{equation}
We can do this since $N_0 \in 2 \uparrow$ by \eqref{eq:vars-pow-2}, while $M < N_0$ by Lemma~\ref{lemma:coding-upper-bounds}, and
\[
g \leq \bLowercase \B^{(\delta +1)^\nu} < \B^{(\delta +1)^\nu + 1} = N_0
\;
\]
by \eqref{eq:g-upper-bound}.

We are now in a position to apply Lemma~\ref{lemma:tau-binom}. The fact that $N \in 2 \uparrow$ is \eqref{eq:vars-pow-2} and $0 \le \S, \T < N$ follows from Lemmas~\ref{lemma:coding-lower-bounds} and \ref{lemma:coding-upper-bounds}.
We can therefore apply the lemma to conclude
\[
N^2 \,|\, \binom{2(N-1)\RR}{(N-1)\RR}
\;,
\]
which translates to
\[ \Y \,|\, \binom{2\X}{\X} \]
using the definition of $\X, \Y$. This completes the proof of the first implication.

\needspace{3\baselineskip}
\proofstep{The converse direction \eqref{Eq:divisibility-binomial} $\implies$ \eqref{Eq:polynomial-zero}}{ (assuming $g \in [0,2\gamma \bLowercase^\alpha)$).}

Fix $g \in [0, 2 \gamma \bLowercase^{\alpha})$ such that
\[ \Y \,|\, \binom{2\X}{\X} \]
where of course $\X := \X(a, f, g)$ as always, also for $\Y$, $N$, $\RR$, etc. Using the definitions of $\X, \Y$ we obtain
\[
N^2 \,|\, \binom{2(N-1)\RR}{(N-1)\RR}
\;.
\]
Since $N \in 2 \uparrow$ by \eqref{eq:vars-pow-2} and $0 \le \S, \T < N$ by Lemmas~\ref{lemma:coding-lower-bounds} and \ref{lemma:coding-upper-bounds}, we can apply Lemma~\ref{lemma:tau-binom} and unfold the definitions of $\S$ and $\T$ to conclude
\[
0=\tau(\S, \T) = \tau\lp g + 2 \K N_0, M + (\B-2) \B^{(\delta+1)^{\nu + 1}}N_0\rp
\;.
\]
As before, we can separate this into two separate conditions involving $\tau$ using Lemma~\ref{Lem:TauBasic} \eqref{item:tauaddsplit} since $N_0 \in 2 \uparrow$, $M < N_0$ by Lemma~\ref{lemma:coding-upper-bounds} and by assumption
\begin{equation} \label{eq:reverse-g-lt-N0}
g
< 2 \gamma \bLowercase^\alpha
= 2\bLowercase (\beta \bLowercase^\delta)^{(\delta+1)^\nu}
= 2\bLowercase \B^{(\delta+1)^\nu}
\leq \B^{(\delta+1)^\nu + 1} =  N_0
\; .
\end{equation}
We obtain
\begin{align} \label{eq:reverse-tau-g-M-is-zero}
&\tau(g, M) = 0 \qquad\text{and} \\
&\tau\lp 2\K, (\B-2) \B^{(\delta+1)^{\nu + 1}}\rp = 0\label{eq:reverse-tau-2-K}  \qquad\text{and thus} \\
&\tau\left(\K, (\B/2-1)\B ^{(\delta+1)^{\nu + 1}}\right) = 0
\label{eq:reverse-tau-2-K-halbiert}
\end{align}
because $\B$ is even.

In order to apply Lemma~\ref{lemma:code-mask}, we have already checked  that $\bLowercase, \B \in 2 \uparrow$, $\bLowercase \leq \B$ and $\B \geq 2$. Note that $g \geq 0$ by assumption and in \eqref{eq:reverse-g-lt-N0} we showed $g < N_0 = \B^{(\delta+1)^\nu + 1} = \B^{n_\nu + 1}$. Therefore, using \eqref{eq:reverse-tau-g-M-is-zero}, the lemma implies
that there are $z_1, \ldots, z_\nu \in [0, \bLowercase)$ with
\[
g = \sum_{i=1}^\nu z_i \B^{n_i}
\;.
\]

The next step is to apply Lemma~\ref{lemma:polynomial-zero-tau-zero}. We have specific values for $z_1, \ldots, z_\nu$ and set again $z_0:=a$.
Note that all $z_1, \ldots, z_\nu$ are bounded by $\bLowercase$ and also $z_0 = a < \bLowercase$ by \eqref{eq:a-lt-b}.
The conditions $\L\ge|a_{\mathbf{i}}|$  and $\B \in 2 \uparrow$ have been checked before. We can further establish the required lower bound on $\B$
\[
\B = \beta \bLowercase^\delta > 2\delta!\L (\nu + 2)^\delta \bLowercase^\delta \geq 2\delta! \L (1 + z_0 + \ldots + z_\nu)^\delta
\;.
\]
Note that $c = 1 + a\B + g = 1 + \sum_{i=0}^\nu z_i \B^{n_i}$, so the definitions of $K$ and $\K$ coincide.

\enlargethispage{2\baselineskip} 
Using \eqref{eq:reverse-tau-2-K-halbiert}, we complete the proof because Lemma~\ref{lemma:polynomial-zero-tau-zero} implies
\[
P(a, z_1, \ldots, z_\nu) = 0.
\]
\end{proof}

%% file: section3.tex
Theorem~\ref{step:polynomial-to-binomial} reduces the existence of solutions of an arbitrary Diophantine polynomial in $\nu$ unknowns to the existence of a single code $g$ that satisfies a divisibility condition on a certain binomial coefficient (\cref{Eq:divisibility-binomial}), subject to the condition that the polynomial $\bLowercase(a,f)$ is a power of two. 

To be a power of two is, of course, not a priori a Diophantine condition, but this condition can be encoded by well-known recursive sequences called \emph{Lucas sequences}. This is one of the key results of this section (Lemma~\ref{lemma:LucasCongruenceNEW}). Next, we have a pair of lemmas (Lemmas~\ref{lemma:Sun4.3} and \ref{lemma:PellDiophantine2}) that recognize values of Lucas sequences in terms of elementary Diophantine conditions. Unfortunately, these Diophantine conditions consist of many sub-equations, which we collect in Definition~\ref{def:VariableDefinition} and relations~\eqref{eq:DFI}--\eqref{eq:weakineq}. 

\begin{definition}[Lucas sequence: \afpfact{DFI_square_1}{DFI_square_1.\%5C\%3Cpsi\%3E_int_def}]
A \emph{Lucas sequence with parameter $A\in\Z$} is a sequence $(x_n)_{n\in \Z}$ that satisfies the recursive relation
\begin{equation}
x_{n+1}=Ax_n-x_{n-1}
\;.
\end{equation}
Every such sequence is defined uniquely by two initial values, and we define two particular Lucas sequences with initial values
\begin{align}
\psi_0(A)=0, \quad \psi_1(A)=1 \;; \\
\chi_0(A)=2, \quad \chi_1(A)=A \;.
\end{align}

\end{definition}

Here are a few examples of their explicit values:
\begin{equation}
\begin{array}{r|rrrrrrr}
n&-1 & 0 & 1 & 2 & 3 & 4  & 5 \\ \hline \\[-.85em] 
\psi_n(A) & -1 & \mathbf 0 & \mathbf 1 & A & A^2-1 & A^3-2A & A^4-3A^2+1  \\ 
\chi_n(A) & A & \mathbf 2 & \mathbf A & A^2-2 & A^3-3A & A^4 - 4A^2+2 & A^5-5A^3 +5A
\end{array}
\label{Eq:LucasExamples}
\end{equation}

Lucas sequences have numerous nice properties; some of these are listed in Lemma~\ref{lemma:LucasElementary} below.

A remarkable classical property of Lucas sequences is that in certain cases they provide the solutions to Pell's equation
\begin{equation}
X^2-dY^2=4
\label{Eq:PellEquation}
\end{equation}
in which $d$ is a positive integer and not a square.

\begin{lemma}[Pell's equation: \afpfact{Pell_Equation}{Pell_Equation.lucas_pell_nat}]
\label{lemma:PellEquation}
If $d=A^2-4$, the Pell equation \eqref{Eq:PellEquation}
has the property that its solutions in $\N$ are exactly the pairs $(X,Y)=(\chi_n(A),\psi_n(A))$ for $n\in\Z$.
\end{lemma}

This result is well known~\cite{Leh28,MR75}.

\begin{corollary}[\afpfact{DFI_square_1}{DFI_square_1.pell_yields_lucas}]
\label{cor:PellEquation}
For given integers $A$ and $X$, 
\[
(A^2-1)X^2+1 \in \square \Longleftrightarrow \exists m\in \Z:\, X=\psi_m(2A).
\]
\end{corollary}

Many of the following properties are formalized in the file \coloredlink{https://isa-afp.org/thys/Diophantine_Universal_Pairs/Lucas_Sequences.html}{Lucas\_Sequences}.

\begin{lemma}[Elementary properties of Lucas sequences] 
\label{lemma:LucasElementary}
The Lucas sequences $\psi_n(A)$ and $\chi_n(A)$ have the following properties:
\begin{enumerate}
\item
strict monotonicity: if $A>1$ then $\psi_{n+1}(A)>\psi_n(A)$ for all $n \ge 0$.
\label{item:LucasMonotonicity}
\item
exponential growth: $(A-1)^n<\psi_{n+1}(A)<A^n$ for $n \ge 2$ and $A>1$.
\label{item:LucasGrowth}
\item
symmetry in $A$: 
 $\psi_n(-A)=(-1)^{n+1}\psi_n(A)$ and $\chi_n(-A) = (-1)^n \chi_n(A)$.
\label{item:LucasSymmetryA}
\item 
symmetry in $n$: 
$\psi_{-n}(A)=-\psi_n(A)$ and $\chi_{-n}(A) = \chi_n(A)$.
\label{item:LucasNegative}
\item 
consecutive elements are coprime: 
$\mathrm{gcd}(\psi_{n}(A), \psi_{n+1}(A)) = 1$ for $n\geq 0$.
\label{item:LucasConsecCoprime}
\item
special behavior for $A=2$: 
$\psi_n(2)=n$ and $\chi_n(2)=2$ for all $n$.
\label{item:LucasCoeff2}
\item
law of apparition: for every $A$ and every $N>0$ there are infinitely many $n$ so that $N|\psi_n(A)$.
\label{item:LucasApparitionRepetition}
\end{enumerate}
Moreover, the Lucas sequences $\psi_n(A)$ and $\chi_n(A)$ have the following properties. Abbreviate $\alpha = \sqrt{A^2 - 4} \in \mathbb{R}$.
\begin{enumerate}[resume]
\item
if $A$ is even, then $\psi_n(A)$ has the same parity as $n$, so $\psi_n(A)-n$ is even.
\label{item:LucasEven}
\item
if $A$ is even, then $\chi_n(A)$ is even.
\label{item:ChiEven}
\item 
if $A > 2$, then $\alpha \cdot 2^n \psi_n(A) = (A+\alpha)^n - (A-\alpha)^n$ in the reals.
\label{item:PsiClosedForm}
\item 
if $A > 2$, then $2^n \chi_n(A) = (A+\alpha)^n + (A-\alpha)^n$.
\label{item:ChiClosedForm}
\item
if $A > 2$, then $\psi_{n+1}(A) > \frac{5}{2} \cdot \psi_n(A)$.
\label{item:psibound}
\item
if $A > 2$, then $\chi_n(A) > \sqrt{5} \cdot \psi_n(A)$.
\label{item:ChiPsiBound}
\item
doubling: $2\psi_n(A) = A\psi_{n+1}(A) - \chi_{n+1}(A)$.
\label{item:ChiPsiRecursion}
\item
index shift: $\psi_{n+r}(A) = \psi_r(A) \chi_n(A) + \psi_{n-r}(A)$.
\label{item:PsiIndexShift}
\end{enumerate}
\end{lemma}

All of these properties are proven by induction and are well known~\cite{MR75,9var}. Note that property~\eqref{item:LucasGrowth} implies exponential growth\footnote{This notion of exponential growth is distinct from the ``relation of exponential growth'' introduced by Julia Robinson~\cite{robinson1952}.} only when $A>2$; compare property~\eqref{item:LucasCoeff2}.

\begin{FormalizationNote}\label{formalization:A-gt-2}
Compared to our original version of \cref{lemma:LucasElementary}, a lot of edge cases have been clarified through the formalization. For example, we often needed to use some of the statements in this lemma for the case $A=2$, when it had originally only been proven for $A > 2$, instead of $A > 1$.
\end{FormalizationNote}

\begin{lemma}[\afpfact{DFI_square_3}{DFI_square_3.lucas_modN_int}]\label{lemma:LucasSequenceCongruence}
Let $A, A'$ be integers. If there is an $n > 1$ such that $A \equiv A' \pmod n$, then $\psi_m(A) \equiv \psi_m(A') \pmod n$ for every $m \in \Z$.
\end{lemma}

\begin{proof}
The cases $m = 0, 1$ are immediate. For $m>1$, this follows by induction, and for $m<0$ it follows from symmetry in $m$; see Lemma~\ref{lemma:LucasElementary}~\eqref{item:LucasNegative}.
\end{proof}

A particular case will be especially useful. Note that we allow modular congruence for negative $n < 0$ by saying that two integers $A \equiv A' \pmod{n}$ are congruent if $A \equiv A' \pmod{|n|}$.

\begin{corollary}[\afpfact{Bridge_Theorem_Rev}{Bridge_Theorem_Rev.lucas_modN_int}]\label{cor:LucasSequenceCongruence}
Let $|\alpha| > 2$ and $m\in \Z$, then 
\begin{equation}
    \psi_m(\alpha)\equiv \psi_m(2) =m \pmod {\alpha-2}.
\end{equation}
\end{corollary}

\sunref{\cite[Lemma 2]{SunDFI}}
\begin{lemma}[\afpfact{DFI_square_1}{DFI_square_1.sun_lemma2}]\label{lemma:psi-linear-expansion}
Let $A, n, k, r$ be integers with $n, k, r \geq 0$. Then
    \begin{equation}\label{eq:psi-linear-expansion}
    \psi_{nk+r}(A) = \sum_{i=0}^n {\binom{n}{i}} \cdot (\psi_{k+1}(A) - A \psi_k(A))^{n-i} \cdot \psi_k(A)^i \cdot \psi_{r+i}(A)
    \end{equation}
\end{lemma}
\begin{proof}
In this proof, the parameter of the Lucas sequence is always $A$ so we omit it for a more compact presentation (writing $\psi_k$ for $\psi_k(A)$). The proof will proceed by induction on $n$. To start, we prove the following statement which will be useful in the inductive step (this is also the lemma statement for $n=1$).
\begin{equation} \label{eq:psi-n-1} \psi_{k+r} = (\psi_{k+1} - A \psi_k) \cdot \psi_r + \psi_k \cdot \psi_{r+1}
\end{equation}
We can prove this identity by induction on $k$. The base cases $k=0$ and $k=1$ can be checked directly from the recursive definition of $\psi$. For the induction step, assume that Equation~\eqref{eq:psi-n-1} holds for $k$ and $k+1$. We then obtain
\[ \begin{split}
    \psi_{k+2+r} &= A\psi_{k+1+r} - \psi_{k+r} \\
    &= A\left( (\psi_{k+2}-A\psi_{k+1})\psi_r + \psi_{k+1}\psi_{r+1} \right) 
        - \left( (\psi_{k+1}-A\psi_{k})\psi_r + \psi_{k}\psi_{r+1} \right) \\
    &= \left( A (\psi_{k+2}-A\psi_{k+1}) - (\psi_{k+1}-A\psi_{k}) \right)\psi_r 
    + (A\psi_{k+1} - \psi_k)\psi_{r+1}\\
    &= \left( A\psi_{k+2}-\psi_{k+1} - A(A\psi_{k+1}-\psi_{k}) \right)\psi_r 
    + \psi_{k+2}\psi_{r+1}\\
    &= (\psi_{k+3}-A\psi_{k+2})\psi_r 
    + \psi_{k+2}\psi_{r+1}
\end{split} \]
which is \eqref{eq:psi-n-1} for $k+2$.

We can now prove \eqref{eq:psi-linear-expansion} by induction. For $n=0$ the statement is trivial since there is only one summand. We proceed with the inductive step, assuming \eqref{eq:psi-linear-expansion} for $n$. One can then calculate 
\begin{align*}
\psi_{(n+1)k+r} 
=& \psi_{(nk+r)+k} 
\overset{\eqref{eq:psi-n-1}}{=} (\psi_{nk+r+1} - A\psi_{nk+r}) \psi_k + \psi_{nk+r} \psi_{k+1} \\
=& (\psi_{k+1} - A\psi_k)\psi_{nk+r} + \psi_k \psi_{nk+(r+1)} \\
\overset{\text{IH}}{=}& (\psi_{k+1} - A\psi_k)\left(\sum_{i=0}^n \binom{n}{i} (\psi_{k+1} - A \psi_k)^{n-i} \psi_k^i \psi_{r+i} \right) \\
&+ \sum_{i=0}^n \binom{n}{i} (\psi_{k+1} - A \psi_k)^{n-i} \psi_k^{i + 1} \psi_{r+1+i} \\
=& \sum_{i=0}^n \binom{n}{i} (\psi_{k+1} - A \psi_k)^{n+1-i} \psi_k^i \psi_{r+i}  
+ \sum_{i=1}^{n+1} \binom{n}{i-1} (\psi_{k+1} - A \psi_k)^{n+1-i} \psi_k^{i} \psi_{r+i} \\
=& \sum_{i=0}^{n+1} \binom{n+1}{i} (\psi_{k+1} - A \psi_k)^{n+1-i} \psi_k^{i} \psi_{r+i}
\end{align*}
where the last step follows because $\binom{n+1}{i} = \binom{n}{i} + \binom{n}{i-1}$.
\end{proof}

\begin{lemma}[\afpfact{DFI_square_1}{DFI_square_1.sun_lemma7}]
\label{lemma:sun7} \sunref{\cite[Lemma 3.7]{Sun}}
    Let $A,k,m$ be integers with $|A|\ge 2$, $k>0$ and assume that $\psi_k(A)^2|\psi_m(A)$. Then $\psi_k(A)|m$.
\end{lemma}
\begin{proof}
    First, we will prove that if $\psi_k(A)|\psi_m(A)$ then $k|m$. Write $m = nk + r$ with $r < k$. By Lemma~\ref{lemma:psi-linear-expansion} we have
    \begin{equation}\label{eq:psi-linear-expansion2}
    \psi_{nk+r}(A) = \sum_{i=0}^n {\binom{n}{i}} \cdot (\psi_{k+1}(A) - A \psi_k(A))^{n-i} \cdot \psi_k(A)^i \cdot \psi_{r+i}(A)
    \end{equation}
    so that only the $i=0$ term survives after reduction modulo $\psi_k(A)$. Hence
    \begin{align*}
    \psi_{m}(A) = \psi_{nk+r}(A) &\equiv (\psi_{k+1}(A) - A \psi_k(A))^n \cdot \psi_{r}(A) \pmod{\psi_k(A)} \\
    &\equiv \psi_{k+1}(A)^n \cdot \psi_{r}(A) \pmod{\psi_k(A)}
    \end{align*}
    Using the fact that consecutive elements of the Lucas sequence $\psi_k(A),\psi_{k+1}(A)$ are coprime (\cref{lemma:LucasElementary} \eqref{item:LucasConsecCoprime}), we conclude that $\psi_{r}(A) \equiv \psi_{m}(A) \equiv 0 \pmod{\psi_k(A)}$ using the assumption. But $r < k$ directly implies $\psi_{r}(|A|) < \psi_k(|A|)$ by strict monotonicity (\cref{lemma:LucasElementary} \eqref{item:LucasMonotonicity}), which means that $|\psi_{r}(A)| < |\psi_k(A)|$. So we have proven $|\psi_r(A)| = 0$ which requires $r = 0$.

    Now, revisit equation~\eqref{eq:psi-linear-expansion2} but reduce modulo $\psi_k(A)^2$ instead. The $i=1$ term initially survives, which yields
    \[
    0 \equiv \psi_m(A) \equiv n (\psi_{k+1}(A) - A \psi_k(A))^{n-1} \cdot \psi_k(A) \cdot 1 \pmod{\psi_k(A)^2}
    \]
    This directly implies that $0 \equiv n \cdot \psi_{k+1}(A)^{n-1} \pmod{\psi_k(A)}$. Again using that consecutive elements of the Lucas sequence $\psi_k(A),\psi_{k+1}(A)$ are coprime, we find $\psi_k(A) | n$. 
    Finally, $m = nk$ is a multiple of $n$ so we conclude that $\psi_k(A) | m$.
\end{proof}

\begin{lemma}[\afpfact{DFI_square_2}{DFI_square_2.distinct_residus}]
\label{lemma:psiInjective}
Let $A,n,k$ be integers with $A > 2$, $n > 3$, such that $\chi_{n}(A) = 2k$. Then the residue map $\Psi : [-n,+n] \rightarrow \mathbb{Z}/k\mathbb{Z}$ defined by $\Psi(i) = (\psi_i(A) \!\mod k)$ is injective.
\end{lemma}
\begin{proof}
    We will show that for two integers $i, j \in [-n, +n]$ with $i \neq j$, we have $\psi_{i}(A) \not\equiv \psi_j(A) \pmod{k}$ by distinguishing four cases.

    \begin{enumerate}[label={(\roman*)}]
        \item $0 \leq i < j \leq n$. Recall that $\psi_0(A) = 0$. We find that 
        \[
            0 \leq \psi_i(A) < \psi_j(A) \leq \psi_n(A) < \chi_n(A)/2 = k
        \]
        using the elementary properties in \cref{lemma:LucasElementary} \eqref{item:LucasMonotonicity} and \eqref{item:ChiPsiBound}.
        Flipping the signs of this chain of inequalities gives the same result for $-n \leq j < i \leq 0$.
        
        \item $-n < i < 0 < j < n$. In this case of $i$ and $j$ with opposite signs, we want to show that $\psi_j(A) - \psi_i(A) \not\equiv 0 \pmod{k}$. Indeed,
        \[
            \qquad\qquad 0 < \psi_j(A) - \psi_i(A) = \psi_j(A) + \psi_{|i|}(A) \leq 2 \psi_{n-1}(A) < \psi_n(A) < k, 
        \]
        using \cref{lemma:LucasElementary}~\eqref{item:psibound}.

        \item $i = -n$ and $0 < j < n$. Use \cref{lemma:LucasElementary} \eqref{item:LucasMonotonicity}, \eqref{item:psibound}, \eqref{item:ChiPsiBound} and \eqref{item:ChiPsiRecursion}. This time, we find $0 < \psi_j(A) - \psi_{-n}(A) < 2\psi_n(A) < 2k$.
        \begin{enumerate}[label={\greek*)}]
            \item $j \leq n-3$. Then, 
            \begin{align*}
                \psi_j(A) - \psi_{-n}(A) &< \psi_{n-3}(A) + \psi_n(A) \\
                &< \left(\tfrac{8}{125} + 1\right) \psi_n(A) < \tfrac{\sqrt{5}}{2} \psi_n(A) < k.
            \end{align*}
            \item $j = n-2$. By contradiction, assume that $\psi_{n-2}(A) + \psi_{n}(A) = \frac{1}{2} \chi_n(A)$. Then,
            \[
                \qquad\qquad \chi_n(A) = 2A \psi_{n-1}(A) = A^2 \psi_n(A) - A \chi_n(A)
            \]
            Using Pell's equation for $\chi_n^2(A)$ we obtain the condition
            \[
                (1+A)^2 \left(4 + (A^2 - 4) \psi_n^2(A)\right) = A^4 \psi_n^2(A)
            \]
            which factors to
            \begin{equation*}
                \qquad\qquad \left((2A+1)(A+1)(A-3)-1\right) \psi_n^2(A) + 4(A+1)^2 = 0.
            \end{equation*}

            Because $A > 2$, we have $4(A+1)^2 \geq 64$ which implies
            \begin{align*}
                \qquad\qquad (2A+1)(A+1)(A-3) - 1 &< 0, \\
                \qquad\qquad \text{i.e.} \quad (2A+1)(A+1)(A-3) &\leq 0.
            \end{align*}
            This requires that $A=3$, but plugging back into the above condition yields $\psi_n(3) = 8$. By monotonicity of $\psi_n(3)$, this yields $n=3$, which is the desired contradiction.
            \item $j = n-1$. By contradiction, assume that $\psi_{n-1}(A) + \psi_{n}(A) = \frac{1}{2} \chi_n(A)$. Then,
            \[ 
                \chi_n(A) = A\psi_n(A) - \chi_n(A) + 2\psi_n(A).
            \]
            Using Pell's equation for $\chi_n^2(A)$ we now obtain the condition
            \[
                (3A-10)(A+2) \psi_n^2(A) + 16 = 0
            \]
            This requires $3A - 10 < 0$ which again implies $A = 3$ under the assumptions of this lemma. Now, plugging back into this condition gives $5 \psi_n^2(3) = 16$, which is the desired contradiction as $5$ does not divide $16$.
        \end{enumerate}

        \item $i = -n$ and $j = n$. Again, $0 < \psi_n(A) - \psi_{-n}(A) = 2\psi_n(A) < 2k$. However, because $\psi_n(A) > 2$ under the assumptions of this lemma, we now have
        \begin{align*}
            4 \neq (20-A^2) \psi_n(A)^2 &= 16 \psi_n(A)^2 - (A^2-4)\psi_n(A)^2 \\ &= (4 \psi_n(A))^2 - \chi_n(A)^2 + 4
        \end{align*}
        so that $4 \psi_n(A) \neq \chi_n(A)$ and in particular $\psi_n(A) - \psi_{-n}(A) = 2\psi_n(A) \neq k$.
    \end{enumerate}    
    
    We conclude that $\psi_j(A) - \psi_i(A) \not\equiv 0 \pmod{k}$ in all of the above cases. This establishes that the residue map $\Psi : [-n,+n] \rightarrow \mathbb{Z}/k\mathbb{Z}$ is injective.
\end{proof}

\begin{lemma}[\afpfact{DFI_square_3}{DFI_square_3.sun_lemma10_dir_int}]
\label{lemma:sun10int} \sunref{\cite[Lemma 3.10]{Sun}}
Let $A,n,k,s,t$ be integers with $A > 2$, $n > 3$, such that $\chi_{n}(A) = 2k$ and $\psi_s(A)\equiv \psi_t(A) \pmod {k}$. Then $s\equiv t \pmod {2n}$ or $s\equiv -t \pmod {2n}$.
\end{lemma}
\begin{proof}
    Recall that \cref{lemma:LucasElementary}~\eqref{item:PsiIndexShift} states that $\psi_{n+r}(A) = \psi_r(A) \chi_n(A) + \psi_{n-r}(A)$.
    Specializing to $r = n+m$ gives $\psi_{2n+m}(A) \equiv -\psi_m(A) \pmod{\chi_n(A)}$. Using the assumption that $\chi_n(A) = 2k$ we thus have
    \begin{equation}\label{eq:lem-sun10-congruence1}
        \psi_{2n+m}(A) \equiv -\psi_m(A) \pmod{k}.
    \end{equation}
    Now, induction on $q$ additionally proves
    \begin{equation}\label{eq:lem-sun10-congruence2}
        \psi_{4nq+m}(A) \equiv \psi_m(A) \pmod{k}.
    \end{equation}


    Consider $\psi_s(A)\equiv \psi_t(A) \pmod {k}$ as in the assumption of the lemma. Using congruence~\eqref{eq:lem-sun10-congruence2}, we can reduce generic integers $s, t$, without loss of generality, to the case where $s, t \in [0, 4n]$.
    
    Now, using the congruence property \eqref{eq:lem-sun10-congruence1} for the case $[n, 3n]$ and congruence property \eqref{eq:lem-sun10-congruence2} for the case $[3n, 4n]$, we can further reduce this interval to new indices $s', t' \in [-n, n]$ and the relation $\psi_{s'}(A)\equiv \pm \psi_{t'}(A) \equiv \psi_{\pm t'}(A) \pmod {k}$.
    
    We proceed to use the injectivity of the map $\Psi$ established in \cref{lemma:psiInjective} to find $s' = \pm t'$. Now remember that $s', t'$ differ from $s, t$ by an integer multiple of $2n$ each.
    After reducing modulo $2n$, this completes the proof of the lemma.
\end{proof}

Next, we will establish a congruence of the Lucas sequence $\psi_B(A)$ with a genuine Diophantine expression; this is Lemma~\ref{lemma:LucasCongruenceNEW} below. In its proof, a central step is the congruence to an exponential Diophantine expression, i.e.\ Lemma~\ref{lemma:S4.5}.

\begin{lemma}[\afpfact{Lucas_Diophantine}{Lucas_Diophantine.lucas_exponential_diophantine}]
\label{lemma:S4.5} \sunref{\cite[Lemma 4.5]{Sun}}
Let $A,B,U,V \in \Z$ with $B>0$ and $U^2 - AUV + V^2 \neq 0$. Then 
\[
(UV)^{B-1}\psi_B(A)\equiv \sum_{r=0}^{B-1} U^{2r}V^{2(B-1-r)} \pmod{U^2-AUV+V^2}.
\]
\end{lemma}
\begin{proof}This is a straightforward proof by induction. For $B=1,2$ the statement is clear as $\psi_1(A)=1$ and $\psi_2(A)=A$. For $B\ge 3$, assume the statement holds for $B-1$ and $B-2$. Then, using the recursive definition of $\psi_B(A)$, we have
\begin{align*}
(UV)^{B-1}\psi_B(A) &= 
\left((UV)(UV)^{B-2}A\psi_{B-1}(A)-(UV)^2(UV)^{B-3}\psi_{B-2}(A)\rule{0pt}{12pt}\right) \\
			&\equiv AUV  \sum_{r=0}^{B-2} U^{2r}V^{2(B-2-r)}-(UV)^2\sum_{r=0}^{B-3} U^{2r}V^{2(B-3-r)} \\
			&\equiv (U^2+V^2)\sum_{r=0}^{B-2} U^{2r} V^{2(B-2-r)}-\sum_{r=0}^{B-3} U^{2r+2}V^{2(B-2-r)} 
\\			
&= U^2 \sum_{r=0}^{B-2} U^{2r} V^{2(B-2-r)} +V^2\sum_{r=1}^{B-2}U^{2r} V^{2(B-2-r)} + V^2U^0V^{2(B-2-0)} 
\\
&\qquad -\sum_{r=0}^{B-3} U^{2r+2}V^{2(B-2-r)} 
\\
&= \sum_{r=1}^{B-1} U^{2r} V^{2(B-1-r)}+ U^0V^{2(B-1-0)}
\\
&= \sum_{r=0}^{B-1} U^{2r}V^{2(B-1-r)} \pmod{U^2-AUV+V^2}
   \;.
\end{align*}
(From line 4 to 5, the second and third summations cancel after an index shift).
\end{proof}

Note that $U^2 - AUV + V^2 = 0$ can only happen when $U=V$ and $A=2$. In that case, the left and right-hand sides of the previous claim are already equal because $\psi_B(2) = B$. Either way, we will only use the case $U \neq V$ below.

\begin{lemma}[\afpfact{Lucas_Diophantine}{Lucas_Diophantine.lucas_diophantine_rec}]
\label{lemma:LucasCongruenceNEW} \sunref{\cite[Lemma 4.6]{Sun}}
For fixed $B>0$ consider the congruence
\begin{equation}
3W \psi_B(A)\equiv 2(W^2-1) \pmod {2A-5}
\;.
\label{Eq:CongruenceNew}
\end{equation}
If $W=2^B$, then the congruence is satisfied. Conversely, when the congruence is satisfied and $|A| \ge \max\{W^4, 2^{4B}\}$, then $W=2^B$.
\end{lemma}

\begin{FormalizationNote}\label{formalization:A-gt-0}
Through the formalization we discovered that \cref{lemma:LucasCongruenceNEW} holds for all $A \in \Z$ in the forward direction rather than just $A > 0$, as we had initially stated in our mathematical draft. 
\end{FormalizationNote}

\begin{proof}
We prove the first part via induction on $B$, even for $B\ge 0$. For $B=0$ we have $3\cdot 2^0\psi_0(A)=0= 2(2^0-1) $ and for $B=1$ we have $3\cdot 2^1\psi_1(A)=6= 2(2^2-1)$. To show that the congruence survives the inductive step, we compute
\begin{align*}
3\cdot 2^{B+1}\psi_{B+1}(A) &= 6A\cdot 2^B \psi_B(A)-12\cdot 2^{B-1}\psi_{B-1}(A) 
\\
&\equiv
4A(2^{2B}-1)-8(2^{2B-2}-1) \\
&= (4A-10)(2^{2B}-1)+10(2^{2B}-1)-8(2^{2B-2}-1) \\
&\equiv 2^{2B-2}(40-8)-2=2( 2^{2B+2}-1)
\;.
\end{align*}

For the converse direction, we multiply our hypothesis \eqref{Eq:CongruenceNew} on both sides by $2^{B-1}$ and obtain 
\begin{equation}
3W2^{B-1}\psi_B(A) \equiv 2^B(W^2-1) \pmod{2A-5}
\;.
\label{Eq:Sun4.6a}
\end{equation}
Now we use Lemma~\ref{lemma:S4.5} with $U = 2$ and $V = 1$:
\[
2^{B-1}\psi_B(A) \equiv \sum_{r=0}^{B-1} 2^{2r} \pmod{2A-5}
\;.
\]
Together, we obtain
\begin{equation}
3W2^{B-1}\psi_B(A) \equiv 3W \sum_{r=0}^{B-1} 2^{2r} \equiv (2^{2B}-1)W \pmod{2A-5}
\;.
\label{Eq:Sun4.6b}
\end{equation}

Subtracting equations \eqref{Eq:Sun4.6a} and \eqref{Eq:Sun4.6b}, we obtain \begin{equation}
(2^{2B}-1)W - 2^{B}(W^2 - 1)  \equiv 0 \pmod{2A-5}
\;.
\label{Eq:Sun4.6c}
\end{equation}

Next, we will bound the left-hand side by $2A - 5$ to show that it equals zero. To do so, consider first the case $|W| \geq 2$, such that
\[
|(2^{2B}-1)W| = 2^{2B}|W| - |W| < |A|^{\frac34} - 2 \leq |A| -2
\;,
\]
and similarly $2^B(W^2-1) < |A|-2$. Combining both bounds gives the required bound $|(2^{2B}-1)W - 2^B(W^2-1)| < 2|A|-5\le |2A-5|$.

Finally, the cases $|W| \leq 1$ all lead to a contradiction. If $W = 0$, then $2^B \equiv 0 \pmod{ 2A-5}$ which is absurd as $B >0$ and $|A| > 2^{4B}$ . Similarly, if $|W| = 1$, we obtain $2^{2B}-1 \equiv 0 \pmod{2A-5}$ which is absurd for the same reason.

Now we know that equation~\eqref{Eq:Sun4.6c} equals zero and we can factor 
\[
0 = 2^{B}(W^2 - 1) - (2^{2B}-1)W = (2^BW+1)(W-2^B)
\;.
\]

At last, observe that $|2^BW+1| \geq 2^B|W|-1\geq 3$ so we must have $W = 2^B$.
\end{proof}

%% file: section4.tex
In this section, we apply the results on Lucas sequences to prove our ``Bridge Theorem'' (Theorem~\ref{step:Sun4}), which provides a Diophantine expression for the power of two and the binomial coefficient that come from Theorem~\ref{step:polynomial-to-binomial}. Before that, we need to show a couple of lemmas and introduce some auxiliary variables to simplify the expressions. The reader is advised to keep a copy of this page for reference.

\begin{definition}[Defining several variables: \afpfact{Bridge_Theorem_Imp}{Bridge_Theorem_Imp.bridge_variables.U_def}]
\label{def:VariableDefinition}
Given variables $X,Y,b$, $g,h,k,l,w,x,y$ we define the following variables:
\begin{subequations}
\begin{align}
U &:= 2lXY \label{eq:Def_U} \\
    V &:= 4gwY \label{eq:Def_V} \\
A &:= U(V+1) \label{eq:Def_A} \\
B &:= 2X+1 \label{eq:Def_B} \\
C &:= B+(A-2)h \label{eq:Def_C} \\
D &:= (A^2-4)C^2+4 \label{eq:Def_D} \\
E &:= C^2Dx \label{eq:Def_E} \\
F &:= 4(A^2-4)E^2+1 \label{eq:Def_F} \\
G &:= 1+CDF-2(A+2)(A-2)^2E^2 \label{eq:Def_G} \\
H &:= C+BF+(2y-1)CF \label{eq:Def_H} \\
I &:= (G^2-1)H^2+1 \label{eq:Def_I} \\
J &:= X+1+k(U^2V-2) \label{eq:Def_J} 
\end{align}
\end{subequations}
\end{definition}

The main theorem of this section will refer to the following relations, using the variables introduced in Definition~\ref{def:VariableDefinition}:
\begin{subequations}
\begin{align}
&DFI \in \square \label{eq:DFI} \\
&(U^{4}V^2-4)J^2+4\in\square \label{eq:UVK} \\
&(2A-5) \,\big|\, (3bwC-2(b^2w^2-1)) \label{eq:pAppWW_new}\\
& \left(\frac{C}{J}-lY\right)^2 < \frac{1}{16g^2} \label{eq:g}
\\
&\left(\frac{C}{J}-lY\right)^2 < \frac{1}{4}
\label{eq:weakineq}
\;.
\end{align}
\end{subequations}

The following two lemmas are Diophantine encodings of Pell equations. One of the key steps in the solution of Hilbert's 10th problem is proving that exponential Diophantine equations are Diophantine. Below are suitably adapted versions of parts of this proof, minimizing the required number of variables.

\begin{lemma}[\afpfact{DFI_square_0}{DFI_square_0.bridge_variables.lemma_4_2_cor}] 
\label{lemma:Sun4.3} \sunref{\cite[Lemma 4.3]{Sun}}
Given integers $A\ge 4$, $B \ge 3$ and $A$ even, suppose that $C=\psi_B(A)$. Then there are arbitrarily large positive integers $x,y$ such that, defining $D, E, F, G, H, I$ as in \eqref{eq:Def_D}--\eqref{eq:Def_I}, we have $DFI\in\square$.
\end{lemma}

\begin{proof}
By definition \eqref{eq:Def_D} we have $D:=(A^2-4)C^2+4$. This is Pell's equation with $d=A^2-4$, so by Lemma~\ref{lemma:PellEquation}, all its solutions have the form $(\chi_n(A),\psi_n(A))$, and since $C=\psi_B(A)$ by hypothesis, we have  $D = (\chi_B(A))^2 \in\square$.

The law of apparition (Lemma~\ref{lemma:LucasElementary} \eqref{item:LucasApparitionRepetition}) for the Lucas sequence $\psi_n(A)$ says that for $N = 4C^2D$ there are infinitely many $q$ so that $x:=\psi_q(A)/4C^2D$ is a positive integer, and by the exponential growth law (Lemma~\ref{lemma:LucasElementary}~\eqref{item:LucasGrowth}) these can become arbitrarily large. By definitions~\eqref{eq:Def_F} and \eqref{eq:Def_E}, we have 
\begin{align}\label{eq:Fx}
F&=4(A^2-4)E^2+1= 4(A^2-4)C^4D^2x^2+1 \nonumber \\
&=4(A^2-4)(\psi_q(A)/4)^2+1= \left( (A^2-4)\psi_q(A)^2+4 \right)/4
\;.
\end{align}
hence $F = (\chi_q(A)/2)^2 \in\square$ by Pell's equation and Lemma~\ref{lemma:LucasElementary}~\eqref{item:ChiEven}.

Finally, by definition \eqref{eq:Def_G}, we set $G:=1+CDF-2(A+2)(A-2)^2E^2$ and further define $H_0:=\psi_B(2G)$ and $I_0:=(G^2-1)H_0^2+1$. Again by Pell's equation (Corollary~\ref{cor:PellEquation}) we have $I_0\in\square$. Note that $H=H_0$ implies $I=I_0$ and thus $DFI\in\square$ by \eqref{eq:Def_H} and
\eqref{eq:Def_I}. Hence, it remains to show that there exist arbitrarily large $y$ such that $H_0 = H := C+BF+(2y-1)CF$.

By Lemma \ref{lemma:LucasSequenceCongruence}, we have
\begin{equation}\label{eq:H-equiv-G}
	H_0 = \psi_B(2G) \equiv \psi_B(2)=B \pmod{ 2G-2}
\end{equation}
and 
\begin{equation}\label{eq:eq:H-equiv-C}
	H_0 = \psi_B(2G) \equiv \psi_B(A)=C \pmod{ 2G-A}
\;.
\end{equation}
Note that
\begin{align}\label{2GA}
    2G-A &= 2(1+CDF-2(A+2)(A-2)^2E^2)-A \nonumber \\
&= 2(1+CDF)-4(A^2-4)E^2(A-2)-A \nonumber
\\
&= 2(1+CDF) -(F-1)(A-2)-A 
\\
&= 2 + 2CDF - F(A-2) +A-2-A \nonumber \\
&= F(2CD-A+2) \nonumber
\;.
\label{eq:G relation}
\end{align}

Therefore, \eqref{eq:eq:H-equiv-C} implies that
\begin{equation}\label{eq:H-equiv-F}
H_0 \equiv C \equiv C+BF -CF \pmod F
\;.
\end{equation} 
Since $C|E$, note further that
\[
2G-2 = 2(1+CDF-2(A+2)(A-2)^2E^2-1) \equiv  0 \pmod {2C}
\;.
\]
Observe that $2|(C-B)$ by Lemma~\ref{lemma:LucasElementary} \eqref{item:LucasEven} and $C|(F-1)$, so using \eqref{eq:H-equiv-G} we have
\begin{equation}\label{eq:H-equiv-2C}
H_0 \equiv B \equiv B+(B-C)(F-1) = C + BF - CF \pmod{2C}
\;. 
\end{equation}

Clearly, $F$ and $E$ are coprime, and since $C|E$, also $F$ and $C$ are coprime. Thus \eqref{eq:H-equiv-F} and \eqref{eq:H-equiv-2C} imply
\[
H_0 \equiv C + BF -CF \pmod{2CF}
\]
and there is an integer $y$ so that $H_0 = H$ as defined in \eqref{eq:Def_H}.

We still need to show that there are arbitrarily large $y$. Note that $A$ and $B$ are given and $C$, $D$ depend only on these while $F$ is a positive, increasing polynomial as a function of $x$ by \eqref{eq:Fx}. So the same holds for $G$ by \eqref{2GA}. 
Finally by Lemma~\ref{lemma:LucasElementary} \eqref{item:LucasGrowth}, we have exponential growth of  $H=\psi_B(2G)$ as a function of $x$. Since $H=C+BF+(2y-1)CF$, the same holds for $y$.
\end{proof}

The previous lemma has the following converse, albeit under slightly different hypotheses, adapted to the context in which we will use it later.

\begin{lemma}[\afpfact{DFI_square_3}{DFI_square_3.bridge_variables.sun_theorem}]
\label{lemma:PellDiophantine2} \sunref{\cite[Lemma 4.4]{Sun}}
Let $A,B,C$ be integers so that  $1<B<|A|/2-1$ and $ (A-2)|(C-B)$. For integers $x,y$ with $x\neq 0$ define $D,F,I$ as in \eqref{eq:Def_D}--\eqref{eq:Def_I}. 

Then $DFI\in\square$ implies $C=\psi_B(A)$.
\end{lemma}

\begin{proof}
We have the hypothesis $|A-2|\ge |A|-2>2B>2$, hence $|A|>4$. If $C=0$, we cannot have $(A-2)|B$, so we must have $C\neq 0$. Likewise, if $\abs{C} = 1$, we still cannot have $(A-2)|(\pm 1 - B)$, so we must have $\abs{C} \geq 2$.

Thus, using \eqref{eq:Def_D} we have $D := (A^2-4)C^2+4\ge 4$. 
According to \eqref{eq:Def_E} we have $E=C^2Dx$ and so $E \neq 0$.

By definition in \eqref{eq:Def_F} we have $F := 4(A^2-4)E^2+1 $, and $D|E$ implies that $F\equiv 1\pmod D$. Similarly, the definitions in \cref{eq:Def_G,eq:Def_H,eq:Def_I} imply $D|(G^2-1)$ and $I\equiv 1\pmod D$. Thus $D$ is coprime to $FI$, so $D\in\square$ and $FI\in\square$. 

Let us show that $F$ and $I$ are also relatively prime, hence that $F\in\square$ and $I\in\square$. To this extent, suppose that there is some $d \ge 1$ such that $d|F$ and $d|I$; we want to show that $d=1$. By definition \eqref{eq:Def_H} we then have $H \equiv C \pmod d$. Thus, the definition of $I$ in \eqref{eq:Def_I} implies that 
\[
0\equiv I =(G^2-1)H^2+1\equiv (G^2-1)C^2+1 \pmod d
\;,
\]
hence 
\begin{equation}
d\,| \left((4G^2-4)C^2+4\right)
\;.
\label{eq:d divides 4GC}
\end{equation}
The definition $G := 1+CDF-2(A+2)(A-2)^2E^2$ from \eqref{eq:Def_G} then implies
\[
2G \equiv 2+2CDF-(F-1)(A-2) \equiv A \pmod d,
\]
and so \eqref{eq:d divides 4GC} becomes $d \,|\left((A^2-4)C^2+4\right)$, that is $d \, | \, D$. But since $D$ is coprime to $FI$, we must indeed have $d=1$, so all three of $D,F,I$ are squares separately. 

Now, because $D$ is defined in the form of a Pell equation, Lemma \ref{lemma:PellEquation} implies that $C=\psi_p(A)$ for some $p \in \Z$. It is left to show that $p=B$. Similarly, since $F$ and $I$ are squares, there are integers $m$ and $t$ such that $4E=\psi_m(A)$ and $H=\psi_t(2G)$.

Our hypothesis $(A-2)|(C-B)$ together with Corollary~\ref{cor:LucasSequenceCongruence} implies
\[
B \equiv C = \psi_p(A) \equiv p \pmod {A-2}.
\]

From this result, it follows that $\abs{p} \geq 2$ because the three cases $p = -1, 0, 1$ lead to a contradiction: they would imply $C = \psi_p(A) = -1, 0, 1$, respectively, by the chosen initial conditions of the Lucas sequence.

Next, we can establish that
\begin{equation}
    \label{eq:Bpequiv}
    B \equiv \pm p \pmod{2C}.
\end{equation}
First note that by definition of the variables, $F\equiv 1 \pmod {2C}$ and thus $H=BF+((2y-1)F+1)C\equiv B \pmod{2C}$. Moreover 
\[
2G=2+2CDF-4(A+2)(A-2)^2C^4D^2x^2\equiv 2 \pmod{2C}.
\]
Hence by Corollary~\ref{cor:LucasSequenceCongruence}, $H=\psi_t(2G)\equiv t \pmod {2C}$. Putting these equivalences together, we can reduce \eqref{eq:Bpequiv} to
\begin{equation}
    \label{eq:tsequiv}
    t \equiv \pm p \pmod {2C}.
\end{equation}
Since $C^2|4E$, Lemma~\ref{lemma:sun7} yields $\psi_{|p|}(A)|m$ and by symmetry $C|m$. So it suffices to show
\begin{equation}
    \label{eq:tsnequiv}
    t \equiv \pm p \pmod {2n}.
\end{equation}
where $n=|m|$.

Observe that $\psi_n(|A|) = 4C^2 D |x| \geq D > A^2 - 1 = \psi_3(|A|)$. This implies that $n > 3$ by the strict monotonicity property of Lemma~\ref{lemma:LucasElementary} \eqref{item:LucasMonotonicity}.

As a solution of the Pell equation, we have $\chi_m(A)^2 = 4F$ and thus $\chi_n(|A|)^2 = 4F$. Because $F$ is square, write $F=k^2$. By definition of the variables and \eqref{2GA}, we also find $H\equiv C \pmod {F}$ and $2G\equiv A\pmod {F}$. Thus using Lemma~\ref{lemma:LucasSequenceCongruence}, 
\[
\psi_t(A)\equiv \psi_t(2G)\equiv \psi_p(A) \pmod {F}.
\]
and also
\[
\psi_t(|A|) \equiv \psi_p(|A|) \pmod {k}.
\]
Everything is set up in order to apply Lemma~\ref{lemma:sun10int} to conclude equation~\eqref{eq:tsnequiv} and consequently equation~\eqref{eq:Bpequiv}.

\medskip

Before we can finish the proof, note that the initial value $\psi_1(A) = 1$ and the strict monotonicity property (Lemma~\ref{lemma:LucasElementary} \eqref{item:LucasMonotonicity}) imply that if $A > 1$ then $\psi_p(A) \geq p$.
Moreover, this implies, together with symmetry in both $A$ and $p$ (Lemmas~\ref{lemma:LucasElementary}~\eqref{item:LucasSymmetryA}~and~\eqref{item:LucasNegative}),
that if $\abs{A}>1$ then $\abs{\psi_p(A)} = \psi_{\,\abs{p}}(\,\abs{A}) \geq \abs{p} > 2$.

We will first show that $B = -p$ is absurd. Recall that $B \equiv p \pmod{A-2}$ so assuming $B = -p$, we get that $-2p \equiv 0 \pmod{A-2}$. In that case, because $p\neq 0$, we have $\abs{A-2} \leq 2\abs{p} = 2B$. This directly contradicts $\abs{A-2} > 2B$ established earlier.

In order to now conclude, suppose that there is a non-zero integer $z \in \mathbb{Z}$ such that $B \mp p = 2Cz$. Then $\abs{z} \geq 1$ and so the triangle inequality implies
\begin{align*}
\abs{B} &\geq \abs{B \mp p} - \abs{p} \\
&\geq 2 \abs{C} \abs{z} - \abs{p} \\
&\geq 2 \abs{C} - \abs{p} \\
&= \abs{C} + \psi_{\,\abs{p}}(\,\abs{A}) - \abs{p} \\
&\geq \abs{C}.
\end{align*}
Finally, monotonicity and $\abs{p} \geq 2$ imply that $\abs{C} = \psi_{\,\abs{p}}(\,\abs{A}) \geq \psi_2(\,\abs{A}) = \abs{A}$ so we find
\[
\abs{B} = B \geq \abs{A},
\]
which is absurd by our original assumption. This shows that $B = p$. Indeed, we have shown $C = \psi_B(A)$, concluding the proof.
\end{proof}

\begin{FormalizationNote}\label{formalization:lost-proof}
A concrete example of the deficiency of paper-based proofs is the history of this Lemma~\ref{lemma:PellDiophantine2}. Although this is a key technical step used in \cref{step:Sun4}, it is the most recent lemma added to the manuscript. As a converse statement to the preceding \cref{lemma:Sun4.3}, this lemma was intended for an appendix but had been forgotten for years. While the statement was correct, our first pen-and-paper proof contained false reasoning. Without noticing the wrong proof in a comment in our \LaTeX{} source, the team responsible for its formalization had rediscovered a correct proof on their own.

The subtle but severe mistake in our typeset proof was discovered only during the final review phase of the manuscript. Fortunately, we were able to refer to our already complete Isabelle formalization. Delving into the formal statement and all its dependencies then revealed close to 2000 lines of formal proof (!), including many auxiliary results missing from our manuscript. Fixing this omission in our manuscript added four extra pages, including the four \cref{lemma:psi-linear-expansion,lemma:sun7,lemma:psiInjective,lemma:sun10int}. Thanks to the structured proof language of Isabelle, even those of us who had never actively formalized were able to easily read the formal proof and write a corresponding typeset version.
\end{FormalizationNote}

\Newpage

Finally, we will set up two more technical lemmas for the proof of Theorem~\ref{step:Sun4}.

\begin{lemma}[\afpfact{Lemma_4_4}{Lemma_4_4.lemma_4_4_cor_rho_abs}]
\label{lemma:CKlemmaNew}
Let $U, V, X$ be integers. If $V,X\ge 1$ and $|U|\ge 2X$, then the rational number
\begin{equation}
\rho:=\frac{(V+1)^{2X}}{V^X}
\label{eq:Rho_new}
\end{equation}
satisfies
\[
\left| \frac{\psi_{2X+1}(U(V+1))}{\psi_{X+1}(U^2V)}-\rho\right|
\le \rho \frac{2X}{|U|V}
\;.
\]

\end{lemma}

\begin{proof}
Using the symmetry of Lucas sequences established in Lemma~\ref{lemma:LucasElementary} \eqref{item:LucasSymmetryA}, we have $\psi_{2X+1}(U(V+1)) = \psi_{2X+1}(|U|(V+1))$.
It now follows from the exponential growth property Lemma~\ref{lemma:LucasElementary} \eqref{item:LucasGrowth} that
\[
\frac{(|U|(V+1)-1)^{2X}}{(U^2V)^X}
\le
\frac{\psi_{2X+1}(|U|(V+1))}{\psi_{X+1}(U^2V)}
\leq \frac{(|U|(V+1))^{2X}}{(U^2V-1)^X}
\;.
\]
Consider first the lower bound. By the Bernoulli inequality,
\[
	\rho\left(1-\frac{2X}{|U|(V+1)}\right)\leq \rho\left(1-\frac{1}{|U|(V+1)}\right)^{2X}=\frac{(|U|(V+1)-1)^{2X}}{(U^2V)^X}
\;.
\]
Combining the two lines above and subtracting $\rho$ on both sides yields the required lower bound
\[
-\rho\frac{2X}{|U|V}\le -\rho \frac{2X}{|U|(V+1)}\leq \frac{\psi_{2X+1}(U(V+1))}{\psi_{X+1}(U^2V)}-\rho
\;.
\]

\medskip

For the upper bound, the Bernoulli inequality and our hypotheses give
\[
\left(1+\frac{2X}{U^2V}\right)\left( 1-\frac{1}{U^2V}\right)^{X}
\ge 
\left(1+\frac{2X}{U^2V}\right)\left( 1-\frac{X}{U^2V}\right)
=1+\frac{X}{U^2V}\left(1-\frac{2X}{U^2V}\right)
>1
\;,
\]
which implies
\[
\left( 1-\frac{1}{U^2V}\right)^{-X}-1 \le \frac{2X}{U^2V}
\;.
\]
Finally, the calculation
\[
\frac{(U(V+1))^{2X}}{(U^2V-1)^X}
=\frac{(V+1)^{2X}}{V^X}\left(\frac{U^2V}{U^2V-1}\right)^X=\rho \left( 1-\frac{1}{U^2V}\right)^{-X}
\;
\]
allows us to conclude after again subtracting $\rho$ on both sides.
\end{proof}

\begin{lemma}[\afpfact{Bridge_Theorem_Rev}{Bridge_Theorem_Rev.bridge_variables.theorem_II_2_1}]
Given integers $X,Y,b,g,h,l,w$, define the variables $U,V,A,B,C$ as in Definition~\ref{def:VariableDefinition}. Then $|Y|\ge 2$ and \eqref{eq:pAppWW_new} imply $bw\neq 0$. 
\label{lemma:bwneq0}
\end{lemma}
\begin{proof}
    Suppose $bw=0$. Then \eqref{eq:pAppWW_new} reduces to $(2A-5)|2$. Thus $2A-5=\pm 1$, hence $A\in\{2,3\}$. However, using \eqref{eq:Def_A} and \eqref{eq:Def_U} we have $A=U(V+1)=2lYX(V+1)$, and since $|Y|\ge 2$ by hypothesis this is a contradiction.
\end{proof}

\Newpage

\subsection{From Exponential to Ordinary Diophantine Equations}

Although it is well known that the exponential term and the divisibility of the binomial coefficient from Theorem~\ref{step:polynomial-to-binomial} are Diophantine expressions, the elegance of the theorem below lies in jointly encoding them in order to save on the number of unknowns. 

\begin{mainstep}[Bridge Theorem: \afpfact{Bridge_Theorem}{Bridge_Theorem.bridge_variables.theorem_II}]
\label{step:Sun4} \sunref{\cite[Theorems 4.1 and 4.2]{Sun}}
Given integers $b\ge 0$, $X\geq 3b$, $Y\geq \max \{ b,2^8 \}$ and $g\ge 1$, the following are equivalent:
\begin{enumerate}
\item 
\label{proofitem:Binom}
 $b\in 2\uparrow$ and $Y | \binom{2 X}{X}$;

\item 
\label{proofitem:Dioph2} $\exists h,k,l,w,x,y\geq b$ such that the Diophantine conditions \eqref{eq:DFI}, \eqref{eq:UVK}, \eqref{eq:pAppWW_new} and \eqref{eq:g} hold;

\addtocounter{enumi}{-1}
\renewcommand{\theenumi}{\arabic{enumi}a}
\renewcommand{\labelenumi}{(\theenumi)}
\item
\label{proofitem:Dioph1} $\exists h,k,l,w,x,y\in\Z$ with $lx\neq 0$ such that the Diophantine conditions \eqref{eq:DFI}, \eqref{eq:UVK}, \eqref{eq:pAppWW_new} and \eqref{eq:weakineq} hold.
\end{enumerate}
Moreover, the implication \eqref{proofitem:Binom} $\implies$ \eqref{proofitem:Dioph2} holds already if $X\ge b$ (rather than $X\ge 3b$).
\end{mainstep}

\begin{proof}
The main idea is to translate the exponential condition $b\in 2\uparrow$ into an expression involving Lucas sequences using Lemma~\ref{lemma:LucasCongruenceNEW}. Then, based on Pell's equation \eqref{Eq:PellEquation}, the Lucas sequence can be expressed as a Diophantine expression \eqref{eq:DFI} using Lemma~\ref{lemma:Sun4.3} and Lemma~\ref{lemma:PellDiophantine2}.

\needspace{3\baselineskip}
\proofstep{The implication \eqref{proofitem:Binom}$\implies$\eqref{proofitem:Dioph2}}{ (assuming only $X\ge b$).} 

\proofsubstep{Constructing $h,k,l,w,x,y$ and showing a common lower bound.}{} 
We first construct the variable $w$. We estimate, using \eqref{eq:Def_B} 
\[
2^B=2^{2X+1}> 4^X\geq X^2\geq b^2\geq b
\;.
\]
Since $b\in 2\uparrow$ by hypothesis, $w:=2^{B}/b$ is an integer and $w\geq b$. 

 Next we construct $l$. A central element of the proof is the ratio $\rho \in \Q$ that already appeared in \eqref{eq:Rho_new}. Note that $V=4gwY$ (defined in \eqref{eq:Def_V}) is positive because $Y$ and $g$ are by assumption. We then expand using the binomial theorem
\begin{align}
\rho
&:=\frac{(V+1)^{2X}}{V^X}=\frac{1}{V^X}\cdot \sum_{i=0}^{2X}\binom{2X}{i}V^i\\&=\frac{1}{V}\sum_{i=0}^{X-1}\binom{2X}{i}\frac{1}{V^{X-1-i}}+\binom{2X}{X}+V\sum_{i=X+1}^{2X}\binom{2X}{i}V^{i-X-1}
\;.
\label{eq:RhoDecomposition}
\end{align}
The second and third terms are positive integers while the first one is in $(0,1)$:
\begin{equation}
\begin{aligned}
&\frac{1}{V}\sum_{i=0}^{X-1}\binom{2X}{i}\frac{1}{V^{X-1-i}}<\frac 1 V \sum_{i=0}^{2X}\binom{2X}{i} = \frac{2^{2X}}{4gwY} \stackrel{(*)}{\le}\frac{1}{8g}<1
\end{aligned}
\label{eq:fracrhobound}
\end{equation}
where in (*) we used $V=4gwY\ge 4gwb=4g2^B=4g2^{2X+1}$. 

We can thus separate the integer and fractional parts of $\rho$ as follows:
\begin{equation}\label{rhoineq}
\lfloor \rho \rfloor = \binom{2X}{X}+V\sum_{i=X+1}^{2X}\binom{2X}{i}V^{i-X-1}\geq V
\end{equation}
and the fractional part $\{\rho\}$ is equal to the first summand in \eqref{eq:RhoDecomposition}.

Since $Y$ divides $\binom{2X}{X}$ (by hypothesis) and $V=4gwY$, we have an integer
\begin{equation}
l:=\frac{\lfloor \rho \rfloor}{Y}\geq\frac{V}{Y}=4gw\geq w \geq b
\;.
\label{eq:rho_floor}
\end{equation}

The integer $h$ is derived from Corollary~\ref{cor:LucasSequenceCongruence}, applied to $\alpha=A$ and $m=B$: we have $\psi_B(A)\equiv B\pmod{A-2}$, so there is a unique $h$ such that 
\begin{equation}
\psi_B(A)=B+(A-2)h
\;.
\label{Eq:PsiB=}
\end{equation}
Noticing that $A=U(V+1)=2lYX(V+1)\ge 4$ and $B=2X+1\ge 3$, we use Lemma~\ref{lemma:LucasElementary} \eqref{item:LucasGrowth} to get 
\begin{align}\psi_B(A) &\ge (A-1)^{B-1} = (1+A-2)^{B-1} \nonumber \\ &\ge 1+(B-1)(A-2) + \frac{(B-1)(B-2)}{2}(A-2)^2 \ge B + (A-2)^2
\;.
\end{align}
Combining the previous two equations, we obtain $h\ge A-2\ge V>w\ge b$ as required.

Combining \eqref{Eq:PsiB=} with \eqref{eq:Def_C}, we conclude $C=\psi_B(A)$. Since $A$ is clearly even, we can then use Lemma~\ref{lemma:Sun4.3} to obtain integers $x,y\ge b$.
Similarly for $k$, we apply Lemma~\ref{lemma:LucasSequenceCongruence} for $\alpha=U^2V\ge 4$ and $m=X+1\ge 2$, followed by Lemma~\ref{lemma:LucasElementary} \eqref{item:LucasCoeff2}, to obtain
\begin{equation*}
	\psi_{X+1}(U^2V)\equiv \psi_{X+1}(2)=X+1 \pmod {(U^2V-2)}
\;,
\end{equation*}
so $k$ is the integer satisfying
\begin{equation}\label{eq:k}
	\psi_{X+1}(U^2V)=X+1+(U^2V-2)k
\;,
\end{equation}
which is just $J$. To show that $k\ge b$, we distinguish the cases $X=1$ and $X>1$. In the first case, the hypotheses $X\ge b$ and $b\ge 1$ imply $b=1$. Hence \eqref{eq:k} becomes
\[
U^2V= \psi_2(U^2V)= 	2+(U^2V-2)k
\]
which implies $k=1=b$. 

If $X\ge 2$, again by Lemma~\ref{lemma:LucasElementary}~\eqref{item:LucasGrowth}
\begin{align*}
    J&=\psi_{X+1}(U^2V)
    \geq (U^2V-1)^X\\&=(1+(U^2V-2))^X 
    \ge 1+X+(U^2V-2)^2.
\end{align*}

Comparing this with \eqref{eq:k}, we obtain $k\ge U^2V-2\ge 4b-2\ge b$.

\proofsubstep{Deriving relations \eqref{eq:DFI} to \eqref{eq:g}.}{} Equation \eqref{eq:DFI} immediately holds by construction and Lemma~\ref{lemma:Sun4.3}. We continue with \eqref{eq:UVK}, which claims that $(U^{4}V^2-4)J^2+4$ is a square. This is Pell's equation, and the claim holds because of Lemma~\ref{lemma:PellEquation} and \eqref{eq:k}.

To establish \eqref{eq:pAppWW_new} we use Lemma~\ref{lemma:LucasCongruenceNEW} with $W=2^B=bw$, so $C=\psi_B(A)$ satisfies
\[
3bwC\equiv 2(b^2w^2-1) \pmod{2A-5}
\;,
\]
and the difference between both sides is \eqref{eq:pAppWW_new}.

Left is \eqref{eq:g}. By \eqref{rhoineq} and \eqref{eq:rho_floor} we have
\begin{equation}
	U^2V\geq U = 2 \,l\, Y X = 2 \lfloor \rho \rfloor X \geq 2VX \geq 2X
\;.
\end{equation}
Since $C=\psi_B(A)=\psi_{2X+1}(U(V+1))$, it now follows from Lemma~\ref{lemma:CKlemmaNew} that
\[
\left| \frac{C}{J}-\rho \right| \leq \frac{2X}{UV}\rho = \frac{\rho}{lVY}=\frac{\rho}{\lfloor \rho \rfloor} \cdot \frac{1}{V}\leq \frac{2}{V}
\;.
\]
Since $V/4g=wY\geq wb=2^B= 2^{2X+1}\geq 8$ it follows that 

\[
	\left| \frac{C}{J}-\rho \right| \leq \frac{1}{16g}
\;.
\]
Applying the triangle inequality and \eqref{eq:fracrhobound}, and remembering that $lY=\lfloor\rho\rfloor$, we obtain

\[
\left| \frac{C}{J}-lY \right|\le  \left| \frac{C}{J}-\rho \right|+\left| \rho -\lfloor \rho \rfloor \rule{0pt}{11pt} \right| < \frac{1}{16g}+\frac{1}{8g}<\frac{1}{4g}
\;.
\]
Squaring yields the desired statement and concludes the first part of the proof. 

\bigskip

\proofstep{The implication \eqref{proofitem:Dioph2} $\implies$ \eqref{proofitem:Dioph1}.}{}

By Lemma~\ref{lemma:bwneq0}, we have $b\neq 0$. Hence, $h,k,l,w,x,y\ge b\ge 1$. Since also \eqref{eq:g}$\implies$\eqref{eq:weakineq}, this direction is clear.

\proofstep{The implication \eqref{proofitem:Dioph1} $\implies$ \eqref{proofitem:Binom}.}{}

Given the variables $X,Y,h,k,l,x,w,y$, the variables $D,F,I,\dots,J$ are introduced in Definition~\ref{def:VariableDefinition}, \eqref{eq:Def_D}--\eqref{eq:Def_J}, and our hypothesis is that they satisfy the simple Diophantine expressions \eqref{eq:DFI}, \eqref{eq:UVK}, \eqref{eq:pAppWW_new} and \eqref{eq:weakineq}. The first step is to derive quantities equal to some Lucas sequences, from there we derive exponential and binomial expressions. 

\proofsubstep{Deriving Lucas sequences.}{} Since $Y\ge 2^8$ by assumption, we have $bw\neq 0$ by Lemma~\ref{lemma:bwneq0}. So $b\ge 1$ and $X\geq3b\geq3$. Since $Y\ge 2^8$ and $2^{10}|V$ by \eqref{eq:Def_V} we obtain a lower bound for $A$:
 \[
 |A|=|U(V+1)|\geq 2^{10}|U|=2^{11}|l|XY\geq 2^{11}XY \geq 2^{19}
\;.
 \]
Therefore $|A|/2-1>2X+1=B$ (using \eqref{eq:Def_B}), 
and $X\ge 3$ implies $B\ge 7$.
By definition of $C := B+(A-2)h$ in \eqref{eq:Def_C}, 
we have $(A-2)|(C-B)$. Now everything is set up to apply Lemma~\ref{lemma:PellDiophantine2} to obtain 
\begin{equation}
C=\psi_B(A)
\label{eq:Variable_C}
\end{equation}
(note in particular that $DFI \in \square$ is our hypothesis \eqref{eq:DFI}).

Since $|V+1|\le |V|+1 \le 2|V|-1 \le |U^2V|-1$, we obtain, using Lemma~\ref{lemma:LucasElementary} \eqref{item:LucasGrowth}:
 \begin{equation}\label{423}
 |C|=\psi_B(|A|)\leq |A|^{B-1}=|U(V+1)|^{2X}\leq U^{2X}(|U^2V|-1)^{2X}
\;.
 \end{equation}
 
We have $J = X+1+k(U^2V-2)$ from \eqref{eq:Def_J}. 
Hypothesis \eqref{eq:UVK} is $(U^{4}V^2-4)J^2+4\in\square$. By Pell's equation (Lemma~\ref{lemma:PellEquation})  this implies $J=\psi_R(U^2V)$ for some $R\in \Z$ (or $|J|$ for some $R\in\N$). 
Since clearly $U^2V>2$, we obtain the following congruence relation, using Lemma~\ref{lemma:LucasSequenceCongruence} and Lemma~\ref{lemma:LucasElementary} \eqref{item:LucasCoeff2}:
\[
X+1 \equiv J = \psi_R(U^2V)\equiv \psi_R(2)=R \quad (\bmod \, (U^2V-2))
\;.
\]
Thus we may write $R=X+1+r(U^2V-2)$ for some $r\in \Z$.

Our next claim is $r=0$. Indeed, if $r\neq 0$, then
\[
|R|\ge |r||U^2V-2|-|X+1|\ge|U^2V|-2-X-1 > 4X+4-2-X-1>3X
\]
and using Lemma~\ref{lemma:LucasElementary} \eqref{item:LucasMonotonicity}  and \eqref{item:LucasNegative} 
\[
|J|=|\psi_R(U^2V)|=\psi_{|R|}(|U^2V|)\ge(|U^2V|-1)^{|R|-1}\ge (|U^2V|-1)^{3X}.
\]
Since $|U^2V|\ge 4U^2>2U^2+1$, combining this with \eqref{423} yields
\[
\left|\frac{C}{J}\right|\le\left(\frac{U^2}{|U^2V|-1}\right)^X<\left(\frac{1}{2}\right)^X\le \frac{1}{2}
\;.
\]
But together with \eqref{eq:weakineq} this implies
\[
|lY|\leq \left|lY-\frac{C}{J}\right|+\left|\frac{C}{J}\right|<\frac{1}{2}+\frac{1}{2}=1
\]
which contradicts $lY\neq 0$. Therefore, we indeed have $r=0$ and $R=X+1$, hence 
\begin{equation}
J=\psi_{X+1}(U^2V)
\label{eq:Variable_K}
\;.
\end{equation}

\medskip

\proofsubstep{Showing that $b$ is a power of $2$.}{} By \eqref{eq:weakineq} we have

\begin{equation}
\left|\frac CJ\right|\le \left| \frac CJ -lY\right|+|lY|<\frac 12+|lY|.
\label{eq:cj12}
\end{equation}
We can bound $|C/J|$ as follows:
\[
    \left|\frac{C}{J}\right| = \frac{\psi_{2X+1}(|A|)}{\psi_X(|U^2V|)} 
    \ge \frac{(|A|-1)^{2X}}{|U|^{2X}|V|^{X+1}} 
    \ge \frac{|U|^{2X}\left(|V+1|-\left|\frac1U\right|\right)^{2X}}{|U|^{2X}|V|^X} \ge \frac{(|V|-2)^{2X}}{|V|^X}.
\]
Note that in the case $V > 0$, the same reasoning yields
\begin{equation}
    \left|\frac{C}{J}\right| \ge V^X
    \label{eq:rho and V and L}
\end{equation}

Recall that $|V| = 4g|bw|Y \ge 2^{10}$. We can thus easily check that $(2|V|-4)^2 \ge |V|(|V-1|)$, meaning that
\[
\frac{(|V|-2)^{2X}}{|V|^X} \ge \left(\frac{|V|-1}{2}\right)^X
\]
Together with \eqref{eq:cj12} this yields
\[
|l|Y \ge \left(\frac{|V|-1}{2}\right)^X - \frac12 \ge \frac12 \left(\frac{|V|-1}{2}\right)^X.
\]
Now, notice that:
\begin{equation}\label{430}
    |A| \ge |U|(|V|-1) \ge 2 |l| XY (|V|-1) \ge \left(\frac{|V|-1}{2}\right)^{X+1}\\
\end{equation}

This together with $(|V|-1)/2>2^8$ then yields $|A|\ge|bw|^{X+1}\ge (bw)^4$ (since $X\geq 3$) and $|A| \ge 2^{8(X+1)} \ge 2^{4B}$. As $C=\psi_B(A)$ together with \eqref{eq:pAppWW_new} we have enough to apply Lemma~\ref{lemma:LucasCongruenceNEW} in the backward direction with $W=bw$ to obtain $bw=2^{2X+1}$.

\proofsubstep{Divisibility of the binomial coefficient.}{}
Once again we will look at $\rho$ as in \eqref{eq:RhoDecomposition}
\[
\rho=\{\rho\}+\lfloor\rho\rfloor= \frac{1}{V}\sum_{i=0}^{X-1}\binom{2X}{i}\frac{1}{V^{X-1-i}}+\binom{2X}{X}+V\sum_{i=X+1}^{2X}\binom{2X}{i}V^{i-X-1}
\;.
\]
From $bw=2^B$ we conclude
\begin{equation}
V=4gwY\ge 4gwb \ge 4 \cdot 2^{2X+1}= 8 \cdot 2^{2X}\ge 2^{9}
\label{eq:V is power of 2}
\end{equation}
and hence
\begin{equation}
\{\rho\}=\frac{1}{V}\sum_{i=0}^{X-1}\binom{2X}{i}\frac{1}{V^{X-1-i}}<\frac{1}{V}\sum_{i=0}^{2X}\binom{2X}{i}=\frac{2^{2X}}{V}\le \frac{1}{8}
\;.
\label{eq:V8}
\end{equation}
Since $Y$ divides $V$, we can establish our final claim $Y|\binom{2X}{X}$ by proving $Y|\lfloor \rho\rfloor$.

We start with the following identity, remembering the definition of $\rho$
\[
\frac{\rho}{|l|Y} = \frac{(V+1)^{2X}}{V^{2X}} \frac{V^X}{|l|Y+1/2}\left(1+\frac{1}{2|l|Y}\right).
\]

Combining \eqref{eq:cj12} with \eqref{eq:rho and V and L} yields $0< V^X \le |l|Y+1/2$. Together with \eqref{eq:V is power of 2} we then obtain the bound
\begin{align*}
\frac{\rho}{|l|Y}&\le \left(1+\frac 1V\right)^{2X}\frac{V^X}{V^X} \left(1+\frac{1}{2V^X-1}\right)\\
&\le (1+2^{-2X})^{2X}\left(1+\frac{1}{2^{10}-1}\right)\\
&\le (1+2^{-6})^6 \left(1+\frac{1}{2^{10}-1}\right) < 1.1.
\end{align*}
In the third line we used $X\ge 3$ and that the function $(1+2^{-n})^n$ decreases for $n\ge 2$. Thus,
\begin{equation}
	\left| \frac{C}{J}-\rho \right|\le \rho \frac{2X}{|U|V}=\frac{\rho}{|l|Y}\cdot \frac{1}{V} < \frac{1.1}{V}\le \frac{1.1}{8\cdot 2^{2X}} <2^{-8}.
\label{eq:cjrho}
\end{equation}
where the first inequality is Lemma~\ref{lemma:CKlemmaNew}, then $U=2lXY$ from \eqref{eq:Def_U}, and once again \eqref{eq:V is power of 2}.

Now we use the triangle inequality, then \eqref{eq:V8}, \eqref{eq:cjrho} and \eqref{eq:weakineq} to estimate
\[
|\lfloor\rho\rfloor-lY|\le |\lfloor\rho\rfloor-\rho|+\left| \frac{C}{J}-\rho \right| + \left| \frac{C}{J}-lY \right| <\frac{1}{8}+2^{-8}+\frac{1}{2}<1.
\]
We conclude that $\lfloor\rho\rfloor=lY$ since both are integers. This completes the proof of the theorem.
\end{proof}

%% file: section5.tex
The idea of relation-combining is the following: imagine we want to check the two conditions $x > 0, y \ne 0$ with Diophantine polynomials in integer variables. A naive method would use Lagrange's four squares theorem for each condition, i.e.
\begin{align*}
x &= a^2+b^2+c^2+d^2+1 \\
y^2 &= e^2+f^2+g^2+h^2+1
\end{align*}
for some $a,b,c,d,e,f,g,h \in \Z$.

However, there is a more efficient method because the two conditions together are equivalent to the single condition
\[
xy^2 = a^2+b^2+c^2+d^2+1
\;,
\]
using only $a, b, c, d \in \Z$, cutting the number of integer variables required in half.

The following lemma
allows us to simultaneously check a divisibility condition, an inequality, and integers being perfect squares with only one (!) non-negative integer variable.

\begin{lemma}[\afpfact{Matiyasevich_Polynomial}{Matiyasevich_Polynomial.matiyasevich_polynomial.relation_combining}]\label{lemma:MPoly}
For every $q>0$ there is a polynomial $M_q$ with the property that for all integers $A_1,\dots,A_q,R,S,T$ such that $S\neq 0$ the following are equivalent:
\begin{enumerate}
\item $S|T$, $R>0$ and $A_1,\dots,A_q$ are squares,
\item $\exists n \ge 0 : \, M_q(A_1,\dots,A_q,S,T,R,n)=0$.
\end{enumerate}
Specifically, $M_q$ can be given as	
\begin{equation} \label{eq:matiyasevich-polynomial} \begin{split}
&M_q(A_1,\dots,A_q,S,T,R,n)\\
&=\prod_{\varepsilon_1,\dots,\varepsilon_q\in\{\pm1\}^q}\left(S^2n + T^2 - S^2(2R-1)\left(T^2 + X_q^q +\sum_{j=1}^q \eps_j \sqrt{A_j}X_q^{j-1}\right) \right) \\
\end{split} \end{equation}
where $X_q:=1+\sum_{j=1}^{q}A_j^2$.
\end{lemma}

This lemma appears in Matiyasevich-Robinson~\cite[Theorem III]{MR75}, so we do not reproduce a proof here. However, we did formalize this result.

\begin{FormalizationNote}\label{formalization:exists-n}
We simplified \cref{lemma:MPoly} which previously gave a complicated expression for $n$. Once we noticed that this expression for $n$ had not been formalized at all, we realized that the existence of some $n \geq 0$ would be sufficient for what we needed to prove, and adapted the lemma accordingly.
\end{FormalizationNote}

\subsection{Construction of a Polynomial with 9 Unknowns in \texorpdfstring{$\Z$ and $\N$}{Z and N}}
We can now construct the polynomial which will be the basis for our universal pair. In the first step we define a polynomial $Q$ with 9 unknowns of which 8 are integers and one is a natural number. Section~\ref{Sec:UniversalPairs} will then transform the natural number unknown into integer unknowns; yielding the desired universal pair in the integers. 

Constructing the polynomial $Q$ requires combining the various previous definitions of polynomials such as \cref{def:codingVariables}, \cref{def:VariableDefinition} and $M_q$ from Lemma~\ref{lemma:MPoly}.

\begin{definition}[The nine-unknown polynomial $Q$: \afpfact{Nine_Unknowns_N_Z_Definitions}{Nine_Unknowns_N_Z_Definitions.combined_variables.Q_def}]\label{def:nine-unknown-polynomial}
Fix a polynomial $P \in \Z[a, z_1, \ldots, z_\nu]$. We construct a polynomial $Q \in \Z [a, f, g, h, k, l, w, x, y, n]$ with nine unknowns in several steps. 

First, let
\[
\overline{P} = P^2 + (z_{\nu + 1} - 1)^2.
\]
and define $\bLowercase$, $\X$, and $\Y$, using Definition~\ref{def:codingVariables} based on $\overline{P}$ instead of $P$. The quantities $\bLowercase, \X, \Y$ are polynomials in $a$, $f$, and $g$.
Moreover, we define $A, C, D, F, I, J, U, V$ as in Definition~\ref{def:VariableDefinition}. These are polynomials in $b$, $X$,  $Y$, and the remaining input variables of $Q$; for instance $A=A(a,f,g,h,k,l,w,x,y,n,b,X,Y)$, though not all variables occur in all polynomials. For the three last arguments $b,X,Y$ we substitute the polynomials $\bLowercase$, $\X$ and $\Y$ defined above, so that eventually $A, C, D, F, I, J, U, V$ only depend on $a,f,g,h,k,l,w,x,y,n$. 

Based on these, we define the following quantities. They are the main constituents of the relations \eqref{eq:DFI}--\eqref{eq:weakineq} which we will need to check.
\begin{equation} \begin{split}
A_1 &:= \bLowercase \\
A_2 &:= DFI \\
A_3 &:= (U^4V^2-4)J^2+4 \\
S &:= 2A-5 \\
T &:= 3\bLowercase wC-2(\bLowercase^2w^2-1)
\end{split}
\label{eq:A123ST_def}
\end{equation}
We set $\mu := \gamma \bLowercase^\alpha$ (with $\alpha, \gamma$ from Definition~\ref{def:codingVariables}). We can finally define
\begin{equation} \label{eq:R}
R:=f^2l^2x^2\left(8 \mu^3gJ^2-g^2\left(32(C-lJ\Y)^2 \mu^3+g^2J^2\rule{0pt}{10pt}\right)\rule{0pt}{11pt}\right).
\end{equation}
Now we use the polynomial $M_q$ from Lemma~\ref{lemma:MPoly} for $q=3$ and define
\begin{equation}
Q:=M_3(A_1, A_2, A_3, S, T, R, n)
\;.
\label{eq:PolyComposition}
\end{equation}
\end{definition}

Note that $A_1$, $A_2$, $A_3$, $S$, $T$ and $R$, and thus $Q$, are polynomials in $a$ and nine unknowns $f,g,h,k,l,w,x,y,n$ as desired.

We will consider the unknowns $f, g, h, k, w, l, x, y$ to be integers and the unknown $n$ to be a natural number. It then holds that $P$ has a solution for parameter $a$ if and only if $Q$ has a solution for parameter $a$. This is the content of the following statement. 

\begin{mainstep}[Nine unknowns over $\Z$ and $\N$: \afpfact{Nine_Unknowns_N_Z}{Nine_Unknowns_N_Z.nine_unknowns_over_Z_and_N}] 
\label{thm:mainthm} \sunref{\cite[Theorem 1.1]{Sun}}
Let $P \in \Z[a, z_1, \ldots, z_\nu]$ and construct $Q$ based on $P$ as in Definition~\ref{def:nine-unknown-polynomial}. Then for all $a \in \N$ the following are equivalent: 
\begin{equation}\label{eq:p-a-z-zero}
\exists \mathbf{z} \in \N^\nu \colon P(a, \mathbf{z}) = 0
\end{equation}
\begin{equation}\label{eq:9-variable-polynomial}
\exists f, g, h, k, l, w, x, y \in\Z, \exists n \in \N \colon Q(a,f, g, h, k, l, w, x, y, n)=0
\;.
\end{equation}
\end{mainstep}

\begin{proof}
Throughout this proof we will make reference to the various quantities that appear in the construction of $Q$. We will do so without referencing \cref{def:nine-unknown-polynomial} explicitly every time. 

The construction of $Q$ begins with defining a polynomial $\overline{P}$. The motivation for this is that in order to apply Theorem~\ref{step:polynomial-to-binomial} we need a polynomial that is suitable for coding. Lemma~\ref{lemma:p-suitable-for-coding} ensures that $\overline{P}$ has this property. 

Furthermore, observe that $\overline{P}$ has a solution for parameter $a$ if and only if $P$ has a solution for $a$ (this also follows from \cref{lemma:p-suitable-for-coding}). In other words, equation \eqref{eq:p-a-z-zero} is equivalent to 
\begin{equation} \label{eq:p-overline-zero} \tag{\ref{eq:p-a-z-zero}a}
\exists \mathbf{z} \in \N^{\nu + 1}: \overline{P}(a, \mathbf{z}) = 0	 \; .
\end{equation}
The proof of this theorem can thus proceed by showing two separate implications $\eqref{eq:p-overline-zero} \implies \eqref{eq:9-variable-polynomial}$ and $\eqref{eq:9-variable-polynomial} \implies \eqref{eq:p-overline-zero}$. 

\proofstep{Forward direction: proof that $\eqref{eq:p-overline-zero} \implies \eqref{eq:9-variable-polynomial}$}{}.
Let $z_1, \ldots, z_{\nu+1}\in\N$ be such that $\overline{P}(a, z_1, \ldots, z_{\nu+1}) = 0$. We need to show the existence of $f,g,h,k,l,w,x,y,n$ such that $Q(a,f,g,h,k,l,w,x,y,n)=0$ for the same $a$.

By Lemma~\ref{lemma:arbitraryF} there are arbitrarily large $f$ such that $\bLowercase := 1 + 3 (2a +1) f$ satisfies
\begin{equation} \label{eq:bProperties}
 \bLowercase \in \square\; \qquad \text{and} \quad \bLowercase~ \in 2 \uparrow
\;.
\end{equation}
Choose $f > 0$ large enough so that the resulting $\bLowercase$ satisfies $\bLowercase >\max\{z_1,\ldots,z_{\nu+1}\}$.

Therefore, Theorem~\ref{step:polynomial-to-binomial} can be applied in the direction \eqref{Eq:polynomial-zero} $\implies$ \eqref{Eq:divisibility-binomial}: there exists $g \in [\bLowercase, \gamma \bLowercase ^ \alpha)=[\bLowercase,\mu)$ such that
\begin{equation}
	\Y(a, f, g) \;|\; \binom{2\X(a, f, g)}{\X(a, f, g)}
\;.
\label{Eq:BinomialDivisibility}
\end{equation}
Note that $\X$ and $\Y$ are defined as polynomials in $a$, $f$, and $g$; at this point, these latter three variables now have fixed values, so $\X(a,f,g)$ and $\Y(a,f)$ are also fixed integers.

The proof in this direction will be accomplished by using Theorem~\ref{step:Sun4} in the direction \eqref{proofitem:Binom} $\implies$ \eqref{proofitem:Dioph2} and then applying Lemma~\ref{lemma:MPoly}.
So let us check the conditions of Theorem~\ref{step:Sun4}.
We have integers $g\ge \bLowercase>0$ from the discussion above.

The inequalities $\X\ge 3\bLowercase$ and $\Y\ge \max \{\bLowercase, 2^8 \}$ follow from Lemma~\ref{lemma:xybounds}. The conditions $\bLowercase \in 2\uparrow$ and $\Y|\binom{2\X}{\X} $ have been verified above in \eqref{eq:bProperties} and \eqref{Eq:BinomialDivisibility}. So indeed Theorem~\ref{step:Sun4} applies in the direction \eqref{proofitem:Binom} $\implies$ \eqref{proofitem:Dioph2} and shows that there are $h,k,l,w,x,y\ge \bLowercase$ that satisfy \labelcref{eq:DFI,eq:UVK,eq:pAppWW_new,eq:g}.

The quantities $A_1$, $A_2$, $A_3$, $S$, $T$, and $R$ have been defined above in \eqref{eq:A123ST_def} and \eqref{eq:R}; at this time they are no longer polynomials but fixed integers with $S = 2A - 5 \neq 0 $. In order to apply \cref{lemma:MPoly} we need to check three conditions:
\begin{itemize}
\item $S|T$ is exactly \eqref{eq:pAppWW_new},
\item $R>0$ (see below),
\item $A_1=\bLowercase\in\square$ is \eqref{eq:bProperties}; $A_2=DFI\in\square$ is \eqref{eq:DFI}, and $A_3\in\square$ is \eqref{eq:UVK}.
\end{itemize}

To check that $R>0$, first note that $f, l, x > 0$ as well as $\mu> g > 0$. We start with \eqref{eq:g} that we established above and conclude (multiplying by $32\mu^3J^2g^2$ and using $\mu>g$)
\[
32\mu^3 g^2 (C-lJ\Y)^2 +{g^4J^2}<  2\mu^3 J^2+ \mu^3gJ^2  < 8\mu^3 g J^2
\,,
\]
which upon rearranging yields
\[
8\mu^3gJ^2 > g^2\left(32(C-lJ\Y)^2 \mu^3+g^2J^2\rule{0pt}{10pt}\right)
\,,
\]
the inequality required to establish $R>0$, based on its definition \eqref{eq:R}.

We thus conclude from \cref{lemma:MPoly} that there is an $n\ge 0$ such that
\[
0=
M_3(A_1,A_2,A_3,S,T,R,n)=
Q(a, f, g, h, k, l, w, x, y, n)
\;,
\]
where the last equality is the definition of $Q$ in \eqref{eq:PolyComposition}. This completes the proof in the forward direction.

\proofstep{Converse direction: proof that $\eqref{eq:9-variable-polynomial} \implies \eqref{eq:p-overline-zero}$}{}. Fix $a$ and let $f, g, h, k, l, w, x, y, n$ be given such that $Q(a, f, g, h, k, l, w, x, y, n) = 0$ and $n \geq 0$. All the quantities introduced in the construction of $Q$ are defined as polynomials in $a, f, g, h, k, l, w, x, y, n$; since these have fixed values, all quantities used can now be viewed as fixed integers, rather than polynomials.

By Definition~\eqref{eq:PolyComposition}, $Q(a, f, g, h, k, l, w, x, y, n) = 0$ translates into
\[
M_3(A_1, A_2, A_3, S, T, R, n) = 0
\;.
\]

\needspace{4\baselineskip}
Since $S = 2A-5 \neq 0$, we can apply \cref{lemma:MPoly} and obtain that
\begin{itemize}
    \item $S | T$,
    \item $R > 0$,
    \item $A_1, A_2$ and $A_3$ are squares.
\end{itemize}

We plan to apply Theorem~\ref{step:Sun4} in the direction \eqref{proofitem:Dioph1} $\implies$ \eqref{proofitem:Binom} and thus have to verify the various hypotheses.

We have $\bLowercase=A_1\in\square$ and in particular $\bLowercase\ge 0$.

Since $R>0$ and $R$ contains the factor $flxg$, it follows that $f\neq 0$, $l\neq 0$, $x\neq 0$ and $g \neq 0$. From $\bLowercase \ge 0$ and Definition~\ref{def:codingVariables} that $\bLowercase:=1+3(2a+1)f$, it further follows that $f>0$.

As above, spelling out the condition $R>0$ in \eqref{eq:R} and dividing by $8\mu^3g^2$ means
\[
\Big(4(C-lJ\Y)^2+\frac{g^2J^2}{8\mu^3}\Big) < \frac{J^2}{g}
\;.
\]
Since $\mu=\gamma \bLowercase ^{\alpha}>0$, we obtain the lower bound
\begin{equation}
0 \leq  \frac{g^2J^2}{8\mu ^3} \leq 4(C-lJ\Y)^2+\frac{g^2J^2}{8\mu ^3} < \frac{J^2}{g}
\;.
\label{eq:ThatInequality}
\end{equation}
This implies $J\neq 0$, $g > 0$ as well as $g^3 < (2\mu)^3$, hence $g\in[1,2\mu)$.

The purpose of the inequality $g > 0$ is to be able to use Lemma~\ref{lemma:xybounds}, which yields the desired inequalities $\X  \ge 3 \bLowercase$ as well as $\Y\ge \bLowercase$ and $\Y\ge 2^8$.

This brings us close to being able to apply Theorem~\ref{step:Sun4} as desired: all we need to check is conditions \labelcref{eq:DFI,eq:UVK,eq:pAppWW_new,eq:weakineq}.

Conditions \labelcref{eq:DFI,eq:UVK} are $A_2\in\square$ and $A_3\in\square$ as defined in~\eqref{eq:A123ST_def}. Moreover, condition \eqref{eq:pAppWW_new} is exactly $S|T$, while \eqref{eq:weakineq} follows directly from \eqref{eq:ThatInequality}, written as $4(C-lJ\Y)^2\ge 1$ (using $g\neq 0$ and $J\neq 0$).

Therefore, we can finally conclude $\bLowercase \in 2\uparrow$ and $\Y|\binom{2\X}{\X}$ from Theorem~\ref{step:Sun4}. Remember that \mbox{$g \in [1, 2\mu)$} and $\mu = \gamma \bLowercase^\alpha$, which now allows applying Theorem~\ref{step:polynomial-to-binomial} in the direction \mbox{\eqref{Eq:divisibility-binomial} $\implies$ \eqref{Eq:polynomial-zero}}. At last, we have established $\overline{P}(a,z_1,\dots,z_{\nu+1})=0$ for appropriate values \mbox{$z_1,\dots,z_{\nu+1}\ge 0$}.
\end{proof}

%% file: section6.tex
In this section we will derive an explicit universal pair in the integers based on the construction of $Q$ from the previous section. The remaining step is to remove the natural number unknown $n \in \N$ by introducing additional integer unknowns.

We have previously referenced the well-known result of Lagrange, that every natural number can be written as the sum of four squares. However, to minimize the number of additional unknowns introduced, we will use the following, more efficient, result instead.

\begin{lemma}[Gauss--Legendre theorem of three squares: \coloredlink{https://isa-afp.org/browser_info/current/AFP/Diophantine_Universal_Pairs/Three_Squares.Three_Squares.html\#Three_Squares.four_decomposition_int|fact}{Formalization}~\cite{three-squares-afp}]
\label{lemma:threeSquares}
Let $n \in \mathbb{Z}$. Then
\[
n \geq 0 \iff \exists x,y,z \in \Z, n = x^2 + y^2 + z^2 + z
\;.
\]
\end{lemma}

\begin{proof}
It is a classical fact that any integer of the form $4n+1$ can be written as the sum of three squares \cite[Lemma~1.9]{nathanson}. Write $4n+1 = a^2 + b^2 + c^2$. Reducing modulo $4$ implies that one and only one of $a, b, c$ is odd, the others even.

Without loss of generality, assume $a \equiv b \equiv 0$ and $c \equiv 1 \pmod{ 4}$. In that case we have $x, y, z$ such that $a = 2x$, $b = 2y$, $c = 2z+1$ and conclude that
\begin{align*}
    4n+1 = 4x^2 + 4y^2 + 4z^2 + 4z + 1 \implies n = x^2 + y^2 + z^2 + z
\end{align*}
The converse direction is immediate as $x^2+y^2+z^2+z$ is always non-negative.
\end{proof}

By replacing the natural number $n \in \N$ in \cref{thm:mainthm} with three new integer variables, we immediately get the universality of 11 unknowns in $\Z$ as a consequence.

\begin{corstep}[Eleven unknowns over $\Z$: \afpfact{Eleven_Unknowns_Z}{Eleven_Unknowns_Z.eleven_unknowns_over_Z}]
\label{cor:11unknowns}
Let $P \in \Z[a, z_1, \ldots, z_\nu]$. There is a polynomial $\tilde Q \in \Z[a, z_1, \ldots, z_{11}]$ such that, for all $a \in \N$, $P$ has a solution for the parameter $a$ if and only if $\tilde Q$ has a solution for the parameter $a$.
\end{corstep}

\subsection{Proof of Theorem~\ref{step:FirstUniversalPair}}

We can now proceed to prove our main result which gives a way to calculate universal pairs in $\Z$. To do so, we explicitly construct the polynomial $\tilde Q$ from \cref{cor:11unknowns} above. We restate the theorem before its proof.

\setcounter{THEOREM}{0}
\begin{THEOREM}[\afpfact{Universal_Pairs}{Universal_Pairs.universal_pairs_Z_from_N}]
Let $(\nu, \delta)_\N$ be universal. Then
\[ \big(11, \eta(\nu, \delta) \big)_\Z \]
is universal where
\[
\eta(\nu, \delta) = 
15 \, 616 + 233\,856 \; \delta + 233\,952 \; \delta \, (2 \delta + 1)^{\nu+1} + 467\,712 \; \delta^2 \, (2 \delta + 1)^{\nu+1}
\, .
\]
\end{THEOREM}

\begin{proof}
The proof succeeds by constructing an explicit polynomial with $11$ unknowns and of degree at most $\eta(\nu, \delta)$.

Let $\CalA \subset \N$ be a Diophantine set. Since $(\nu, \delta)_\N$ is a universal pair there exists a polynomial $P_\CalA$ in $\nu$ unknowns and of degree $\delta$ such that for all $a \in \N$
\[ a \in \CalA \iff \exists \mathbf{z} \in \N^\nu: P_\CalA(a, \mathbf{z}) = 0 \; . \]

We can apply the construction in \cref{def:nine-unknown-polynomial} to $P_\CalA$ to obtain a polynomial $Q \in \Z[a, f, g, h, k, l, w, x, y, n]$. By Theorem~\ref{thm:mainthm} we know that it must satisfy
\[ a \in \CalA \iff \exists z_1, \ldots, z_8 \in \mathbb{Z}, \exists z_9 \ge 0 \colon Q(a, z_1, \ldots, z_8, z_9) = 0 \; . \]
We now define
\begin{equation}\label{eq:def-q-tilde}
    \tilde{Q}(a, z_1, \ldots, z_{11}) := Q(a, z_1, \ldots, z_8, z_9^2 + z_{10}^2 + z_{11}^2 + z_{11})
\end{equation}
which has 11 unknowns as desired. Using the Gauss--Legendre three-squares theorem (Lemma~\ref{lemma:threeSquares}) we can see
\begin{align*}
a \in \mathcal{A} &
\iff \exists z_1, \ldots, z_8 \in \mathbb{Z}, \exists z_9 \ge 0 \colon Q(a, z_1, \ldots, z_8, z_9) = 0 \\
& \iff \exists z_1, \ldots, z_8, z_9, z_{10}, z_{11} \in \mathbb{Z} \colon Q(a, z_1, \ldots, z_8, z_9^2 + z_{10}^2 + z_{11}^2 + z_{11}) = 0 \\
& \iff \exists z_1, \ldots, z_{11} \in \mathbb{Z} \colon \tilde{Q}(a, z_1, \ldots, z_{11}) = 0
\;.
\end{align*}

It remains to show that the degree of $\tilde{Q}$ is bounded by $\eta(\nu, \delta)$. To do so we need to analyze the construction of $Q$ in \cref{def:nine-unknown-polynomial}; counting the degree in all steps of the definitions of polynomials throughout this paper.

Calculating the degree of the polynomial $Q$ mostly consists in straightforward calculations of the many polynomials in Definition~\ref{def:codingVariables}, Definition~\ref{def:VariableDefinition} and equation~\eqref{eq:A123ST_def}. This calculation is carried out in \cref{subsec:degree-calc-thmI,subsec:degree-calc-thmII,subsec:degree-calc-A123ST}.

Note that the first step of the construction consists in defining the polynomial $\overline{P}$ which is suitable for coding. This polynomial is given as
\[ \overline{P} = P^2 + (z_{\nu + 1} - 1)^2 \]
which means that the number of unknowns in $\overline{P}$ is $\nu' = \nu + 1$ and its degree is $\delta' := 2\delta$. From this point on in the construction, only the polynomial $\overline{P}$ is used and it is no longer necessary to refer to $P$. In particular, Definition~\ref{def:codingVariables} is used with $\overline{P}$ and not $P$. To reflect this fact, the calculation of polynomial degrees in \cref{subsec:degree-calc-thmI} uses $\nu', \delta'$ rather than $\nu, \delta$.

The remaining step of the calculation is to determine
\[ \deg \tilde Q = \deg Q(a, f, g, h, k, l, w, x, y, n) = \deg M_3(A_1, A_2, A_3, S, T, R, n) \]
where $n$ is given by $n = z_9^2 + z_{10}^2 + z_{11}^2 + z_{11}$. We examine the polynomial $M_3$ using the expression from Lemma~\ref{lemma:MPoly}:
\begin{align*}
&M_3(A_1,\dots,A_3,S,T,R,n)\\
&=\prod_{\varepsilon_1,\dots,\varepsilon_3\in\{\pm1\}^3}\left(S^2n + T^2 - S^2(2R-1)\left(T^2 + X^3 +\sum_{j=1}^3 \eps_j \sqrt{A_j}X^{j-1}\right) \right)
\end{align*}
where $X := 1 + A_1^2 + A_2^2 + A_3^2$.

The degree of $M_3$ does not change when forgetting summands that do not contribute to the monomial(s) of highest degree. In particular, the summand $X^3$ is of higher degree than all of the summands $\sqrt{A_j}X^{j-1}$ with $j \leq 3$. Hence
\begin{align*}
& \deg M_3(A_1,A_2,A_3,S, T, R, n) \\
&= \deg \left[ \prod_{\varepsilon_1,\varepsilon_2,\varepsilon_3\in\{\pm1\}^3}
\left(S^2n + T^2 - S^2(2R-1) \left(T^2+ X^3\right) \right) \right] \\
& = 8 \deg [S^2n + T^2 - S^2(2R-1)(T^2 + X^3)] \\
& = 8 \max \{ \deg [S^2n], \deg[T^2], \deg[S^2(2R-1)(T^2 + X^3)] \} \\
& = 8 \max \{ 2\deg S + \deg n, 2\deg T, 2\deg S + \deg R + \deg[T^2 + X^3] \}
\end{align*}

Now we can combine this calculation with the aforementioned results in \cref{subsec:degree-calc-thmI,subsec:degree-calc-thmII,subsec:degree-calc-A123ST}. First, note that $\deg R > 2$ and $\deg n = \deg [z_9^2 + z_{10}^2 + z_{11}^2 + z_{11}] = 2$. Therefore, the maximum cannot be $2\deg S + \deg n$. Moreover, note that the calculation in the appendix yields $\deg R > \deg T$ and therefore the maximum is attained by

\begin{align*}
&\hphantom{{=}2} 2\deg S + \deg R + \deg[T^2 + X^3] \\
&= 2\deg S + \deg R + \max \{ 2 \deg T, 3 \deg [1 + A_1^2 + A_2^2 + A_3^2] \} \\
&= 2\deg S + \deg R + \max \{ 2 \deg T, 6 \deg A_1, 6 \deg A_2, 6 \deg A_3 \} \\
&= 2\deg S + \deg R + 6\deg A_2 \\
&= 6 + 20 + 6(\delta'\smash{(\delta'+1)^{\nu'}} + 1) + 6 \cdot 320 \\
&\hphantom{{=}\;} + 2 \deg \X + 4 \deg \X + 6 \cdot 86 \deg \X + 4 \deg \Y + 8 \deg \Y + 6 \cdot 172 \deg \Y
\end{align*}
again making use of the values calculated in Appendix~\ref{subsec:degree-calc-A123ST}. The formula for $\eta$ is now obtained by substituting the values for $\deg \X$ and $\deg \Y$ and then using $\delta' = 2\delta$ and $\nu' = \nu + 1$.
\end{proof}

\subsection{Formalizing the Calculation of Polynomial Degrees}
The calculation of a bound on the degree of polynomials in the above proof as well as in \Cref{sec:appendix} is elementary mathematics, nevertheless it was tricky to formalize. The reason for this is the nested nature of the polynomials in the constructions of \cref{eq:def-q-tilde} and \cref{def:nine-unknown-polynomial}. This section discusses how to automate and formalize such degree calculations for general and iterated substitutions of polynomials into polynomials. 

First of all, note that when calculating the degree of nested polynomials one almost never gets to write equalities and can only obtain bounds. Consider the example $Q(x, y) = x^2 + y$ and the substitution $y \mapsto P \in \Z[x]$. The degree of $Q(x, P(x))$ is bounded above by $\max(\deg P, 2)$, but it is possible for the degree to be lower than the bound. Worse yet, if $P(x)=-x^2$, one obtains the zero polynomial whose degree is not defined. There is little mathematical interest in bounding the degree of our construction from below, so we limit ourselves to formalizing upper bounds instead of equalities. 

A naive approach to calculate this degree is to unfold the definition of all intermediate polynomials until one obtains a polynomial expression that contains only elementary operations such as addition and multiplication. The degree can then be bounded using Isabelle's built-in automation (\lstinline|simp|, \lstinline|auto|, \lstinline|metis|) configured to use identities such as $\deg(PQ) \leq \deg P + \deg Q$. Unfortunately, this is not practical because the resulting polynomial becomes too large to handle effectively. Moreover, such brute force is not appealing from a mathematical perspective: to improve the universal pair by meddling with our constructions, it is handy to have an idea of the degrees of the various intermediate polynomials. This is obscured by an unfold-and-brute-force approach.

\medskip

Instead, we would like to calculate the degree of every polynomial separately and then combine them to achieve the overall bound. However, this is not possible unless one accepts bounds that are far from tight: Given $P \in \Z[x_1, x_2]$ with $\deg P = 7$ and substituting $x_1 \mapsto Q(x)=x^7$ and $x_2 \mapsto R(x)=x^2$, we can only bound the degree of the resulting polynomial by $49$ in general. But if $P(x_1, x_2)=x_1^2 + x_2^7$ then the degree after the same substitution is just $14$. 

Therefore, we calculate a more general expression for the degree of a polynomial that is also meaningful for substitutions. In Isabelle, we call this expression \lstinline|total_degree_env| (\afpfact{Total_Degree_Env}{Total_Degree_Env.total_degree_env_def}), the total degree relative to a substitution ``environment.'' Given a function $f : \N \to \N$ that associates to each variable $x_i$ some degree, \lstinline|total_degree_env $f$ $P$| is a bound for the total degree of $P \in \Z[x_1, \ldots, x_\nu]$ after substituting each $x_i$ with some polynomial whose degree is bounded by $f(i)$. Concretely, in the above example one could calculate 
\[ \texttt{total\_degree\_env $f$ $P$} = \max (2 f(1), 7 f(2)) \; . \]

\subsubsection{Metaprogramming Infrastructure.} \label{metaprogramming} Describing total degrees of polynomials in this way allows us to formalize the bounds in a modular fashion. However, doing this for all polynomials constructed in this paper would be laborious, which is why we built two custom metaprogramming commands that can automate this task.

The first metaprogramming command, called \lstinline|poly_extract|, was contributed by Anna Danilkin (\afpthy{Poly_Extract}). It translates a function of type \lstinline|int $\Rightarrow$ ... $\Rightarrow$ int| into a polynomial, i.e. a term of type \lstinline|int mpoly|, and proves that evaluating the generated polynomial gives the original function. This allows us to use ordinary Isabelle functions wherever suitable, but conveniently switch to a polynomial when we need to calculate explicit degrees or count the number of variables. 

Note that \lstinline|poly_extract| is built in such a way as to respect the nested structure of our definitions. If one polynomial is defined in terms of another, then \lstinline|poly_extract| will output an expression that corresponds to substituting polynomials. We represent this using the function
\begin{lstlisting}
poly_subst_list :: 'a mpoly list $\Rightarrow$ 'a mpoly $\Rightarrow$ 'a mpoly
\end{lstlisting}
It works such that \lstinline|poly_subst_list xs P| is the polynomial obtained by replacing the $i$'th variable \lstinline|Var i| of \lstinline|P| with the $i$'th element of \lstinline|xs|.

Concretely, if the command processes the definition $P(x_0, x_1) = Q(x_0^3, 2x_1)^5 + x_1$ where $Q \in \Z[x_0, x_1]$ is assumed to have been processed by \lstinline|poly_extract| previously, it does not unfold the definition of $Q$; instead, it constructs the following substitution of polynomials: 
\begin{lstlisting}
P_poly $\equiv$ (poly_subst_list [(Var 0)^3, 2*Var 1] Q_poly)^5 + Var 1
\end{lstlisting}

On an implementation level, the command proceeds by pattern matching on the various constructors of polynomials, turning integer literals into the corresponding polynomial constants and transforming addition of integers into addition of polynomials etc. Exponents can be arbitrary terms of type \lstinline|nat|. Integer arguments become variables \lstinline|Var 0 ... Var n| and non-integer arguments are turned into parameters of the result. For example, the polynomial $c$ in \cref{def:codingVariables} depends on a given polynomial $P$, therefore \lstinline|c_poly| also becomes a function in $P$. 

Finally, the equivalence lemma between the original function and the generated polynomial is discharged by unfolding a set of definitions and lemmas (for every pattern that we match on) and using Isabelle's \lstinline|auto| method.

\medskip

While the \lstinline|poly_extract| command provides us with explicit polynomials, it does not help in estimating their degrees. A second command \lstinline|poly_degree|, contributed by Annie Yao, helps with the tedious work of bounding total degrees (\afpthy{Poly_Degree}). 

The command can be invoked on any identifier referring to a polynomial. On a high level, the metaprogram first calculates a new term, the degree bound, by recursively pattern matching on the definition of the provided polynomial. It then adds the term to the environment and assembles a new lemma stating that the total degree is indeed bounded by the given term. Finally, it formally proves this lemma using various simplification rules.

The crucial complexity is that this command has to appropriately deal with polynomials substituted into each other. The central lemma which connects substitution and \lstinline|total_degree_env| is the following (\afpfact{Total_Degree_Env}{Total_Degree_Env.total_degree_env_poly_subst_list}\footnote{The listing here shows a slightly modified version for the sake of better exposition.}).

\begin{lstlisting}
lemma total_degree_env_poly_subst_list:
  fixes P :: "'a::comm_ring_1 mpoly" 
    and f :: "nat $\Rightarrow$ nat" and qs :: "'a mpoly list"
  shows "total_degree_env f (poly_subst_list qs P)
   $\leq$ total_degree_env ($\lambda$m :: nat. total_degree_env f (qs !$_0$ m)) P"
\end{lstlisting}

There is a bit to unravel here: \lstinline|total_degree_env| tells us what the bound on the total degree of a polynomial would be after substituting polynomials with degree bounds \lstinline|f| into it. This lemma states how we can bound \lstinline|total_degree_env| for a polynomial that is defined in terms of substitution (specifically, the polynomial \lstinline|poly_subst_list qs P :: 'a mpoly|). It allows us to move the effect of the substitution \lstinline|poly_subst_list|, to the outside, to our degree environment bookkeeping. The upper bound which we obtain is a term of the shape \lstinline|total_degree_env f' P|, with \lstinline|f'| containing the appropriate degree bounds of the polynomials \lstinline|qs| being substituted.

\subsection*{Limitations.} We note that the \lstinline|poly_degree| command is not powerful enough to calculate all bounds we need in the formalization. This is a pragmatic choice---for some polynomials it was easier to manually prove a bound on their total degree than to build metaprogramming infrastructure for their special case. For example, the $\coeffs$ polynomial (\cref{Def:ConsolidatingPoly}) is defined as a complicated finite sum; instead of implementing metaprogramming infrastructure that deals with multi-indices in summation, it was much quicker to manually formalize the respective bound.

Whenever one of our commands needs to add a new definition or theorem to the environment, it creates a new identifier by appending to the name of the given polynomial (e.g. \lstinline|P_poly| when running \lstinline|poly_extract| on a polynomial called \lstinline|P|). This ensures consistent naming inside our own project but might not be desirable if the command were reused in other contexts. 

In general, our metaprograms have been developed to satisfy the specific requirements of formalizing universal pairs. We caveat that they may not be easily adaptable to other projects. We have written a complementary text ``A First Introduction to Isabelle/ML Metaprogramming''~\cite{isabelle-metaprogramming-intro} based on the example of \lstinline|poly_degree|. This tutorial-style article explains a simplified version of the command that we believe to be more readily extensible, since it strips away much of the complexity that is specific to our use case. In addition, we document and explain the underlying mechanisms and metaprogramming APIs therein.

\subsection{Mathematical Outlook}
In this paper, we have established the first universal pair $(\nu, \delta)_\Z$ over the integers with $\nu = 11$. A natural direction for future research is to explore the solvability of Hilbert's Tenth Problem for other values of $\nu$ and $\delta$. In particular, it would be interesting to investigate ways to reduce the extremely high degree of the current construction (approximately $10^{63}$) by introducing additional integer unknowns, as well as calculate and then reduce the degree in Sun's recent $\nu=10$ result~\cite{sun10}.

Another interesting direction is the extension of universal pairs to multi-parameter Diophantine sets. In this regard, we let $\delta$ and $\nu$ become two maps of natural numbers $\delta(n)$, $\nu(n)$. Universality now requires that, for every $n \in \N$, every Diophantine set $A \subseteq \N^n$ can be represented by an equation
\[
D(a_1, \ldots, a_n, y_1, \ldots, y_{\nu(n)}) = 0
\]
of total degree at most $\delta(n)$. Existing techniques are sufficient to determine pairs such that either $\nu(n)$ or $\delta(n)$ is a constant function. However, it remains an open question whether one can construct a multi-parameter universal pair in which both $\nu(n)$ and $\delta(n)$ are independent of $n$. In other words, can every Diophantine equation with $n$ parameters be reduced to another Diophantine equation---over the same parameters---that involves a fixed number of additional unknowns and has bounded total degree, regardless of $n$?

%% file: section7.tex
This retrospective discusses our interactions with the Isabelle prover and weighs the costs and benefits of writing new mathematics in interaction with a computer. See also our related paper~\cite{itp-paper}, which documents the formalization and gives a high-level overview of the mathematical argument. In total, our formalization~\cite{up-afp} comprises 20,000 lines of code.\footnote{The core argument in this article spans 37 pages (pp.~8--44), at about 45 lines per page, which yields a \emph{de Bruijn factor} of approximately 12.}

Our interactions with Isabelle have not seriously influenced the mathematical argument, but they have had a large impact on the mathematical exposition within this article. However, the correctness of the degree bound in our main theorem is an important exception to this general observation, see \Cref{formalization:degree-calculation}.

To understand the proof strategies in the literature (and even some of the statements) well enough to formalize them, substantial effort was required to reconstruct every argument in sufficient detail. This work was led by Bayer and Hassler with guidance by Matiyasevich and Schleicher. Afterwards, David and Bayer led the formalization, together with many contributors\footnote{Contributors: Timothé Ringeard, Xavier Pigé, Anna Danilkin, Annie Yao, Mathis Bouverot-Dupuis, Paul Wang, Quentin Vermande, Théo André, Loïc Chevalier, Charlotte Dorneich, Eva Brenner, Zhengkun Ye, Kevin Lee.}~\cite{up-afp} whose contributions are explained in the \nameref{sec:acknowledgements}.

\subsection{What Formalization Can Do for Your Mathematical Draft}
Throughout this article, we have documented when the formalization influenced the mathematics or the exposition through \emph{formalization remarks}. The most interesting remarks can be grouped into the following categories: 
\begin{itemize}
    \item Observing where details are swept under the rug in the informal proof (\Cref{formalization:K-gt-0,formalization:lost-proof})
    \item Removing unnecessary detail and streamlining proofs (\Cref{formalization:exists-n,formalization:b-squared})
    \item Fixing bugs (\Cref{formalization:A-gt-0,formalization:A-gt-2,formalization:afg-assumption})
    \item Precisely determining which lemmas from the literature our proof depends on, and decoupling the correctness of our own result from the correctness of the dependencies (by formalizing all of our dependencies herein,\footnote{One of these formal proofs has been independently published in the Isabelle AFP by our workgroup members Anna Danilkin and Loïc Chevalier~\cite{three-squares-afp}.} stemming from a variety of sources~\cite{JonesMersenne, MR75, SunDFI, nathanson})
\end{itemize}
Two of these benefits deserve further explanation.

Proof assistants make the hierarchical layers of a proof fully explicit, as opposed to conventional mathematical writing where facts are presented linearly (cf. Lamport~\cite{lamport-proofs}). In this article, \Cref{lemma:psiInjective} is an example explicitly written with such a hierarchical proof. In \Cref{formalization:lost-proof}, we note that this informal proof was only written after the formalization had been completed, which is why its exposition is modeled directly after the formal proof.

The following sounds obvious, but is worth stating: using an interactive theorem prover has made our proof \emph{interactive}, and thus suitable for experimentation and discovery jointly with the proof assistant.\footnote{In our case, the humans experimented directly with the interactive proof, but in the future they are likely to also ask AI agents to do so.} For example, the original version of our \Cref{step:polynomial-to-binomial} had an extra assumption originating from the corresponding theorem in Sun's article~\cite{Sun}. At a later point, we noticed that the assumption might not be needed in our proof and using Isabelle it was easy to confirm this to be the case (\Cref{formalization:b-squared}). After removing the assumption from the code, the proof assistant flagged where it had been used. For all occurrences, Isabelle's automation tools quickly allowed us to reprove each respective goal from the remaining context.

\subsection{Comparison of Formal and Informal Proof}
There is no substantial structural difference between the argument presented herein and the formal proof in Isabelle. We believe that the main reason for this is that both writing Isabelle code and writing this article in \LaTeX{} affected each other bidirectionally throughout the project. This is different from the (\emph{in-vitro}) formalization of a published result, where the authors of the formal proof might choose to draw on a variety of mathematical sources. We observe that in the Lean community, the friction that stems from using multiple sources with different conventions is often resolved by first writing formalization \emph{blueprints}~\cite{massot_leanblueprint} that compile the existing mathematical literature into a detailed and coherent informal text, and then formalizing the blueprint. In our project, this process happened in parallel and in both directions (see e.g. \Cref{formalization:lost-proof}), instead of in succession.

We aim to make correspondence between formal and informal argument transparent throughout the article by linking each lemma and each theorem to their formal versions. We do not believe that the choice of the specific proof assistant had a significant impact on this project; a comparable effort would have been possible in other major ITPs with an existing library for multivariate polynomials, elementary number theory, and a meta-programming interface (cf.~\Cref{metaprogramming}).

\subsection{Novel Working Modes for Mathematical Collaboration} 
Formalization also introduces new possibilities for collaboration among coauthors, especially with junior researchers (or students) who may be more likely to make mistakes. In this project, many students contributed formalized mathematics through in-person meetings of a workgroup at ENS Paris that met regularly for two years. 

Although the project naturally divides into several independent parts, managing and verifying the input of these thirteen students in a traditional \mbox{\LaTeX{}-based} project would have been challenging. Using a proof assistant, most of the burden of checking the correctness of arguments was shifted to Isabelle. For intermediate results that build on each other, Isabelle also makes it clear when contributions have been correctly pieced together. After reviewing that the formal statements accurately reflect the informal text, this freed up time to focus on the mathematics itself.

In fact, we recruited a new team of students every academic year due to previous contributors graduating or going on exchange. To pass on progress from one generation of students to the next, the formal language provided a much more robust way to track progress than the \LaTeX{} document. For example, in \Cref{formalization:lost-proof} we discuss \cref{lemma:PellDiophantine2} which initially only survived in Isabelle, before it was reintegrated into this article.

\subsubsection*{Relation to other large mathematical collaborations.}
Computer-assisted, online, and large-scale mathematical collaborations structured akin to Gowers' Polymath effort\footnote{\url{https://polymathprojects.org/about/}} have recently become increasingly popular. Notable examples that involve formalization include the liquid tensor~\cite{scholze-lte} and equational theories projects~\cite{eq-theories}, as well as the busy beaver challenge $\mathrm{BB}(5)$~\cite{BB5,mxdys}. 

While the above examples have been primarily coordinated online, we collaborated primarily in-person. This allowed us to iterate often between formal and informal proofs, and collaborate closely with one another. Running the same project online would have enabled more people to contribute, but would have required a more structured project, e.g. through a ready-to-formalize blueprint (which our manuscript often was not), and a review process for external contributions (which we instead handled via in-person discussions, looking at the same screen).

Many contributors were advanced undergraduate students at the time of developing both the mathematical results and the formalization, in part trained for this specific project. This compares to the equational theories project and the $\mathrm{BB}(5)$ challenge, which have also enabled many non-experts to contribute. In our own personal experience, and from what students have told us, we note: if students are exposed to formalization early on, the formal language of the prover quickly becomes a natural environment in which to work. After some time, translating to and from traditional mathematical writing happens quickly and automatically.

\subsection{Conclusion}
Our project illustrates that (manually) formalizing highly technical mathematics can provide real benefit to mathematical texts, correcting many mistakes and enhancing clarity and reliability. Isabelle helped us understand our proofs at a substantially more granular level than traditional mathematical writing would have. The additional time and labor required cannot be neglected, but in our opinion, these are outweighed by the benefits discussed herein.

In retrospect, we identify two vital factors that contributed to the success of this project. First, we crucially relied on our previous experience from formalizing the well-established DPRM theorem to make important decisions for this project. Second, it has been immensely helpful to have recruited an interdisciplinary team of people with different skill sets. In our case, this included both pure mathematicians and computer scientists, and it was their joint effort which succeeded in bridging all the way from natural language proofs to the meta-language in which Isabelle itself is programmed.

Encouraging mathematicians to integrate proof assistants into their workflow now will help shape the development of better tools, eventually making formalization in systems like Isabelle or Lean as seamless as writing in \LaTeX{}. This is further reinforced by recent advances in AI, which provide powerful automation tools for formalization. By formalizing mathematics now, we are ensuring that our fields benefit from AI and automation as these tools evolve.

%% file: degree-calculation.tex
In this section, the degree of the polynomials $\bLowercase, \X, \Y$ from Definition~\ref{def:codingVariables}, the degree of the variables from Definition \ref{def:VariableDefinition} as well as the degree of the polynomials in Equation~\eqref{eq:A123ST_def} is calculated.

\subsection{Degree of Polynomials \texorpdfstring{$\bLowercase, \X, \Y$}{b, X, Y} from Definition~\ref{def:codingVariables}}
\label{subsec:degree-calc-thmI}
Given a fixed polynomial $\overline{P}$ with $\nu'$ variables and degree $\delta'$, we investigate the degree of polynomials $\bLowercase, \X, \Y$ in the variables $a, f, g$. We calculate
\[ \begin{split}
\deg \bLowercase &= 2 \\
\deg \B &= 2 \delta' \\
\deg M &= \deg[\mask(\bLowercase, \B, \mathbf{n})] \leq \deg [(\B - \bLowercase) \B^{(\delta'+1)^{\nu'}}] = (1 + (\delta' + 1)^{\nu'}) 2 \delta' \\
\deg N_0 &= (1 + (\delta'+1)^{\nu'}) 2 \delta' \\
\deg N_1 &= ((2\delta' + 1) (\delta'+1)^{\nu'} + 1) 2 \delta' \\
\deg N &= (2 + (2 \delta' + 2)(\delta' + 1)^{\nu'}) 2 \delta' = (1 + (\delta' + 1)^{\nu' + 1}) 4\delta' \\
\deg c &= \max\{ \deg[a\B], \deg g \} = 1 + \deg \B = 1 + 2 \delta' \\
\deg[\coeffs(\B)] &= \deg \left[\sum_{\substack{\mathbf{i} \in \mathbb{N}^{\nu'+1} \\ \|\mathbf{i}\| \leq \delta'}} \mathbf{i} ! (\delta' - \|\mathbf{i}\|)! a_{\mathbf{i}} \B^{(\delta'+1)^{\nu'+1} - \sum_{s=0}^{\nu'} i_s (\delta'+1)^s } \right] 
= \deg[\B^{(\delta'+1)^{\nu'+1}}]  \\
&= (\delta'+1)^{\nu'+1} 2 \delta'\\
\deg \K &= \deg[ \values(c, \B)] 
= \deg[c^{\delta'} \coeffs(\B) + \textstyle\sum_{j=0}^{(2\delta'+1)(\delta'+1)^{\nu'}} \frac{\B^{j+1}}{2} ]\\
&= \max \{ \delta' \deg c + \deg [\coeffs(\B)], 
	\deg [\B^{(2\delta'+1)(\delta'+1)^{\nu'} + 1}] \} \\
&= \max \{ \delta' (1 + 2 \delta') + (\delta'+1)^{\nu'+1} 2 \delta', 
	(1 + (2\delta'+1)(\delta'+1)^{\nu'})2\delta' \}\\
&\overset{1}{=} (1 + (2\delta'+1)(\delta'+1)^{\nu'})2\delta'\\
\deg \S &= \deg \K + \deg N_0 
= (2 + (2 \delta' +2)(\delta'+1)^{\nu'}) 2 \delta' 
= (1 + (\delta'+1)^{\nu'+1}) 4 \delta'\\
\deg \T &= \max \{ \deg M, \deg N_0 + (1 + (\delta'+1)^{\nu'+1} )2 \delta'  \} \\
&= (1 + (\delta'+1)^{\nu'}) 2 \delta'+ (1 + (\delta'+1)^{\nu'+1} ) 2 \delta' 
= (2 + (\delta' + 2)(\delta'+1)^{\nu'} ) 2 \delta'  \\
\deg \RR &= \deg N + \max\{\deg \T, \deg \S\} \\ 
&= (1 + (\delta' + 1)^{\nu' + 1}) 4\delta' + (1 + (\delta'+1)^{\nu'+1}) 4 \delta'
= (1 + (\delta'+1)^{\nu'+1}) 8 \delta' \\
\deg \X &= \deg N + \deg \RR \\
&= (1 + (\delta'+1)^{\nu'+1}) 4 \delta' + (1 + (\delta'+1)^{\nu'+1}) 8 \delta' = (1 + (\delta'+1)^{\nu'+1}) 12 \delta' \\
\deg \Y &= 2 \deg N = (1 + (\delta' + 1)^{\nu' + 1}) 8 \delta' \, .\\
\end{split} \]
Note that Equality~1 follows because
\[ \begin{split} \delta' (1 + 2 \delta') + (\delta'+1)^{\nu'+1} 2 \delta' 
&= \delta' + (\delta' + (\delta'+1)^{\nu' + 1})2\delta' \\
&\leq (1 + \delta'(\delta'+1)^{\nu'} + (\delta'+1)^{\nu'+1})2\delta' \\ 	
&= (1 + (2\delta'+1)(\delta'+1)^{\nu'})2\delta' \, .
\end{split} \] 

\subsection{Degree of Polynomials in Definition~\ref{def:VariableDefinition} (Equations~\eqref{eq:Def_U}--\eqref{eq:Def_J})}
\label{subsec:degree-calc-thmII}	
In the construction of $Q$ the polynomials $\X$ and $\Y$ are substituted for $X$ and $Y$. Therefore, we calculate the degree of the following polynomials in terms of $\deg \X$ and $\deg \Y$.  
\[ \begin{split}
\deg U &= 1 + \deg \X + \deg \Y \\
\deg V &= 2 + \deg \Y \\
\deg A &= 3 + \deg \X + 2 \deg \Y \\
\deg B &= \deg \X \\
\deg C &= \max \{ \deg B, 1 + \deg A \} 
= 4 + \deg \X + 2 \deg \Y \\
\deg D &= 2\deg A + 2 \deg C 
= 14 + 4 \deg \X + 8 \deg \Y \\
\deg E &= 1 + 2 \deg C + \deg D 
= 23 + 6\deg \X + 12 \deg \Y \\
\deg F &= 2 \deg A + 2 \deg E 
= 52 + 14 \deg \X + 28 \deg \Y \\
\deg G &= \max \{\deg CDF, 3 \deg A + 2 \deg E \} \\
&= \max \{70 + 19 \deg \X + 38 \deg \Y, 55 + 15 \deg \X + 30 \deg \Y \} \\ 
&= 70 + 19 \deg \X + 38 \deg \Y \\
\deg H &= \max \{\deg BF, 1 + \deg CF \} 
= 57 + 15 \deg \X + 30 \deg \Y \\
\deg I &= 2 \deg G + 2 \deg H 
= 254 + 68 \deg \X + 136 \deg \Y \\
\deg J &= 1 + 2 \deg U + \deg V = 5 + 2 \deg \X + 3 \deg \Y 
\end{split} \]

\subsection{Degree of Polynomials in Definition~\ref{def:nine-unknown-polynomial}}\label{subsec:degree-calc-A123ST}
Using the values calculated previously, one can continue to determine the degree of $A_1, A_2, A_3, S, T$ and $R$:

\[ \begin{split}
\deg A_1 &= \deg \bLowercase = 2 \\
\deg A_2 &= \deg [DFI] = 320 + 86 \deg \X + 172 \deg \Y \\
\deg A_3 &= \deg [(U^4V^2-4)J^2 + 4] = 4 \deg U + 2 \deg V + 2 \deg J 
= 18 + 8 \deg \X + 12 \deg \Y \\
\deg S &= \deg [2A - 5] = 3 + \deg \X + 2 \deg \Y \\
\deg T &= \max \{ \deg [\bLowercase w C], \deg[\bLowercase^2 w^2] \} 
= \max \{ 2 + 1 + (4 + \deg \X + 2\deg \Y), 4 + 2 \} \\
&= 7 + \deg \X + 2 \deg \Y \\
\deg \mu &= \deg [\gamma \bLowercase^{\alpha}] = \alpha \deg \bLowercase = 2 (\delta'(\delta'+1)^{\nu'} + 1) \\
\deg R &= 6 + \deg \left[ \mu^3gJ^2-g^2\left((C-lJ\Y)^2\mu^3+g^2J^2\right)\right] 
\overset{2}{=} 6 + \deg \left[ \mu^3gJ^2-g^2(lJ\Y)^2\mu^3\right] \\
&= 6 + \deg \left[g^2l^2J^2\Y^2\mu^3\right]
= 10 + 2 \deg J + 2 \deg \Y + 3 \deg \mu\\
&= 20 + 6(\delta'(\delta'+1)^{\nu'} + 1) + 4 \deg \X + 8 \deg \Y \, . \\
\end{split} \]

Note that Equality~2 is justified since $\deg [lJ\Y]$ is larger than $\deg C$ and, moreover, $\deg [(lJ\Y)^2]$ is larger than $\deg [g^2J^2]$.